\documentclass[10pt]{amsart}
\usepackage[margin=1.3in,marginparwidth=1.1in,marginparsep=0.1in]{geometry}
\usepackage{import}

\usepackage{graphicx}
\makeatletter
\DeclareRobustCommand{\cev}[1]{%
  {\mathpalette\do@cev{#1}}%
}
\newcommand{\do@cev}[2]{%
  \vbox{\offinterlineskip
    \sbox\z@{$\m@th#1 x$}%
    \ialign{##\cr
      \hidewidth\reflectbox{$\m@th#1\vec{}\mkern4mu$}\hidewidth\cr
      \noalign{\kern-\ht\z@}
      $\m@th#1#2$\cr
    }%
  }%
} 
\makeatother
\usepackage{float}
\usepackage{wasysym}
\usepackage{marginnote,color}

\usepackage{bm}
\usepackage{tikz-cd}
\usepackage[toc,title,page]{appendix}
\usepackage{yhmath}
\usepackage{tikz}\usepackage{float}
\usepackage{fancyhdr}
\usepackage{blkarray, bigstrut}
\usepackage[hidelinks]{hyperref}
\usepackage{setspace}
\usepackage{amstext} % for \text macro
\usepackage{array}
\usetikzlibrary{calc}
\allowdisplaybreaks
\theoremstyle{plain}
\newtheorem{thm}{Theorem}[section]
\newtheorem{lem}[thm]{Lemma}

\newtheorem{prop}[thm]{Proposition}

\newtheorem{cor}[thm]{Corollary}
\theoremstyle{definition}
\newtheorem{defn}[thm]{Definition}
\newtheorem{defn2}[thm]{Definition/Lemma}

\theoremstyle{remark}
\newtheorem{rem}[thm]{Remark}
\theoremstyle{definition}
\newtheorem{ex}[thm]{Example}

\newtheorem{nt}[thm]{Notation}
\newtheorem{claim}[thm]{Claim}
\newtheorem{conv}[thm]{Convention}

\newcommand{\ZZ}{\mathbb{Z}}
\renewcommand{\SS}{\mathbb{S}}
\newcommand{\lag}{\langle}
\newcommand{\rag}{\rangle}

\newcommand{\hh}{,\hdots,}

\renewcommand{\AA}{\mathbb{A}}

\newcommand{\aaa}{\alpha}
\newcommand{\bs}{\backslash}

\newcommand{\ep}{\epsilon}
\newcommand{\fa}{\forall}
\newcommand{\fb}{\mathfrak{B}}

\newcommand{\ti}{\widetilde}

\newcommand{\G}{\Gamma}

\newcommand{\wW}{\widetilde{W}}
\newcommand*{\shom}{\mathscr{H}\kern -.5pt om}

\newcommand{\cF}{\mathcal{F}}
\newcommand{\cA}{\mathcal{A}}
\newcommand{\cB}{\mathcal{B}}
\newcommand{\cI}{\mathcal{I}}
\newcommand{\cJ}{\mathcal{J}}

\newcommand{\cP}{\mathcal{P}}

\newcommand{\sg}{\sigma}

\newcommand{\ot}{\otimes}
\newcommand{\op}{\oplus}

\newcommand{\cd}{\cdot}

\let \La=\Lambda

\let \ot=\otimes
\let \op=\oplus
\let \ov=\overline

\let \wt=\widetilde
\let \:=\colon
\numberwithin{equation}{section}
\usepackage{amsfonts}

\usepackage{pdflscape}
\usepackage{tikz-cd}
\usepackage{amsmath,amssymb,latexsym,amsthm}
\usepackage{mathtools}
\usepackage{graphicx}
\usepackage{mathrsfs}
\usepackage{amscd}
\usepackage{color}

\makeatletter
\renewcommand*\env@matrix[1][*\c@MaxMatrixCols c]{%
  \hskip -\arraycolsep
  \let\@ifnextchar\new@ifnextchar
  \array{#1}}
\makeatother

\DeclareMathOperator\Conv{Conv}
\DeclareMathOperator\rk{rank}
\DeclareMathOperator\Gr{Gr}
\DeclareMathOperator\GL{GL}
\DeclareMathOperator\SL{SL}

\DeclareMathOperator\im{Im}
\DeclareMathOperator\id{id}
\DeclareMathOperator\Id{Id}

\DeclareMathOperator\End{End}

\DeclareMathOperator\aut{Aut}

\DeclareMathOperator\spec{Spec}

\DeclareMathOperator\h{Hom}

\input xy
\xyoption{all}
\DeclareRobustCommand{\SkipTocEntry}[5]{}
\begin{document}
\title{Degenerations of Grassmannians via lattice configurations II}
\author{Xiang He, Naizhen Zhang}
\begin{abstract}
This paper is a continuation of our study of degenerations of Grassmannians in \cite{he2023degenerations}, called \textit{linked Grassmannians}, constructed using convex lattice configurations in Bruhat-Tits buildings. We describe the geometry and topology of linked Grassmannians associated to a larger class of lattice configurations, generalizing the results in \textit{loc.cit.}. %, and study the relation between linked Grassmannnians for different configurations. 
In doing so, we utilize results on the geometry of quiver Grassmannians and affine Schubert varieties,  and make comparison to the study of certain local models of Shimura varieties.   
\end{abstract}
\address[Xiang He]{Yau Mathematical Sciences Center, Tsinghua University, Haidian District, Beijing, China, 100084}
\email{xianghe@mail.tsinghua.edu.cn}
\thanks{Xiang He is supported by the NSFC grant 12301057 and the National Key R$\&$D Program of China 2022YFA1007100, Naizhen Zhang was supported by the DFG Priority Programme 2026 ``Geometry at infinity".}
\address[Naizhen Zhang]{Fairleigh Dickinson University Vancouver Campus, 842 Cambie St, Vancouver, BC V6B 2P6, Canada}
\email{n.zhang1@fdu.edu}
\maketitle
\tableofcontents

\section{Introduction}
This paper is the second one in a sequel studying degenerations, $LG_r(\Gamma)$, of the Grassmannian, called \textit{linked Grassmannians}, that are specified by convex lattice configurations $\Gamma$ in the Bruhat-Tits building of the special linear group, where $r$ is the dimension of the sub-vector spaces. In \cite{he2023degenerations}, we studied a special class of such degenerations, associated to what we call \textit{locally linearly independent} lattice configurations, for which the topology and geometry of the special fibers $LG_r(\Gamma)_0$ can be largely understood using techniques of quiver representations. Guided by our original motivation coming from the study of moduli of curves, we established new smoothing theorems for degenerations of linear series on smooth projective curves. %However, in our quest for stronger results, we quickly came to realize that one technical condition we had for locally linearly independent configurations did not generalize to more general lattice configurations namely, the quiver representations always decompose into a direct sum of irreducible representations whose dimension vector is at most $(1,1,\hdots,1)$\com{this point is mentioned later, maybe remove?}.
Meanwhile, as mentioned in \textit{loc.cit.}, there is a natural connection between linked Grassmannians and certain local models of Shimura varieties. More precisely, when the configuration $\Gamma$ is a face of the standard alcove of the Bruhat-Tits building (Notation~\ref{nt:standard alcove}, or equivalently, a simplex in the building) and making the same assumptions on the underlying field, they coincide. In the study of the latter, there is a long-standing technique that involves embedding the special fiber of the local model into some affine flag variety (see, for example, \cite{gortz2001flatness},\cite{gortz2003flatness},\cite{pappas2003local},\cite{PR2}). With this in mind, we naturally ask the following question: can we apply both the technique of quiver representation and the structure of affine flag varieties to generalize the results of our previous paper?

\medskip
The current paper is a summary of some first-stage structural results we managed to establish using a combination of the aforementioned techniques.  In particular, while the family of locally linearly independent lattice configurations is almost disjoint with the family of faces of standard alcoves (see Example~\ref{ex:quiver of a simplex} and~\ref{ex:locally linearly indep}, in fact,  the intersection of the two families is the set of convex configurations that contains at most two lattice classes), we introduce the notion of \textit{locally weakly independent} lattice configurations, generalizing both cases at the same time. Therefore, the corresponding linked Grassmannians include the local models mentioned above.

\medskip 
%The connection between the geometry of affine flag varieties for a reductive group over a discrete valued field and quiver representation is natural and has been well-studied in previous literature. Lusztig first established that orbits of representations of cyclic quivers are type-$A$ affine Schubert cells (\cite[Theorem 11.3]{lusztig1990canonical}). From the quiver representation perspective, such loci can be characterized either using rank vectors or underlying representation types. As affine Schubert cells, they are orbits of parahoric subgroups of the ambient reductive group. The former perspective provides us with convenience in analyzing the closure-containment relations among these loci  using linear-algebro tools. %For example, we can understand whether some orbit is contained in the closure of another by checking whether the rank vector the latter dominates the rank vector of the former. On the other hand, built on the powerful Bruhat-Tits theory, one can provide the answer to this problem in terms of the \textit{Bruhat order}.

The connection between the geometry of affine flag varieties for a reductive group over a discretely valued field and quiver representations is natural and has been well-studied in previous literature. Lusztig first established that orbits of representations of cyclic quivers are type-$A$ affine Schubert cells (\cite[Theorem 11.3]{lusztig1990canonical}). From the quiver representation perspective, such loci can be characterized either using rank vectors (Notation~\ref{nt:rank vector}) or underlying representation types. As affine Schubert cells, they are orbits of parahoric subgroups of the ambient reductive group. The former perspective provides us with convenience in analyzing the relations among these loci  using linear-algebro tools. For example, we can understand whether some orbit is contained in the closure of another by understanding whether the representation type of the latter specializes to the representation type of the former. On the other hand, built on the powerful Bruhat-Tits theory, one can provide the answer to this closure-containment problem in terms of the \textit{Bruhat order}.

\medskip
Originally motivated by the study of moduli of curves, we are naturally led to the question of generalizing the aforementioned connection to the case where (on the quiver-representation side) the underlying quiver is more complicated than a cycle, which (on the affine flag variety side) corresponds to fibered products of affine Schubert cells. The complexity of the induced quiver is due to the fact that the vertices of the quiver correspond to the multi-degrees of limit linear series on the (reducible) special fiber of a \textit{smoothing family} (Definition~\ref{defn:smoothing family}) of curves. If we want to consider reducible curves with more complicated topological configurations, we will be facing with more intricate collections of multidegrees. %\com{Actually the complexity of the quiver depends more on the size of the multidegrees than the configuration of curve}. 
Intuitively, this is the combinatorial manifestation of the more convoluted relations among line bundles with different multidegrees that arise as degenerations of the same line bundle on the generic fiber.
\medskip

Towards this direction, for an arbitrary convex lattice configuration $\Gamma$, we consider two stratifications of the special fiber $LG_r(\Gamma)_0$ of its associated linked Grassmannian, which is naturally identified with certain quiver Grassmannian:
\begin{itemize}
    \item the \textit{quiver stratification} (Notation~\ref{nt:quiver grass}) induced by isomorphic classes of quiver-representations;
    \item the \textit{rank-vector stratification} (Notation~\ref{nt:rank vector}) according to the ranks of compositions of morphisms in a quiver representation.
\end{itemize}
For configurations contained in one apartment of the building, we introduce
\begin{itemize}
    \item the \textit{Bruhat stratification} (Definition~\ref{defn:Bruhat_strata}) induced by fiber products of Schubert cells.
\end{itemize}
For the above stratifications, we investigate the corresponding closure-containment relations between the strata, i.e., their \textit{topological order} (Notation~\ref{nt:topological order}). Moreover, for the Bruhat strata, we introduce a notion of \textit{generalized Bruhat order} (Definition~\ref{defn:order1}) in terms of the ``fiber product" of the usual Bruhat order, while for the rank-vector strata, there is a natural \textit{rank-vector order} (Notation~\ref{nt:rank vector}) induced by the ``dominance" between vectors in $\mathbb Z^d$. Our main results then lie in the comparison of the aforementioned stratifications and their associated order.
\medskip

Historically, we know that the Bruhat stratification is identified with the quiver stratification when $\Gamma$ is a simplex (\cite[Theorem 11.3]{lusztig1990canonical}), and it is the same as rank-vector stratification when $\Gamma$ is the standard alcove (\cite[Proposition 2.6]{gortz2010supersingular}). Moreover, a sequence of recent papers \cite{feigin2021totally,feigin2023generalized} by Feigin et. al. interprets the Bruhat stratification, again, when $\Gamma$ is the standard alcove, as the stratification of a certain quiver Grassmannian induced by the orbits of the automorphism group of the ambient quiver representation. They also introduced several orders on the strata, and studied their connection to the Bruhat order. It turns out that this orbit stratification agrees with the rank-vector order even when $\Gamma$ is just a simplex (see Remark~\ref{rem:rank n orbit strata}) hence is the same as our quiver stratification by the main theorem below.

\medskip
In this paper, we show that, for an arbitrary lattice configuration contained in an apartment, the rank-vector stratification agrees with the Bruhat-stratification (Lemma~\ref{lem:bruhat order and rank vector}), and that the generalized Bruhat order and rank vector order also coincide (Corollary~\ref{cor:bruhat order=rank order}). The strategy is to reduce the question to the case when $\Gamma$ is a simplex, using specific features of the underlying quiver Grassmannian of $LG_r(\Gamma)_0$ developed in \cite{he2023degenerations} (see also Section~\ref{subsec:quiver grass preliminary}). 
For locally weakly independent configurations, which are known to be contained in an apartment (Theorem~\ref{thm:decomposition}), 
there is much more we can say:

\begin{thm}\label{thm:intro 1}
Let $\Gamma$ be a locally weakly independent lattice configuration.  Then 
\begin{enumerate}
    \item {$\mathrm{(Theorem~\ref{thm:decomposition})}$} every point of $LG_r(\Gamma)_0$ as a quiver representation decomposes into a direct sum of irreducible representations of dimension at most $(1,\dotsc,1)$; and
    \item {$\mathrm{(Theorem~\ref{thm:stratification})}$} the quiver stratification of $LG_r(\Gamma)_0$ is the same as the rank-vector and Bruhat stratifications, and the topological order agrees with the rank-vector order and generalized Bruhat order; and
    \item {$\mathrm{(Proposition~\ref{prop:projective dense}, 
    Theorem~\ref{thm:flatness})}$} $LG_r(\Gamma)$ is Cohen-Macaulay, integral, and flat over the base with reduced fibers. Moreover, the irreducible components of $LG_r(\Gamma)_0$ are indexed by isomorphic classes of projective quiver representations.
\end{enumerate}
 \end{thm}

Part (1) and (3) of Theorem~\ref{thm:intro 1} generalize the main conclusion in \cite{he2023degenerations} concerning locally linearly independent lattice configurations.  In Section~\ref{sec:r1},
a similar result as Theorem~\ref{thm:intro 1} (2) is obtained for arbitrary $\Gamma$ when $r=1$, i.e., for linked Grassmannians as degenerations of projective spaces. 
The decomposition result in Theorem~\ref{thm:intro 1} (1) is the key tool that connects the quiver stratification and rank-vector stratification, as well as the topological order and rank-vector order; it also provides a way to analyze whether a quiver representation in $LG_r(\Gamma)_0$ can be realized as the degeneration of a point in the Grassmannian, hence facilitating the deduction of Theorem~\ref{thm:intro 1} (2) and (3). As another application, we compute all possible rank vectors that appear in the special fiber of local models
(Proposition~\ref{prop:simplex strata}),  providing an alternative perspective to the existing result in \cite[\S~2.5]{gortz2010supersingular}.

\medskip
We would also like to mention that the irreducibility of general linked Grassmannians plays an important role in the smoothing theorem of limit linear series (see Section~\ref{sec:lls}). 
Besides our previous paper (\cite{he2023degenerations}), Esteves, Santos and Vital also applied techniques of quiver representations to the study of degenerations of linear series (\cite{esteves2021quiver1}, \cite{esteves2021quiver2}). Based on our own experience, we would like to make some general comments on this approach. First of all, the approach is most efficient if the relevant representation types are all sums of irreducible representations of dimension at most $(1,1\hh 1)$. If this condition is not met, it may still be possible to classify the interesting representation types and draw interesting geometric conclusions, including the irreducibility of linked Grassmannians. We will compute some further examples in this direction in a subsequent paper \cite{hz3}. In particular, we will show that the loci of projective representations dominate the loci of other representation types for certain $LG_r(\Gamma)_0$ where $\Gamma$ is not locally weakly independent. 

\medskip
On the other hand, we believe that in order to establish stronger smoothing theorems (for an example, see Theorem \ref{thm:smoothing of lls}) in this field, it is necessary to study more general linked Grassmannians, for whose special fibers the representation-theoretic analysis could be too complicated using existing methods. \textit{A priori}, the rank vectors 
%the ranks of compositions of linear maps along sub-paths of the quiver 
may fail to control the topology of the special fibers, which is a phenomenon that did not occur in our previous studies. Meanwhile, we believe that a detailed study of the geometry of the linked Grassmannian associated to the lattice configuration $\SS_d\cd \omega$, where $\omega$ is the standard alcove in the Bruhat-Tits building of $\SL_d$, will likely produce the essential ingredient for some general smoothing theorem. To this end, we gradually came to realize that the theory of affine flag varieties provides us with certain structures in the topological data of the special fibers, even in the most general situations. Although a thorough exploration of the general case is beyond the scope of the current paper, we would like to promote the study of related geometric problems. Among other things, the geometry of fibered products of certain affine Schubert cells seems to be an interesting topic of its own right, which we would very much be willing to return to, again, in \cite{hz3}. 

\medskip
In a related but slightly different direction, it is also perceivable that we may further extend our scope and study similar geometric objects whose underlying algebraic group is not of type A. We believe the most promising case is the symplectic group, drawing from multiple findings in the literature. First of all, Osserman and Teixidor-i-Bigas constructed a version of \textit{linked symplectic Grassmanian} in \cite{osserman2014linked}, which laid the foundation for advancements in higher-rank Brill-Noether theory for algebraic curves (\cite{Osp}). On the other hand, local models of Shimura varieties for symplectic groups are well-studied (\cite{gortz2003flatness}), from which we can borrow many group-theoretic techniques again.\footnote{It is worth mentioning that the schemes constructed by Osserman and Teixidor-i-Bigas do not seem to be isomorphic to local models of Shimura varieties in any obvious way. This is quite different from the type A case.} Seeing how local models of Shimura varieties were first defined in terms of moduli of lattice chains (\cite{rapoport2016period}), but later generalized using group-theoretic definitions (\cite{pappas2013local}, \cite{fakhruddin2022singularities}), we think this two-perspective framework will be essential for defining and studying those potential new objects.

%Finally, we would like to mention that a sequence of recent papers \cite{feigin2021totally,feigin2023generalized} by Feigin et. al. also studied certain stratifications and their closure-containment relations for the quiver Grassmannians associated to a cyclic quiver, which corresponds to the special fiber of a local model and, in our case, to $LG_r(\Gamma)_0$ with a simplex $\Gamma$.
%It would be interesting to investigate the connections between their stratifications and the ones in the present paper. 

\subsection*{Roadmap} The plan for the remainder of this paper is as follows. In Section 2, we review the necessary basics of Bruhat-Tits theory and affine flag varieties, as well as the quiver Grassmannian associated to a convex lattice configuration in the Bruhat-Tits building. In particular, the existing (and some new) structural results for linked Grassmannians associated to a simplex are summarized in %Section~\ref{subsec:linked grass one simplex}. Given the technical and notational complexity, we provide a list of precise statements in 
Theorem \ref{thm:open_cover1}, some of which is adapted to our setting, with further detail provided. In Section 3, we establish some structural results that hold for linked Grassmannians associated to \textit{arbitrary} convex lattice configurations inside one apartment of the Bruhat-Tits building. In Proposition \ref{prop:std_open}, we establish the existence of a natural open covering of a such linked Grassmannian, which is the object that governs the topology and some geometry. In addition, we construct the \textit{Bruhat stratification} of the special fiber of the linked Grassmanian which is equipped with the \textit{generalized Bruhat order}. Each stratum is contained in the corresponding open subset of the aforementioned covering.  The foundation is the fact that the special fiber of the linked Grassmannian can always be embedded into a fibered product of affine flag varieties.
In Section 4 we introduce the notion of \textit{locally weakly independent} lattice configuration, and prove Theorem~\ref{thm:intro 1} using the technique of quiver representations, thus justifying the relevance of the generalized Bruhat order. In Section 5, we consider the linked Grassmannian $LG_r(\Gamma)$ for arbitrary $\Gamma$ in the case $r=1$ and prove again that the generalized Bruhat order agrees with the rank-vector order as well as the topological order.
%and 5, we discuss two special classes of lattice configurations for which the generalized Bruhat order coincides with the topological order, thus justifying the relevance of the former. In particular, using the technique of quiver representation, we derive Theorem~\ref{thm:intro 1} regarding the geometry of linked Grassmannian associated to \textit{locally weakly independent} configurations. %, which is a generalization of our previous result in \cite{he2023degenerations}. The same technique can be further strengthened to prove irreducibility of linked Grassmannians associated to some more general configurations (Theorem \ref{thm:two cycle}). A major observation is that when points corresponding to projective representations dominates the special fiber, the liniked Grassmannian is irreducible. 
Eventually in Section 6, we derive a smoothing theorem for limit linear series with certain multidegrees over reducible curves whose dual graph is the complete graph $K_n$. 
%\subsection*{Degeneration and filtration}

\subsection*{Acknowledgement}\addtocontents{toc}{\protect\setcounter{tocdepth}{1}}
We would like to thank Ulrich G\"ortz for helpful conversations. 

\subsection*{Conflict of interest}\addtocontents{toc}{\protect\setcounter{tocdepth}{1}}
The authors state, that there is no conflict of interest.

\subsection*{Data availability statement}\addtocontents{toc}{\protect\setcounter{tocdepth}{1}}
No datasets were created or analysed during the current study.

\subsection*{Notation and Conventions}\addtocontents{toc}{\protect\setcounter{tocdepth}{1}}%\com{maybe some notations should go to their own sections}
\begin{nt}
Throughout the paper, $\kappa$ will always be an algebraically closed field, $R$ will always be a discrete valuation ring with residue field $\kappa$ and fraction field $K$, and $\pi$ will always be a uniformizer of $R$. We do not assume $\kappa$ is of characteristic zero unless otherwise stated. The ring of power series over $\kappa$ will be denoted by $\kappa[\![t]\!]$ and its fraction field by $\kappa(\!(t)\!)$.
\end{nt}

%\begin{nt}
%We use $V$ to denote a $d$-dimensional vector space over $K$. We denote $\mathfrak B_d(R)$ for the Bruhat-Tits building associated to $\PGL(V)$ (\cite{abramenko2008buildings}). We also denote $\mathfrak B^0_d(R)$ to be the set of homothety classes of lattices in $V$. We will mainly use the notion of $\mathfrak B^0_d(R)$, and drop the reference to the DVR $R$ whenever it is clear from the context.   
%\end{nt}

\begin{nt}
For a graph $G$, we denote by $V(G)$ and $E(G)$ the set of vertices and edges of $G$, respectively. If $\ell$ is a path (e.g., a directed edge) in $G$, we denote $s(\ell)$, $t(\ell)$ for its source and target, respectively. By a cycle we mean a path $\ell$ such that $s(\ell)=t(\ell)$ is the only repeating vertex. 
\end{nt}

\begin{nt}\label{nt:quivers}
Throughout the paper, quiver Grassmannians are defined over $\kappa$. Let $Q$ be a quiver. We denote by $Q_0$ (resp. $Q_1$) the set of vertices (resp. arrows) of $Q$.
Let $M$ be a representation of a quiver $Q=(Q_0,Q_1)$ with the data $(f_{\ell})_{\ell\in Q_1}$ of linear maps between the underlying vector spaces $(M_i)_{i\in Q_0}$. We will often write the representation as $M=(M_i)_i$ whenever $Q_0$ and $(f_\ell)_{\ell\in Q_1}$ are clear from the context, and sometimes use subscript $v$ instead of $i$ for vertices in $Q_0$.
 A \textit{sub-representation} of $M$ is represented by a tuple of vector spaces $(V_i)_i$ such that $V_i\subset M_i$ and $f_{\ell}(V_{s(\ell)})\subset V_{t(\ell)}$ for all $\ell\in Q_1$. For a path $\ell'=\ell_n\cdots \ell_1$ in $Q$ where $\ell_i\in Q_1$ we denote $f_{\ell'}$ the compositions of all $f_{\ell_i}$. 
\end{nt}

\begin{nt}\label{nt:quiver grass}
    Let $M$ be a representation of a quiver $Q$. Let $\underline x=(x_i)_{i\in Q_0}$ be a dimension vector. The quiver Grassmannian of dimension $\underline x$ associated to $M$ is denoted by $\Gr(\underline x,M)$. There is a natural stratification of $\mathrm{Gr}(\underline x,M)$ given by the family of subsets $(\mathcal S_{M'})_{[M']}$, where $[M']$ runs over all isomorphic classes of sub-representations of $M$ of dimension $\underline x$ and $\mathcal S_{M'}\subset \mathrm{Gr}(\underline x,M)$ is the subset parametrizing sub-representations isomorphic to $M'$. We call this the \textit{quiver stratification} of $\Gr(\underline x,M)$.
\end{nt}

\begin{nt}\label{nt:dimension vector number}
    We reserve the symbol $\mathbf r$ and $\mathbf 1$ respectively for the dimension vector $(r\hh r)$  and $(1\hh 1)$ of representations of a given quiver.
\end{nt}

\begin{nt}\label{nt:topological order}
    Let $X$ be a topological space. The closure of a subspace $Y$ of $X$ is denoted by $Y^c$. Suppose $(S_x)_x$ is a stratification of $X$. We define a partial order $\preceq_{\text{top}}$ on the set of strata such that $x\preceq_{\text{top}}y$ if and only if $S_x\subset S_y^c$. We call it the \textit{topological (partial) order}.
\end{nt}
\begin{nt}
     We associate a partial order ``$\leq$" on $\mathbb Z^n$ such that $(d_i)_i\leq (d'_i)_i$ if $d_i\leq d'_i$ for all $i$. 
\end{nt}

\begin{nt}\label{nt:set of numbers}
Let $n$ be any positive integer.  Denote by $[n]$ the set $\{0,1,...,n\}$.
\end{nt}

\begin{nt}\label{nt:standard_lattice}
Throughout the paper, $r<d$ will be two positive integers. Let $e_1\hh e_d$ be the standard basis of $K^d$. Denote $\omega$ to be the standard complete lattice chain, i.e.
\[\omega_i=\lag \pi^{-1}e_1\hh \pi^{-1}e_{i},e_{i+1}\hh e_d\rag,i=0\hh d-1.\]
For a subset $I$ of $[d-1]$, let $\omega_I$ be the lattice chain consisting of $\omega_i$ where $i\in I$. Denote $[\omega_i]$ for the homothety class of $\omega_i$. %When writing a lattice in place of an actual homothety class, we mean that a choice of representative has been made.  
The simplex in the Bruhat-Tits building consisting of all $[\omega_i]$ will be denoted by $\Omega$. 
\end{nt}
\begin{nt}\label{nt:groups}
We fix the following group-theoretic notations:
\begin{enumerate}
\item $\widetilde{W}=\SS_d\ltimes\ZZ^d$ is the extended affine Weyl group of $\GL_d$. A general element is denoted as a pair $(\sg,\underline v)$, where $\sg\in\SS_d$ and $\underline v\in\ZZ^d$. We reserve the symbol $\cd $ for the action of $\wt{W}$ on $\ZZ^d$,
namely, $(\sigma,\underline v)\cdot (a_1,\dotsc,a_d)=(a_{\sigma^{-1}(1)},\dotsc,a_{\sigma^{-1}(d)})+\underline v$.
 The multiplication of two elements $g,h\in \wt W$ (or $g,h\in\GL_d$) will simply be denoted $gh$. The group law on $\wt{W}$ is given by
\[(\sg_1,\underline{v}_1)(\sg_2,\underline{v}_2)=(\sg_2\sg_1,\underline{v}_1+\sg_1\cd \underline{v}_2).\]
\item $W_a=\SS_d\ltimes\{(a_i)\in\ZZ^d\mid \sum a_i=0\}$ is the affine Weyl group of $\SL_d$. 
\item We denote $s_0\hh s_{d-1}$ to be the Coxeter generators of $W_a$: $s_0=((1,d),(1,0\hh 0,-1))$ and $s_i=((i,i+1),(0,\dotsc,0))$ for $i=1\hh d-1$. We use $\mathfrak l(\cdot)$ to denote the word length of elements in $W_a$ with respect to these Coxeter generators. As we shall see, up to the choice of a reference alcove, $\mathfrak l(\cd)$ extends to $\wW$.   
%\item We denote $\ov{s}_i=(i,i+1)$ ($i=1\hh d-1$) to be the transpositions which generates the finite Weyl group $W=\SS_d$.  
\item\label{itm:Iwahori} Let $B$ be the subgroup of $\SL_d(\kappa[\![t]\!])$ consisting of elements $(g_{i,j})$ such that $g_{i,j}\in t\cdot\kappa[\![t]\!] $, for all $i>j$. Let $N$ be the subgroup of $\SL_d(K)$ consisting of monomial matrices.
%\item $\fa I\subset \{0\hh d-1\}$, the Parahoric group scheme $P_I$ is defined to be the stabilizer of all $[\omega_i]$ such that $i\in I$. As a special case, $B=P_{0\hh d-1}$. We also denote $B_g$ for the stabilizer group of the complete lattice chain $g\cd\omega$, for $g\in\wt{W}$. 
\end{enumerate} 
\end{nt}
\begin{conv}\label{nt:standard alcove}
Following the convention in existing literature, an \textit{alcove} in a building is a simplex of maximal dimension, and a \textit{face} is a sub-simplex of a given simplex.  
\end{conv}
%\begin{nt}\label{nt:tau}
%Denote $\tau$ to be the alcove $((1^r,0^{d-r}),(1^{r+1},0^{d-r-1})\hh(2^{r-1},1^{d-r+1}))$.
%\end{nt}
\begin{nt}\label{nt:tau2}
Denote $\iota:=((12\hdots d),(1,0^{d-1}))\in \wW$. %, and $\tau(r,d)=\iota^r$. Alternatively, $\tau(r,d)$ is the element in $\wW$ such that $\tau(r,d)\cd \omega=\tau$. 
\end{nt}
\begin{nt}
Let $\underline v=(v(1)\hh v(d))\in \ZZ^d$. Denote $\sum(\underline v)=\sum_{k=1}^dv(k)$.
\end{nt}

%\begin{nt}
%Given any positive integer $d$, we denote $\ep_j=(0\hh \underbrace{1}_{j-\text{th}}\hh 0)\in\ZZ^d$. \com{this notation is not used}
%\end{nt}

\begin{nt}\label{nt:stab2}
When $K=\kappa(\!(t)\!)$, given $\underline a=(a_1,a_2\hh a_d)\in\ZZ^d$, denote $G_{\underline a}$ to be the following subgroup of $\SL_d(K)$:
\[\{(g_{i,j})\in\SL_d(K)\mid g_{i,j}\in t^{a_j-a_i}\cd R\text{ for all }i,j\}\]
One can check that $G_{\underline a}$ is the stabilizer subgroup of the lattice $\lag t^{-a_1}e_1, t^{-a_2}e_2\hh t^{-a_d}e_d\rag$. More generally, let $A$ be a finite subset of $\ZZ^d$, denote $G_{A}$ to be $\bigcap_{\underline a\in A} G_{\underline a}$, which is the common stabilizer subgroup of all relevant lattices. In particular, given a non-empty subset $I=\{i_1\hh i_m\}$ of $[d-1]$, then
\[G_I:=\{(g_{i,j})\in\SL_d(K)\mid g_{i,j}\in t^{\max_s\{a^{i_s}_j-a^{i_s}_i\}}\cd R\}
\mathrm{\ where\ } \underline a^{i_s}=(a^{i_s}_k)_k=(1^{i_s},0^{d-i_s}),\] is the stabilizer subgroup of $\omega_I=\{\omega_i\}_{i\in I}$ (equivalently, $\{[\omega_i]\}_{i\in I}$). It is straightforward to check that $B=G_{[d-1]}$ (Notation~\ref{nt:groups} (4)). 
\end{nt}

\section{Backgrounds}
In this section, we review the necessary basics of Bruhat-Tits theory and affine flag varieties, as well as the quiver Grassmannian associated to a convex lattice configuration in the Bruhat-Tits building.
In \S~2.2 and \S~2.3, we set $K=\kappa(\!(t)\!)$, and $R=\kappa[\![t]\!]$.

\subsection{Bruhat-Tits building for $\SL_d$}
%In this paper, we will only be using Bruhat-Tits theory of affine buildings for $\SL_d$ ($d\ge 3$) over $\kappa(\!(t)\!)$.\com{needs to be refurbished} 
We first review the lattice-chain description of the Bruhat-Tits building. 
For a general reference on the theory of buildings, see \cite{abramenko2008buildings}.

\begin{defn}
A \textit{lattice} in $K^d$ is a free $R$-module $L$ such that $L\ot_{R}K\to K^d$ is an isomorphism. Two lattices $L_1$ and $L_2$ are \textit{homothetic} if $L_1=aL_2$ for some $a\in K$. Denote by $[L_1]$ the equivalence class of $L_1$ under the homothety equivalence relation. 
Two lattice classes $[L],[L']$ are said to be \textit{adjacent} if there exist $L_1,L_2$ representing $[L],[L']$ respectively, and $\pi L_1\subset L_2\subset L_1$. 
\end{defn}
\begin{defn}\cite[\S 6.9.3]{abramenko2008buildings}
The \textit{affine building} $\fb_d$ associated to $\SL_d$ is the flag complex of homothety classes of lattices in $K^d$, under the incidence relation of adjacency. It is a simplicial complex of dimension $d-1$. %The fundamental alcove is the simplex with vertices $[\omega_i], i= 0\hh d-1$ (Notation \ref{nt:standard_lattice}).
\end{defn} 
\begin{rem}
The general theory of Bruhat and Tits associates to a BN-pair (\cite[Definition 6.55]{abramenko2008buildings})), $(B,N)$, of a reductive group $G$ a \textit{thick building} (\cite[Definition 4.1, Remark 4.2]{abramenko2008buildings}) %$\Delta(G,B)$\com{I think this symbol can be removed}
which admits a strongly-transitive $G$-action (\cite[Theorem 6.56]{abramenko2008buildings}). For us, $G$ is $\SL_d(K)$, $B$ is the standard Iwahori subgroup (Notation \ref{nt:groups}(\ref{itm:Iwahori})) of $\SL_d(K)$ and $N$ is the subgroup of $\SL_d(K)$ consisting of monomial matrices.
  \end{rem}
The building has a far more-refined structure than just being a flag complex: it is a union of special sub-complexes called \textit{apartments}.
Throughout the paper, we fix a  \textit{standard apartment} $\cA$---the sub-complex of $\fb_d$ spanned by all vertices of the form $[\lag \pi^{a_1}e_1\hh \pi^{a_d}e_d  \rag]$ where $\{e_i\}_i$ is a fixed basis of $K^d$. We also call the simplex spanned by $[\omega_i]$ ($i=0\hh d-1$) (Notation \ref{nt:standard_lattice}) the \textit{standard alcove} and denote by $\Omega$. 
Note that vertices in $\cA$ are naturally identified with elements in $\ZZ^d/(1\hh 1)$. Following existing literature, we adopt the following combinatorial convention:
\begin{conv}\label{conv:alcove}
We identify $(a_1\hh a_d)$ with the lattice $\lag \pi^{-a_1}e_1\hh \pi^{-a_i}e_{i}\hh \pi^{-a_d}e_d \rag$, and $[a_1\hh a_d]$ with the corresponding homothety class, if no confusions occur. An alcove is often represented by an element $(v_i)_{i=0}^{d-1}\in (\ZZ^d)^d$, where $v_i$ is a representative of the $i$-th vertex. Similarly, a face of an alcove is represented by an element $(v_i)_{i\in I}\in (\ZZ^d)^{|I|}$, for some appropriate index set $I$. %We often identify an element in $g\in \ti{W}$ with the alcove $(g\cd [\omega_i])_{i=0}^{d-1}$. In particular, the identity element corresponds to the standard alcove.
\end{conv}
We shall adopt some terminologies from the theory of buildings in describing general linked Grassmannians. For us, a \textit{lattice configuration} $\Gamma$ is defined to be a finite collection of vertices in $\cB_d$. The simplicial complex spanned by $\Gamma$ will be denoted $\Theta(\Gamma)$. A configuration $\Gamma= \{[L_i]\}_{i\in \cI}$ is said to be \textit{convex} if for any representatives $L_1$, $L_2$ of two members of $\Gamma$, we have $[L_1 \cap L_2]\in \Gamma$.  Moreover, we make the following definition:
\begin{defn}\label{defn:lattice quiver}
    Given a configuration $\Gamma$ in $\cA$, let $\tau_1\hh\tau_m$ be the maximal faces of $\Theta(\Gamma)$. A \textit{lattice quiver} representing $\Gamma$ is a choice of representatives for the homothety classes in $\Gamma$ such that for every $j$, if $\tau_j$ is of dimension $n_j$, then the chosen representatives for its vertices may be ordered in the following form: $L_{j,0}\subset L_{j,1}\subset\hdots \subset L_{j,n_j}\subset \pi^{-1}L_{j,0}$.     
\end{defn}

The lattice quiver always exists. For example, one can take the representative of each lattice class corresponding, as in Convention~\ref{conv:alcove}, to a (unique) integer vector with non-negative coordinates and at least one 0-coordinate.
%\begin{lem}
 %   For any finite convex configuration $\Gamma$, there is a lattice quiver representing it.
%\end{lem}
%\begin{proof}

%Let $\sg_1\hh \sg_n$ be the maximal faces of $\Theta(\Gamma)$. For $\sg_1$, one can fix a lattice quiver representing its vertices satisfying the following condition: for every vertex, the representative $(a_1\hh a_d)$ is chosen in such a way that all $a_i\ge 0$ and at least one of them is zero. Now, suppose representatives for all vertices in $\sg_1\hh \sg_k$ are chosen in this way. Consider vertices of $\sg_{k+1}$ which lie in $\sg_1\cup\hdots\cup\sg_k$. By the condition that each of them corresponds to a non-negative integer sequence containing at least one zero, one can conclude that the difference between any two of them must be of the form $\pm\mathbf\ep$, where $\mathbf\ep$ is a vector of zeros and ones; otherwise, they wouldn't represent vertices of the same face. It is not hard to see that one can then choose representative for other vertices of $\sg_{k+1}$, still satisfying the aforementioned condition, which together with the previously chosen representatives form a lattice quiver representing the vertex set of $\sg_{k+1}$. By induction, we are done. 
%\end{proof}

\subsection{Affine flag varieties}\label{sec:aff_flag} Since the Bruhat-stratification we are about to introduce in the next section is defined over the residue field $\kappa$, in the sequel (up until Section~\ref{subsec:linked grass one simplex}), we will only work over $K=\kappa(\!(t)\!)$ and its ring of integers $R=\kappa[\![t]\!]$ for simplicity and rigorousness.

As the building $\fb_d$ comes equipped with a natural $\SL_d$-action, some important subgroups of $\SL_d$ are realized as stabilizer subgroups of simplices in it:
\begin{defn}\label{defn:parabolic}\cite[5.2.6]{bruhat1984groupes}
Let $F$ be a simplex in $\cA$. The associated \textit{parahoric subgroup} $P_F\subset\SL_d(K)$ is defined to be the pointwise stabilizer of $F$. 
\end{defn}
\begin{defn}\label{defn:IW}
The \textit{Iwahori-Weyl group} $W_F$ corresponding to a simplex $F$ is defined by $W_F=(N\cap P_F)/T(R)$, where $T\subset \SL_d(K)$ is the subgroup of diagonal matrices. It is a subgroup of $W_a=N/T(R)$.
\end{defn}

Since vertices in $\cA$ are naturally identified with elements in $\ZZ^d/(1\hh 1)$, the affine Weyl group $W_a$ acts naturally on $\ZZ^d/(1\hh 1)$. Since $P_F$ stabilizes vertices of $F$, the Iwahori-Weyl group $W_F$ is the subgroup of elements in $W_a$ stabilizing the corresponding elements in $\ZZ^d/(1\hh1)$. One can think of them as combinatorial incarnations of the parahoric subgroups.  

\begin{ex}\label{ex:calc1}
As an example, $B$ stabilizes $[\omega_i]$ for $i=0\hh d-1$. Hence for the standard alcove we have $W_\Omega=(N\cap B)/T(R)=\{1\}$.

Let $F$ be the face of $\Omega$ spanned by $[\omega_0]\hh[\omega_{d-2}]$. Since $P_F=G_{[d-2]}$ (Notation \ref{nt:set of numbers} and \ref{nt:stab2}), we have 
\[W_F=\Big(T(R)\cup N_{d-1,d}(R)\Big)/T(R)\cong\{s_{d-1},\id\},\]
where $N_{d-1,d}(R)$ is the set of monomial matrices in $\SL_d(R)$ such that the non-zero entries in the $i$-th column lie on the diagonal, for $i=1\hh d-2$, and $s_{d-1}$ is as in Notation~\ref{nt:groups}.
\end{ex}
Finally, we will need the definition of the set $\prescript{}{F'}{\wW}^{F}$, where $F,F'$ are two faces of the standard alcove, which is a combinatorial object that will become useful in describing the topological properties of affine Schubert varieties and linked Grassmannians. 
\begin{defn2}\cite[Lemma 1.6]{richarz2013schubert}\label{defn:min_ele}
Let $w\in \ti{W}$.
\begin{enumerate}
    \item  There exists a unique element $w^F$ of minimal length in the left coset $wW_F$.
    \item There exists a unique element $\prescript{}{F'}{w}^{F}$ of maximal length in $\{(vw)^F\mid v\in W_{F'}\}$.
\end{enumerate}
Define $\prescript{}{F'}{\wW}^{F}:=\{\prescript{}{F'}{w}^{F}\mid w\in \ti{W}\}$.
\end{defn2}
It is well-known that $W_a$ is an infinite Coxeter group with respect to the Coxeter generators $s_0\hh s_{d-1}$ (Notation \ref{nt:groups}) and comes equipped with the Bruhat order (\cite[\S 3.4]{abramenko2008buildings}). The Bruhat order, 
$\preceq$, as well as the corresponding length function, $\mathfrak l(\cdot)$, extend to $\wW$ as follows: Fix a reference alcove and let $W$ be its stabilizer in $\ti{W}$, which is a cyclic group of order $d$ modulo $\ZZ\cd(1,1\hh 1)$ (see next example). Then any element in $\ti{W}$ can be uniquely written in the form $wt$, where $w\in W_a$ and $t\in W$; given $w_1,w_2\in W_a$ and $t_1,t_2\in W$, we have $w_1t_1\preceq w_2t_2$ if and only if $w_1\preceq w_2$ and $t_1=t_2$ (\cite[\S 1.8]{kottwitz2000minuscule}). We also define $\mathfrak l(wt)=\mathfrak l(w)$ for $w\in W_a$ and $t\in W$. 
\begin{ex}\label{ex:iota}
Take the reference alcove above as the standard alcove. For $(\sg,\underline{v})\in\widetilde W$ to stabilize (not necessarily pointwise) $[0^d],[1,0^{d-1}]\hh [1^{d-1},0]$, up to an integral multiple of $(1,1\hh 1)$, we may assume $\underline{v}=(1^i,0^{d-i})$ for some $i$ between $0$ and $d-1$. It is not hard to check that $\sg(1)=i+1$, $\sg(2)=i+2\hh \sg(d)=i$ must hold. In other words, the stabilizer group $W$ (modulo $\ZZ\cd(1,1\hh 1)$) is the cyclic group generated by $\iota:=((12\hdots d),(1,0^{d-1}))$.

As an example of length computation, take $w=(\Id,(1^r,0^{d-r}))$. We have $w=(w\iota^{-r})\iota^{r}$, where $w\iota^{-r}\in W_a$. Applying the well-known identification of $W_a$ with a subgroup of permutations on $\ZZ$ (\cite[\S 4.1]{Shi}), and then applying Shi's formula for the length function (Lemma 4.2.2 in \textit{loc.cit.}), one can conclude that $\mathfrak l(w):=\mathfrak l(w\iota^{-r})=r(d-r)$.   
\end{ex}
\begin{rem}\label{rem:order_double_coset}
Since there is a natural bijection between $\prescript{}{F'}{\wW}^{F}$ and the double coset $W_{F'}\backslash \ti{W}/W_F$, the Bruhat order on $\ti{W}$ induces a partial order on $W_{F'}\backslash \ti{W}/W_F$. %Note that the $(P_{F'},P_F)$-double cosets (a.k.a. the $(P_{F'},P_F)$-Schubert cells (Definition \ref{defn:Schubert})) of $\SL_d(K)$ can be indexed by elements of $W_{F'}\backslash W_a/W_F$ (\cite[\S 8]{pappas2008twisted}). 
\end{rem}
\begin{ex}
There is a natural set map $\ti{W}\to \prescript{}{F'}{\wW}^{F}$ sending $w$ to $\prescript{}{F'}{w}^{F}$. For example, when $F$ and $F'$ are as in Example \ref{ex:calc1}, $\prescript{}{F'}{w}^{F}$ is simply the longer word among $w$ and $s_{d-1}w$. 
\end{ex}
%\subsection{Affine Schubert varieties}\label{subsec:affine schubert variety}
Compared to our approach in \cite{he2023degenerations}, a new ingredient in the current paper is to use the theory of affine Schubert varieties to study the geometry of linked Grassmannians. This is a feasible approach, because, when the lattice configuration is given by the vertex set of a simplex in $\cA$, the special fiber of the linked Grassmannian can be embedded inside the relevant (partial) affine flag variety, which we will introduce next. 
Based on Bruhat-Tits theory (\cite[3.4.1]{tits1979reductive}), given a parahoric subgroup $P_F$, there is smooth affine group scheme $\cP_F$ over $R=\kappa[\![t]\!]$, whose generic fiber is $\SL_d$ and whose set of $R$-points is $P_F$. It is called the \textit{Parahoric group scheme} associated to $F$. 

\begin{defn}
Let $F$ be a simplex in $\fb_d$ and $\cP_F$ its associated parahoric group scheme. The associated \textit{affine flag variety} $\cF_F$ is defined to be the fpqc sheaf associated to the presheaf $(\kappa-\text{algebra})\to(\text{Sets}):S\mapsto \SL_d(S(\!(t)\!)))/\cP_F(S[\![t]\!])$. %We reserve the symbol $\cF$ for the affine flag variety associated to the standard alcove in $\cA$. 
\end{defn}

The definition is a special case of the one in \cite[\S 1.c]{pappas2008twisted}. It is well-known that $\cF_F$ is represented by an ind-$\kappa$-scheme of ind-finite type over $\kappa$ (Theorem 1.4 in \textit{loc.cit.}). Moreover, the affine flag variety for the special linear group enjoys a more concrete description in terms of lattice chains. For the sake of self-containment, we also provide a proof. (In the case that $F$ is a face of the standard alcove in $\cA$, see \cite[\S 4]{PR2}.) 

%the statement is as follows: 

%\begin{lem}\label{lem:concrete_description}$($\cite[\S 4]{PR2}$)$
%Let $I=\{i_1\hh i_m\}$ be a subset of $[d-1]$ and $F$ be the face of the standard alcove in $\cA$ whose vertex set is $\{[\omega_i]\mid i\in I\}$. There is an equivariant isomorphism of ind-schemes between $\cF_F$ and $\cF_I$, where for any $\kappa$-algebra $S$, the set $\cF_I(S)$ parametrizes lattice chains $\Lambda_{i_1}\subset\hdots\subset \Lambda_{i_m}\subset t^{-1}\Lambda_{i_1}$ in $S(\!(t)\!)^d$ satisfying the following conditions:
%\begin{enumerate}
 %   \item $\Lambda_{i_{k+1}}/\Lambda_{i_k}$ (resp. $t^{-1}\Lambda_{i_1}/\Lambda_{i_m}$) is a locally-free $S$-module of rank $i_{k+1}-i_k$, where $k=1\hh m-1$ (resp. $d+i_1-i_m$); 
 %   \item $\det \Lambda_{i_1}=\det(\omega_{i_1,S})=t^{-i_1}S[\![t]\!]$.%\footnote{As pointed out by Pappas and Rapoport in\textit{ loc.cit.}, the stated isomorphism of ind-schemes holds for any choice of the integer $r'$.}    
%\end{enumerate}
%\end{lem}
%For our application, we need a slightly more general version of the statement. Since we couldn't find a reference, we sketch a proof hereafter. 
\begin{prop}\label{prop:chain}
    Let $I=\{i_1\hh i_m\}$ be a subset of $[d-1]$ and $g\in\wW$. Let $F$ be the face in $\cA$ with vertex set given by $\{g\cd [\omega_i]\mid i\in I\}$.  The affine flag variety $\cF_F=\SL_d/\cP_F$ is equivariantly isomorphic to the ind-$\kappa$ scheme $\cF_{I,g}$ such that for any $\kappa$-algebra $S$, the set 
    $\cF_{I,g}(S)$ parametrizes lattice chains $\Lambda_{i_1}\subset\hdots\subset \Lambda_{i_m}\subset t^{-1}\Lambda_{i_1}$ in $S(\!(t)\!)^d$ satisfying
    %parametrizing lattice chains satisfying the following conditions:
\begin{enumerate}
    \item $\Lambda_{i_{k+1}}/\Lambda_{i_k}$ (resp. $t^{-1}\Lambda_{i_1}/\Lambda_{i_m}$) is a locally-free $S$-module of rank $i_{k+1}-i_k$, where $k=1\hh m-1$ (resp. $d+i_1-i_m$); 
    \item $\det \Lambda_{i_1}=\det (g\cd \omega_{i_1,S})$.
\end{enumerate}   
\end{prop}
\begin{proof}
    Note that there exist natural maps $\SL_d(S(\!(t)\!))\to \cF_{I,g}(S):h\mapsto (hg)\cd \omega_{I,S}$. %since $\SL_d(S(\!(t)\!))$ sends one lattice chain to another, without altering the determinants of the terms, nor the ranks of the quotients between consecutive terms.
    
    Define $\aut_{g\cdot \omega_I}$ to be the group scheme representing the following functor: 
    \[(R-\text{algebras})\to(\text{Sets}):S\mapsto \{(g_i)_{i\in I}\mid g_i\in\SL_d(S)\mathrm{\ for\ all\ }i\in I\mathrm{\ and\ }\ (\star)\text{ is commutative}\}\]
    where $(\star)$ is the following diagram:
    \[\begin{tikzcd}
        g\cd\omega_{i_1,S}\ar[r]\ar[d,"g_{i_1}"]&g\cd\omega_{i_2,S}\ar[r]\ar[d,"g_{i_2}"]&\hdots\ar[r]&g\cd\omega_{i_m,S}\ar[d,"g_{i_m}"]\\
        g\cd\omega_{i_1,S}\ar[r]&g\cd\omega_{i_2,S}\ar[r]&\hdots\ar[r]&g\cd\omega_{i_m,S}
    \end{tikzcd}
    \]
and $g\cd\omega_{i,S}$ stands for the $S$-lattice $\lag t^{-(g\cd\omega_i)(1)}e_1,t^{-(g\cd\omega_i)(2)}e_2\hh t^{-(g\cd\omega_i)(d)}e_d \rag$. 
    
    By definition, $P_F$ is the pointwise stabilizer subgroup of vertices of $F$ in $\SL_d(\kappa(\!(t)\!))$, or equivalently, $P_F$ is the stabilizer group of the $R$-lattice chain $g\cd \omega_I$. In fact, this can be strengthened into the statement that $\cP_F$ agrees with $\aut_{g\cd\omega_I}$. (For a reference to the analogous statement in the case of $\GL_d$, see \cite{haines4introduction}.) This follows from the fact that they are \'etoff\'e schemes (\cite[1.7]{bruhat1984groupes}) with identical generic fibers $\SL_{d,K}$:
    \begin{enumerate}
        \item The claim on the generic fibers follows from the definition of $\cP_F$, together with the observation that lattices become $d$-dimensional $K$-vector spaces over the generic point of $\spec(R)$.    
        \item $\cP_F$ is known to be  \'etoff\'e (\cite[4.6.6]{bruhat1984groupes}). 
        \item To see that $\aut_{g\cd\omega_I}$ is \'etoff\'e, one can check that the restriction map $\aut_{g\cd\omega_I}(R)\to\aut_{g\cd\omega_I}(\kappa)$ is surjective by a modification of the argument in page 135 of \cite{rapoport2016period}. %Indeed, the lattice cycle $(L_i)_{i\in I}$ given by $g\cd \omega_I$ restricts to a quiver representation $(V_i)_{i\in I}$ isomorphic to $\bigoplus_{i\in I}P_i^{\op m_i}$, where $P_i$ is a $\mathbf 1$-dimensional projective representation generated by a vector in $V_i$ and $\sum m_i=d$. Fix generators $(f_{i,t_i})^{1\le t_i\le m_i}_{i\in I}$ of the irreducible summands in such a way that each $f_{i,t_i}:=\ov{e}^i_k$ is the image of $e^i_k=t^{-(g\cd\omega_i)(k)}e_k\in L_i$ for some $k$. 
        %Suppose $(g_i)_{i\in I}\in\aut_{g\cd\omega_I}(\kappa)$ and $g_i(f_{i,t_i})=\sum_{s=1}^d a_s\ov{e}_{i,s}$. Then a lift of $(g_i)_{i\in I}$ to $\aut_{g\cd\omega_I}(R)$, $(\ti{g}_{i})_{i\in I}$, can be prescribed as follows: for an index $k\in\{1,2\hh d\}$ such that $\ov{e}^i_{k}=f_{i,t_i}$, up to an appropriate scaling in one column by some unit in $R$ with constant term equal to 1 (so as to set determinants to 1), \[\ti{g}_{j}(e^j_k)=\sum_{s=1}^d a_st^{\ep_j(k,s)} e^j_s\] where $\ep_j(k,s)=[(g\cd\omega_i)(k)-(g\cd\omega_j)(k)]-[(g\cd\omega_i)(s)-(g\cd\omega_j)(s)]$. 
        Meanwhile, the special fiber $(\aut_{g\cd\omega_I})_0$ is a closed sub-scheme of a finite product of special linear groups over $\kappa=\overline{\kappa}$, so that its closed points are dense. Hence the image of $\aut_{g\cd\omega_I}(R)$ in $(\aut_{g\cd\omega_I})_0$ is dense. By \cite[1.7.2]{bruhat1984groupes}, $\aut_{g\cd\omega_I}$ is \'etoff\'e.
    \end{enumerate}    
    Thus, we get functorial injective morphisms
    \[\SL_d(S(\!(t)\!))/\cP_F(S)\to \cF_{I,g}(S)\]
    By definition, the $\SL_d$-action is always preserved. 

    Meanwhile, the same argument as in \cite[Porposition A.4, Ch3]{rapoport2016period} implies that given any lattice chain $(\La_i)_{i\in I}$,  Zariski locally on $\spec(S)$ there exists some $h\in \GL_d(S[\![t]\!])$ such that $\La_i=(hg)\cd\omega_i$ for all $i$. By further multiplying by an appropriate scalar matrix if necessary, we may assume $\det(h)=1$. Thus, the aforementioned functorial morphisms give an equivariant isomorphism $\SL_d/\cP_F\to \cF_{I,g}$ of ind-schemes over $\kappa$. 
\end{proof}
This description of the affine flag variety allows us to embed the special fibers of linked Grassmannians associated to simplices in $\fb_d$ into the corresponding affine flag varieties. Since these special linked Grassmannians serve as building blocks for more general ones, it eventually allows us to apply techniques from affine Schubert varieties to the study of the latter. To this end, we review some general definitions. 

Let $P,P'$ be two parahoric subgroups of $\SL_d(K)$ associated to two simplices $F,F'$ in $\cA$ respectively, and $\cP,\cP'$ be the corresponding parahoric group schemes.
Denote $L^+\cP'$ to be the affine group scheme such that $L^+\cP'(S)=\cP'(S[\![t]\!])$ for any $\kappa$-algebra $S$. By  definition, $L^+\cP'$ acts naturally from the left on $\cF_F$. 

\begin{defn}(\cite[Definition 2.5]{richarz2013schubert})\label{defn:Schubert}
For $w\in W_a=N/T(R)$, the $(P',P)$-\textit{Schubert cell} $C_w(P',P)$ is the reduced subscheme $L^+\cP'\cd n_w$ of $\cF_{F}$, where $n_w$ is any lift of $w$ to $N\subset \SL_d(K)$. The corresponding \textit{Schubert variety} %$S_w(P',P)$\com{is this symbol used later?} 
is the closure of $C_w(P',P)$ equipped with the reduced scheme structure. \end{defn}

%\begin{ex}
%It is well-known that the $(B,B)$-Schubert cell are affine spaces; more precisely, $C_w(B,B)\cong\AA^{\ell(w)}$ (\cite[\S 3]{gortz2001flatness}). For example, the morphism \[u:\AA^1\to \cF:x\mapsto \Bigg[\begin{pmatrix}[c|c]\begin{matrix}
%-x&1\\
%-1&0\\
%\end{matrix}&\\
%\hline
%&\id_{d-2}
%\end{pmatrix}\Bigg]\]
%where $[\cdot]$ denotes the coset represented by an element, identifies $\AA^1_{\kappa}$ with  $C_{s_1}(B,B)$ inside $\cF$. 
%\end{ex} 
Now, fix a subset $I=\{i_1\hh i_m\}$ of $[d-1]$. Let $g\in\wW$ and let $F$ be the simplex spanned by $\{[g\cd \omega_i]\}_{i\in I}$. Given $w\in W_a$, set $(x_i)_{i\in I}=(wg\cd \omega_i)_{i\in I}\in (\ZZ^d)^m$ as in Convention \ref{conv:alcove}. As a lattice chain, it corresponds to a point in $\cF_F$. By Proposition \ref{prop:chain}, given any other face $F'$ in $\cA$, the $\kappa$-points of $C_w(P_{F'},P_F)$ can be described concretely as the $P_{F'}$-orbit of the point $(x_i)_{i\in I}$, namely, the lattice chain $\Lambda_{i_1}\subset\hdots\subset \Lambda_{i_m}\subset t^{-1}\Lambda_{i_1}$ such that
\begin{equation}\label{eqn:orbit}
   \Lambda_{i_j}=\lag t^{-x_{i_j}(1)} e_1,t^{-x_{i_j}(2)} e_2\hh t^{-x_{i_j}(d)} e_d\rag.
\end{equation}

%Set $(x_i)_{i\in I}=(g\cd \omega_i)_{i\in I}\in (\ZZ^d)^m$ as in Convention \ref{conv:alcove}. For $w\in W_a$, the lattice chain $(y_i)_{i\in I}:=(w\cd x_i)_{i\in I}$ corresponds to a point in $\cF_F$. By Proposition \ref{prop:chain}, given any other face $F'$ in $\cA$, the $\kappa$-points of $C_w(P_{F'},P_F)$ can be described concretely as the $P_{F'}$-orbit of the lattice chain $\Lambda_{i_1}\subset\hdots\subset \Lambda_{i_m}\subset t^{-1}\Lambda_{i_1}$ such that
%\begin{equation}\label{eqn:orbit}
%   \Lambda_{i_j}=\lag t^{-y_{i_j}(1)} e_1,t^{-y_{i_j}(2)} e_2\hh t^{-y_{i_j}(d)} e_d\rag.
%\end{equation}\com{I think this fact may need more explanation or reference, or some words about the ``reference lattice chain" in Lemma \ref{lem:concrete_description}.}
%where $x_{i_j}=w\iota^r\cd (1^{i_j},0^{d-i_j})\in\ZZ^d$. %Following existing literature (\cite{gortz2001flatness}), we shall see that the choice of $r'$ facilitates us to apply the theory of affine flag variety to the study of linked Grassmannians. 

\subsection{Linked Grassmannians: the one-simplex case}\label{subsec:linked grass one simplex}%\opn{More definitions to be added...}
The content in this subsection serves as building blocks for our study of linked Grassmannians associated to more complicated lattice configurations involving multiple simplicies inside one apartment $\cA$ of $\fb_d$. Much of the main idea is taken from \cite{gortz2001flatness} and  \cite{kottwitz2000minuscule}, but we adapt the statements thereof to our context and add detail where necessary.  

We first recall the definition of the linked Grassmannian functor as given in \cite{he2023degenerations}. 

\begin{defn}\label{defn:linked_grass}
Let $\G$ be a finite convex configuration of vertices in $\fb_d$ indexed by some set $\cI$ and fix a set of representatives $\{L_i\}_{i\in\cI}$. Define the \textit{linked Grassmannian functor}, $\mathcal{LG}_{r}(\Gamma)$, to be the functor on $R$-schemes $T$ whose $T$-valued points are collections of rank-$r$ locally-free sheaves $\{E_i\}_{i\in\cI}$, where $E_i$ is a sub-bundle of $L_i\otimes \mathcal O_T$, such that for every possible inclusion $\pi^k\colon L_a\hookrightarrow L_b$ ($k\in\ZZ$), the induced morphism $L_a\otimes \mathcal O_T\rightarrow L_b\otimes \mathcal O_T$ maps $E_a$ into $E_b$. 
\end{defn}
The functor $\mathcal{LG}_{r}(\Gamma)$ is represented by a scheme $LG_{r}(\Gamma)$ projective over $R$ which is independent of the choice of the representatives $L_i$ (\cite[Proposition 2.3]{he2023degenerations}). Plainly, the generic fiber of $LG_{r}(\Gamma)$ is the usual Grassmannian over $K$.

\begin{nt}
    The special fiber of $LG_r(\Gamma)$ will be denoted by $LG_r(\Gamma)_0$.
\end{nt}

The first case to consider is the lattice configuration $\Omega=\{[\omega_0]\hh[\omega_{d-1}]\}$ corresponding to the standard alcove. Based on their definitions, when making the same assumptions on the fraction and residue fields, the linked Grassmannian $LG_r(\Omega)$ agrees with the standard local model of Shimura varieties $\mathbf M^{\text{loc}}$ studied by G\"ortz in \cite{gortz2001flatness}. We may thus recall some fundamental results from \textit{loc.cit.} directly.

In \cite[\S 4]{gortz2001flatness}, G\"ortz defined an open covering for $LG_r(\Omega)$, which is compatible with the stratification of its special fiber into a disjoint union of some affine Schubert cells. The elements in this covering are indexed by $\mu$\textit{-admissible alcoves}, which we shall identify as arrays of integer vectors: 

\begin{defn2}\cite[Definition 4.2]{gortz2001flatness}\label{defn:adm1}
Let $\mu=(1^r,0^{d-r})\in\ZZ^d$ and $g\in \widetilde W$. We say that $\mathbf{x}=g\cd \omega\in(\ZZ^d)^d$ is $\mu$-\textit{admissible}, if one of the following two equivalent conditions hold:
\begin{enumerate}
    \item $g\preceq \sigma\cd \mu$ for some $\sigma\in\SS_d$, where $\sigma\cd \mu$ is thought of as an element in the subgroup of translations of $\wW$, and $\preceq$ is the Bruhat order on $\widetilde{W}$. (More concretely, this means $g\iota^{-r}\preceq (\id,\sigma\cd \mu) \iota^{-r}$ in $W_a$.)
    \item $\omega_i\le g\cd \omega_i\le \omega_i+(1,1\hh 1)$ and $\sum(g\cd \omega_i)-\sum(\omega_i)=r$, for $i=0\hh d-1$. 
\end{enumerate}
The set of all $\mu$-admissible $\mathbf x$'s is denoted by $\textbf{adm}_{r,d}$.
\end{defn2}
The equivalence of the two conditions is the content of Theorem 3.5 in \cite{kottwitz2000minuscule}.

The notion of $\mu$-admissible alcoves naturally extends to the notion of $\mu$-admissible faces. Similar to their prototypes, these combinatorial objects also serve the purpose of indexing strata of the special fiber, as well as elements of a natural open covering of linked Grassmannians $LG_r(\Gamma)$, where $\Gamma=\{[\omega_i]\mid i\in I\}$ is a non-empty proper subset of $\Omega$ corresponding to a face $F$ of the standard alcove. Here $I$ is a subset of $[d-1]$ as before. 
In this case, we denote $W_I$ for the Iwahori-Weyl group $W_{F}$.
\begin{defn2}\label{defn:adm2}
Let $\mu=(1^r,0^{d-r})\in\ZZ^d$. Given $g\in\wW$, then $\mathbf{x}=g\cd \omega_I=(g\cd \omega_i)_{i\in I}\in(\ZZ^d)^{|I|}$ is $\mu$\textit{-admissible} if one of the following two equivalent conditions hold:
\begin{enumerate}
    \item $\omega_i\le g\cd \omega_i\le \omega_i+(1,1\hh 1)$, and $\sum(g\cd \omega_i)-\sum(\omega_i)=r$, for $i\in I$;
    \item $W_Ig W_I\preceq W_I(\id,\sg\cd \mu)W_{I}$ for some $\sigma\in\SS_d$, where $\preceq $ is the partial order on $W_I\backslash\ti{W}/W_I$ induced by the Bruhat order on $\ti{W}$ (Remark \ref{rem:order_double_coset}).
\end{enumerate} 
\end{defn2}
The first part in (1) was originally called \textit{minuscule} in (\cite[9.2.1]{kottwitz2000minuscule}). The equivalence of the two conditions is the content of Theorem 9.6(3) in \cite{kottwitz2000minuscule}. Moreover, every admissible face $g\omega_I\in (\ZZ^d)^{|I|}$ is the restriction of some admissible alcove $g\omega$ (\cite[Proposition 9.3]{kottwitz2000minuscule}). Analogous to $\textbf{adm}_{r,d}$, we have
\begin{defn} \label{defn:adm2'}
The set
$\textbf{adm}^I_{r,d}=\{g W_{I}\mid g\in \textbf{adm}_{r,d}\}$ is the collection of $\mu$-admissible integer-vector arrays in $(\mathbb Z^d)^{|I|}$.  
\end{defn}
%The following result is useful, which tells us that 
%\begin{lem}
%The $\mu$-admissible faces of type-$I$ 1-1 correspond to elements in $\textbf{adm}^I_{r,d}$ via $g\omega_I\mapsto gW_I$. 
%\end{lem}
%\begin{proof}
%This follows directly from \cite[Proposition 9.3]{kottwitz2000minuscule}.
%\end{proof}
When defining the stratification of the special fiber and a natural open covering for $LG_r(\Gamma)$, we shall first stratify the special fiber into a union of parahoric Schubert cells, and then define the open covering correspondingly. This is slightly different from the situation presented in \cite{gortz2001flatness}, where Iwahori Schubert cells are considered. In doing so, each stratum of the special fiber can still be described as a rank locus in the context of quiver representations, which builds a natural connection between the two approaches. That being said, we need to introduce one more definition, which is motivated by the fact that $(P_{F'},P_F)$-Schubert cells are indexed by double cosets in $W_{F'}\backslash W_a/W_F$ (\cite[Proof of Lemma 1.6]{richarz2013schubert}). 

%More specifically, we choose to stratify the special fiber into a union of $(P_F,P_F)$-affine Schubert cells, %, where $P$ is the parahoric subgroup associated to the face $F$ spanned by elements of $\Gamma$, 
%$instead of $(B,P_F)$-cells as in the stratification presented in \textit{loc.cit.}. 

\begin{defn}\label{defn:equiv1} Define the following equivalence relation on $\textbf{adm}^I_{r,d}$:
    \[h_1W_I\sim h_2W_I \text{ if and only if } W_Ih_1W_I=W_Ih_2W_I\] 
Given any $\mathbf{x}\in \mathbf{adm}^I_{r,d}$, denote $C_{\mathbf{x}}$ to be the equivalence class of $\mathbf x$ with respect to this relation.   
\end{defn}
\begin{rem}\label{rem:equiv} Let $F$ be the face in $\cA$ represented by $\iota^{r}\cd\omega_I$, where $\iota:=((12\hdots d),(1,0^{d-1}))\in \wW$ (Notation \ref{nt:tau2}). 
    Note that $hW_I=(h\iota^{-r})\cd \iota^r\cd W_I$ and $\iota^r W_I\iota^{-r}=W_F$. Thus, if two admissible faces $h_1W_I$ and $h_2W_I$ are equivalent in the given sense, we have $W_I(h_1\iota^{-r})W_F=W_Ih_1W_I\iota^{-r}=W_Ih_2W_I\iota^{-r}=W_I(h_2\iota^{-r})W_F$. Therefore, $C_{h_1\iota^{-r}}(P_{\omega_I},P_F)=C_{h_2\iota^{-r}}(P_{\omega_I},P_F)$.
\end{rem}
We can now state/generalize the related results in the literature in accordance with our setup. To avoid notational confusion, we treat $\omega_i$ as an integer vector in the following theorem. 

\begin{thm}\label{thm:open_cover1} Let $L_i$ be the lattice identified with $\omega_i$ as in Convention~\ref{conv:alcove} and $\Gamma$ be the configuration $\{[L_i]\mid i\in I\}$, where $I=\{i_1\hh i_m\}$ is a subset of $[d-1]$. Given $\mathbf x,\mathbf y\in \mathbf{adm}^I_{r,d}$, we write $\mathbf{x}=(x_i)_{i\in I}=(g\cd\omega_i)_{i\in I}\in(\ZZ^d)^{|I|}$ and $\mathbf{y}=(h\cd\omega_i)_{i\in I}$, where $g,h\in \textbf{adm}_{r,d}$. Set $F$ to be the face in $\fb_d$ represented by $\iota^r\cd\omega_I$. The following hold for $LG_r(\Gamma)$:
\begin{enumerate}
    \item The special fiber $LG_r(\Gamma)_0$ embeds into the partial affine flag variety $\cF_F=\SL_d/\cP_F$. It can be stratified into a union of $(P_{\omega_I},P_{F})$-Schubert cells of the form $S_{\mathbf{x}}:=C_{(g\iota^{-r})}(P_{\omega_I},P_{F})$ for $\mathbf x\in \mathbf{adm}^I_{r,d}$.
    \item We have
    $S_{\mathbf{x}}\subseteq S_{\mathbf{y}}^c$ if and only if $g\iota^{-r}\preceq h\iota^{-r}$, where $\preceq $ is the Bruhat order on $W_I\backslash W_a/W_F$.
    
    \item %Given $\mathbf{x}=(x_i)_{i\in I}=(g\cd\omega_i)_{i\in I}\in(\ZZ^d)^{|I|}$, where $g\in \textbf{adm}_{r,d}$. 
    Define $U'_{\mathbf{x}}$ to be the open subscheme of $LG_r(\Gamma)$ whose $S$-points ($S$ being any $R$-algebra) parametrize objects $(E_i)_{i\in I}$ such that $E_i\subset L_{i,S}$ and for every $i\in I$, the quotient lattice $L_{i,S}/E_i$ is generated by the images of $\pi^{-\omega_i(j)}e_j\in L_{i,S}$ for which $x_i(j)=\omega_i(j)$. Define $U_{\mathbf{x}}=\bigcup_{\mathbf{x'}\in C_{\mathbf x}}U'_{\mathbf x'}$. 
    
    Then, the $U_{\mathbf{x}}$ together form an open cover of $LG_r(\Gamma)$. Moreover, $U_\mathbf{x}$ contains $S_{\mathbf{x}}$. 
    \item Each aforementioned $S_{\mathbf{x}}$ is the locus in $LG_r(\Gamma)_0$ whose $\kappa$-points $V=(V_i)_{i\in I}$ satisfy the following rank condition for each $i\neq j\in I$:  
    \[r_{i,j}:=\rk(V_i\to V_j)=\sum_k (x_i(k)-\omega_i(k)),\] 
    where the sum runs through $k=j+1,j+2\hh i+d$ (resp. $k=j+1,j+2\hh i$) if $i<j$ (resp. $i>j$). Here we set $x_i(j+d)=x_i(j)$ and  $\omega_i(j+d)=\omega_i(j)$. In particular, $r_{i,j}$ only depends on $C_{\mathbf x}$.

    Conversely, $S_{\mathbf x}$ is uniquely determined by the rank vector $(r_{i,j})_{i,j}$. 
    \item When $\Gamma=\Omega$ is the standard alcove, the irreducible components of the special fiber are indexed by the distinct permutations of the integer vector $(1^r,0^{d-r})$. The corresponding $U_{\mathbf{x}}$ is an affine space of dimension $r(d-r)$, i.e. $U_{\mathbf{x}}\cong \AA^{r(d-r)}$.
    \item The dimension of $S_\mathbf{x}$ is $\mathfrak l(\prescript{}{\omega_I}g^F)$, where $\prescript{}{\omega_I}g^F$ is defined in Definition \ref{defn:min_ele}.
\end{enumerate}
\end{thm}
\begin{proof} (1) The statement was mentioned in \cite[\S 4.2]{gortz2001flatness}, with detail provided in the case $I=[d-1]$. Let $S$ be any $\kappa$-algebra, and for $i\in I$ consider the quotient map $p_i:L_{i,S}\to L_{i,S}/tL_{i,S}$. An $S$-point $(E_i)_{i\in I}$ of the special fiber can be identified with the lattice chain $p_{i_1}^{-1}(E_{i_1})\subset p_{i_2}^{-1}(E_{i_2})\subset\hdots \subset p_{i_m}^{-1}(E_{i_m})$ with $\det t^{-1}p_{i_1}^{-1}(E_{i_1})=t^{-i_1-r}S[\![t]\!]=\det(\iota^r\cdot L_{i_1})$. %where $(t^{-1}L'_{i_k})_{k=1}^m$ is given by $\iota^r\cd\omega_I$. 
After scaling each $p^{-1}_\bullet(E_\bullet)$ by $t^{-1}$, this identification induces an embedding of the special fiber into $\cF_F$ according to Proposition~\ref{prop:chain}. %Set $\La_{i_k}=t^{-1}p_{i_m}^{-1}(E_{i_k})$. Note that $t\La_{i_k}=p_{i_m}^{-1}(E_{i_k})$ satisfies the property $L_{i_k}\supseteq t\La_{i_k}\supseteq tL_{i_k}$.

\medskip
To see that the special fiber can be stratified into disjoint parahoric Schubert cells, recall that each $(P_{\omega_I},P_F)$-cell is a $P_{\omega_I}$-orbit of some fixed lattice chain $\Lambda_{i_1}\subset \Lambda_{i_2}\subset\hdots\subset \Lambda_{i_m}$ (see (\ref{eqn:orbit})).
Up to scaling by $t$, the lattice chain corresponds to a point in the linked Grassmannian as above if and only if $L_{i_k}\supseteq t\La_{i_k}\supseteq tL_{i_k}$ for each $1\leq k\leq m$.
This is equivalant to 
\[b\cd L_{i_k}=L_{i_k}\supseteq b\cd (t\La_{i_k})= t(b\cd \La_{i_k})\supseteq b\cd(tL_{i_k})=tL_{i_k},\] 
for any $b$ in $P_{\omega_I}$. Thus, either $\Lambda_{i_1}\subset \Lambda_{i_2}\subset\hdots\subset \Lambda_{i_m}$ corresponds to a point in the linked Grassmannian and hence the $P_{\omega_I}$-orbit lies in the linked Grassmannian, or the whole $P_{\omega_I}$-orbit is disjoint from the linked Grassmannian.
%the following commutative diagram obviously exists: 
%\begin{center}
%\begin{tikzcd}
%    b\cd L'_{i_k}=L'_{i_k}\ar[hookrightarrow]{r}& b\cd L'_{i_{k+1}}=L'_{i_{k+1}}\\
%    b\cd \Lambda_{i_k}\ar[hookrightarrow]{u}\ar[hookrightarrow]{r}& b\cd \Lambda_{i_{k+1}}\ar[hookrightarrow]{u}
%\end{tikzcd}    
%\end{center}\

\medskip
For (2), the statement is a special case of Proposition 2.8(i) in \cite{richarz2013schubert}.

\medskip
For (3), recall that each $(P_{\omega_I},P_F)$-cell, $C_w(P_{\omega_I},P_F)$, is set-theoretically a union of $(B,P_F)$-cells since the double coset $W_IwW_F=\bigcup_{u\in W_I}(uw)W_F$. Each $(B,P_F)$-cell $C_{uw}(B,P_F)$ is inside the linked Grassmannian if and only if $uw\cd \iota^r\cd \omega_I$ is admissible. Thought of as an array of integer vectors, $(x'_i)_{i\in I}:=(uw\iota^r)\cd\omega_I\in (\ZZ^d)^{|I|}$ gives a representative $\mathbf{x'}$ of $C_{\mathbf{x}}$, such that any object $(E_i)_{i\in I}$ in $C_{uw}(B,P_F)$ satisfies the property that $L_{i,S}/E_i$ can be generated by images of $t^{-\omega_i(j)}e_j\in L_{i,S}$, for which $x'_i(j)=\omega_i(j)$ (see also \cite[\S 4.3]{gortz2001flatness}). It is then not hard to see that the definition of $U_{\mathbf{x}}$ is taken precisely as such so that it contains $S_{\mathbf{x}}$.  As a result, the family $(U_{\mathbf x})_{\mathbf x}$ covers the special fiber $LG_r(\Gamma)_0$, hence covers $LG_r(\Gamma)$ by the properness of $LG_r(\Gamma)$.

\medskip
For (4), in the case $\Gamma=\Omega$, the result is the content of \cite[Proposition 2.6]{gortz2010supersingular}. In general, we use again the fact that a $(P_{\omega_I},P_F)$-cell, $S_{\mathbf x}$, may be described as the $P_{\omega_I}$-orbit of some lattice chain $\Lambda_{i_1}\subset \Lambda_{i_1}\subset\hdots \subset \Lambda_{i_m}$ (see (1) and~(\ref{eqn:orbit})). Such reference lattice chain is not unique in general. However, note that for any $i<j\in I$, and for any $b$ in $P_{\omega_I}$ we have $b\cd t\Lambda_i\cap tL_j=b\cd t\Lambda_i\cap b\cd (tL_j)=b(t\Lambda_i\cap tL_j)$, so the rank of the linear map $b\cd t\Lambda_i/t(b\cd t\Lambda_i)\to b\cd L_j/t(b\cd L_j)=L_j/tL_j$, in other words, $r_{i.j}$, is invariant over points in the cell. Thus it suffices to look at the case $b=\id$. Let us choose the lattice chain in (\ref{eqn:orbit}) to begin with:
\begin{enumerate}
    \item[(i)] When $j+1\le k \le d$ or $1\le k\le i$, since $\omega_i(k)=\omega_j(k)$, we get $t^{-x_i(k)+1}e_k\notin tL_j$ if $x_i(k)=\omega_i(k)+1$, and $t^{-x_i(k)+1}e_k\in tL_j$ otherwise. 
    \item[(ii)] When $i+1\le k<j$, since $\omega_j(k)=\omega_i(k)+1$, $t^{-\omega_i(k)}e_k\in tL_j$, and hence $t^{-x_i(k)}e_k\in tL_j$.  
\end{enumerate}
This proves the rank formula in the case $i<j$. Up to a circular permutation, it can be verified similarly in the case $j>i$.  

Conversely, let $(x_i)_{i\in I}$ and $(y_i)_{i\in I}$ be integer arrays in $\mathbf{adm}^I_{r,d}$. %corresponding to the cosets $h_1W_F$ and $h_2W_F$ resp., i.e. $x_i=h_1\iota^r\cd \omega_i$, $y_i=h_2\iota^r\cd \omega_i$. 
Suppose $(x_i)_{i\in I}$ and $(y_i)_{i\in I}$  generate the same rank vector. This is tantamount to the condition that for every $i\in I$ and $1\leq k\leq m$, we have  
\begin{equation}\label{eqn:rk}
\sum_{j=i_k+1}^{i_{k+1}}\big(x_{i}(j)-\omega_{i}(j)\big)=\sum_{j=i_k+1}^{i_{k+1}}\big(y_{i}(j)-\omega_{i}(j)\big),    
\end{equation}
where $i_{m+1}$ is interpreted as $i_1+d$.
Equivalently, $\sum_{j=i_k+1}^{i_{k+1}}x_{i}(j)=\sum_{j=i_k+1}^{i_{k+1}}y_{i}(j)$, for $1\leq k\leq m$. 
%Since $\sum(x_i)=\sum(y_i)$ for $i\in I$, there exists $s\in W_a$ such that $s\cd(x_i)_{i\in I}=(y_i)_{i\in I}$. 
We need to check that there is an $s\in W_I$ such that $s\cd(x_i)_{i\in I}=(y_i)_{i\in I}$. To do this, we split each vector $x_i$ and $y_i$ into $m$ parts, each has indices from $i_k+1$ to $i_{k+1}$ for $1\leq k\leq m$, and construct $s$ for each part independently. 

We first consider $1\le k\le m-1$. Since $\omega_i=(1^i,0^{d-i})$ for each $i\in I$, by replacing $(\omega_i)_{i\in I}$ with a different lattice quiver representing the same configuration if necessary, we may assume 
\[(\omega_i(i_k+1)\hh \omega_i(i_{k+1}))=(0\hh 0),\qquad \fa i\in I.\] 
By admissibility, $x_i(j),y_i(j)=0$ or $1$ for all $i_k+1\leq j\leq i_{k+1}$. Moreover, each of the sequences $(x_{i_1}(j))_{j=i_k+1}^{i_{k+1}}\le\hdots\le (x_{i_m}(j))_{j=i_k+1}^{i_{k+1}}$ can be obtained from $(0,0\hh 0)$ by adding an integer vector in $\{0,1\}^{i_{k+1}-i_k}$. Recording the positions of the 1's in each $(x_{i_\bullet}(j))_{j=i_k+1}^{i_{k+1}}$ we get a (suitably ordered) sub-sequence of indices in $\{i_k+1,i_k+2\hh i_{k+1}\}$, and similarly for $(y_{i_1}(j))_{j=i_k+1}^{i_{k+1}}\le\hdots\le (y_{i_m}(j))_{j=i_k+1}^{i_{k+1}}$. By Equation~(\ref{eqn:rk}), the two sub-sequences have the same length. Let $\sigma_k$ be the permutation  of $i_{k}+1\hh i_{k+1}$ switching the aforementioned two sub-sequences. Then $\sigma_k\cd (x_i(j))_{j=i_k+1}^{i_{k+1}}=(y_i(j))_{j=i_k+1}^{i_{k+1}}$ for all $i\in I$. Plainly, $\sigma_k$ fixes all $\omega_i$ with $i\in I$.

%This condition is satisfied by $(y_{i_1}(j))_{j=i_k+1}^{i_{k+1}}\le\hdots\le (y_{i_m}(j))_{j=i_k+1}^{i_{k+1}}$ as well. Thus, the definitions of $(x_{i_1}(j))_{j=i_k+1}^{i_{k+1}}\le\hdots\le (x_{i_m}(j))_{j=i_k+1}^{i_{k+1}}$ and $(y_{i_1}(j))_{j=i_k+1}^{i_{k+1}}\le\hdots\le (y_{i_m}(j))_{j=i_k+1}^{i_{k+1}}$reduce to the choice of two sub-sequences of equal lengths (due to Equation \ref{eqn:rk}) of indices in $\{i_k+1,i_k+2\hh i_{k+1}\}$. Therefore, we conclude that there is a permutation $\sigma_k$ of $i_{k}+1\hh i_{k+1}$, switching the aforementioned two sub-sequences, such that $\sigma_k\cd (x_i(j))_{j=i_k+1}^{i_{k+1}}=(y_i(j))_{j=i_k+1}^{i_{k+1}}$ for all $i\in I$. 

Now assume $k=m$. If $i_1=0$, then $$(\omega_i(i_m+1),\hdots,\omega_i(d),\omega_i(1),\hdots,\omega_i(i_1))=(\omega_i(i_m+1),\hdots,\omega_i(d))=(0\hh0)$$ for all $i\in I$. Hence the same argument as above shows that there is a permutation $\sigma_m$ of $i_m+1\hh d$ such that $\sigma_m\cd (x_i(j))_{j=i_m+1}^{d}=(y_i(j))_{j=i_m+1}^{d}$ for all $i\in I$. If $i_1>0$, then $$(\omega_i(i_m+1),\hdots,\omega_i(d),\omega_i(1),\hdots,\omega_i(i_1))=(0\hh 0,1\hh 1).$$ 
One can still get a sub-sequence $\underline v$ (resp. $\underline v'$) of indices in $\{i_m+1\hh d,1\hh i_{1}\}$ by recording those $j$ such that $x_i(j)-\omega_i(j)=1$ (resp. $y_i(j)-\omega_i(j)=1$) for some $i\in I$. Again, $\underline v$ and $\underline v'$ has equal length. Then, there exists $\sigma_m=(\sigma,\underline a)\in W_a$, where $\sg$ is a permutation of $i_m+1\hh d,1\hh i_{1}$ sending $\underline v$ to $\underline v'$, and $\underline a$ is some integer vector whose entries belong to $\{-1,0,1\}$, such that $\sigma_m$ fixes each $\omega_i$ with $i\in I$ and $\sigma_m\cd (x_i(j))_{j=i_m+1}^{i_1}=(y_i(j))_{j=i_m+1}^{i_1}$ for all $i\in I$.

%One can show that there is some \textit{affine} permutation $\sigma_m$ such that $\sigma_m\cd (x_i(j))_{j=i_m+1}^{d}=(y_i(j))_{j=i_m+1}^{d}$ for all $i\in I$, and $\sigma_m$ fixes $(\omega_i(i_m+1),\hdots,\omega_i(d),\omega_i(1),\hdots,\omega_i(i_1))$. Concretely, the definitions of \[(x_{i_1}(i_m+1)\hh x_{i_1}(d),x_{i_1}(1)\hh x_{i_1}(i_{1}))\le\hdots\le(x_{i_m}(i_m+1)\hh x_{i_m}(d),x_{i_m}(1)\hh x_{i_m}(i_1))\] and \[(y_{i_1}(i_m+1)\hh y_{i_1(d)},y_{i_1}(1)\hh y_{i_1}(i_{1}))\le\hdots\le(y_{i_m}(i_m+1)\hh y_{i_m}(d),y_{i_m}(1)\hh y_{i_m}(i_1))\] reduce to the choice of two sub-sequences $\vec{v}$, $\vec{v}'$ of equal lengths (due to Equation \ref{eqn:rk}) of indices in $\{i_m+1\hh d,1\hh i_{1}\}$. 

%$\sigma_m$ can be written as a product of affine permutations of the form $(\sigma,\vec v_\sigma)$ or $(\sigma,\vec v_\sigma)^{-1}$, where $\sigma=(a,b)$ is a permutation such that $1\leq a\leq i_1$ and $i_m+1\leq b\leq d$, and $\vec v_\sigma$ is the transition vector with $1$ on the $a$-th position and $-1$ on the $b$-th position, and zero otherwise.
%\footnote{For example, $s_n\cd (2,0\hh 0,0)=(1,0\hh 0,1)\in\ZZ^d$, yet $s_n\cd (1,1\hh 0,0)=(1,1\hh 0,0)\in\ZZ^d$.} 

Eventually, by multiplying those block-wise (affine) permutations $\sigma_1,\hdots,\sigma_m$ together, we arrive at some element $s\in W_I$ such that $s\cd (x_i)_{i\in I}=(y_i)_{i\in I}$.

\medskip
(5) is given by Proposition 4.5(iii) and Proposition 4.14 in \cite{gortz2001flatness}. 

\medskip
(6) is the content of Proposition 2.8(ii) in \cite{richarz2013schubert}. 
\end{proof}

\begin{rem}
    In part (1) of Theorem~\ref{thm:open_cover1}, when $\Gamma=\Omega$ is the standard alcove, the special fiber of $LG_r(\Gamma)$ is actually stratified into a union of $(B,B)$-cells according to \cite{gortz2001flatness}.  This is essentially due to the fact that the affine flag variety $\cF_F$ is equivariantly isomorphic to $\cF_\Omega$. This is not true in general when $\Gamma$ is just a face of the standard alcove.
\end{rem}

\subsection{Quiver Grassmannians associated to lattice configurations.}\label{subsec:quiver grass preliminary} We recall some notions and properties given in \cite{he2023degenerations}. Let $\Gamma=\{[L_i]\}_{i\in \cI}\subset \fb_d$ be a convex lattice configuration. 

\begin{nt}\label{nt:morphisms in a lattice configuration}
Fix representatives $\{L_i\}_i$. For each pair $i,j\in \cI$ let $n_{i,j}$ be the minimal integer such that $\pi^{n_{i,j}}L_i\subset L_j$. Denote by $F_{i,j}$ the map from $L_i$ to $L_j$ induced by multiplying with $\pi^{n_{i,j}}$. Denote $\ov L_i=L_i/\pi L_i$ (resp. $\ov L_j=L_j/\pi L_j$) and $f_{i,j}\colon \ov L_i\rightarrow \ov L_j$ the map induced by $F_{i,j}$. We say that $f_{i,j}$ factors through $\overline L_k$ (or simply factors through $k$) if $f_{i,j}=f_{k,j}\circ f_{i,k}$.
\end{nt}
Regarding the morphisms $f_{i,j}$ we have the following lemma:
\begin{lem}\cite[Lemma 2.11]{he2023degenerations}\label{lem:convex hull of two points}
Let $\{[L_0],...,[L_a]\}\subset \Gamma$ be the convex hull of $[L_0]$ and $[L_a]$ such that $[L_i]$ is adjacent to $[L_{i+1}]$. Then
\begin{enumerate}
\item $f_{0,a}=f_{a-1,a}\circ\cdots\circ f_{0,1}$ and $f_{a,0}= f_{1,0}\circ\cdots\circ f_{a,a-1}$, in other words, $n_{0,a}=\sum_{i=0}^{a-1}n_{i,i+1}$ and $n_{a,0}=\sum_{i=0}^{a-1}n_{i+1,i}$;
    \item $\ker f_{i,i+1}=\mathrm{Im}f_{i+1,i}$ and $\ker f_ {i+1,i}=\mathrm{Im}f_{i,i+1}$ for $0\leq i\leq a-1$;
    \item $\ker f_{i,i+1}\cap\ker f_{i,i-1}=0$ for $1\leq i\leq a-1$. In particular, $\dim\im f_{0,a}=\dim\im f_{0,1}$.
\end{enumerate}
\end{lem}

We now recall the quiver $Q(\Gamma)$ associated to $\Gamma$. Firstly, let $Q(\Gamma)'$ be the quiver such that $Q(\Gamma)'_0=\cI$ and $Q(\Gamma)'_1=\{(i,j)\in \cI^2\mid i\neq j\}$, where $(i,j)$ represents an arrow with source $i$ and target $j$. 
    
\begin{defn}\label{defn:quiver associated to lattice}
The quiver $Q(\Gamma)$ is obtained from $Q(\Gamma)'$ by removing all arrows $(i,j)$ such that there exists a path $\ell$ in $Q(\Gamma)'$ with length at least 2 such that $f_\ell=f_{i,j}$, where $f_\ell$ is defined in Notation~\ref{nt:quivers}.
\end{defn}

According to Lemma~\ref{lem:convex hull of two points} (3), if $[L_i],[L_j]$ are not adjacent, then $f_{i,j}$ factors through all $\overline L_k$ for $[L_k]$ in the convex hull of $[L_i]$ and $[L_j]$. Hence, adjacent vertices of $Q(\Gamma)$ always correspond to adjacent lattice classes in $\Gamma$, but not the other way around.

\begin{ex}\label{ex:quiver of a simplex}
    Let $\Gamma=\{[L_i]\}_{1\leq i\leq n}$ be a simplex in $\fb_d$, such that $L_1\subset L_2\subset\cdots\subset L_n\subset\pi^{-1}L_1$. Then $Q(\Gamma)$ is the cycle $(1,2,\dotsc,n,1)$.
\end{ex}

\begin{prop}\cite[Proposition 2.12]{he2023degenerations}\label{prop:ambient representation1}
We have
\begin{enumerate}
\item For any $i,j\in \cI$, there exists a path $\ell$ in $Q(\Gamma)$ such that $f_\ell=f_{i,j}$;
    \item for any two paths $\ell_1$, $\ell_2$ in $Q(\Gamma)'$ such that $s(\ell_1)=s(\ell_2)$, $t(\ell_1)=t(\ell_2)$ and $f_{\ell_i}\neq 0$ for both $i=1,2$, we have $f_{\ell_1}=f_{\ell_2}$;
    \item for any cycle $\ell$ in $Q(\Gamma)'$, we have $f_{\ell}=0$;
%     \item \com{where is this part used?}for any non-trivial path $\ell$ in $Q(\Gamma)$ with $f_{\ell}\neq 0$, there exists $\ell'$ such that $t(\ell')=s(\ell)$ and $\ker(f_{\ell})=\im(f_{\ell'})$.
%     \item if $\ell\in Q(\Gamma)_1$, then $[L_{s(\ell)}]$ is adjacent to $[L_{t(\ell)}]$. 
\end{enumerate}
\end{prop}

\begin{nt}\label{nt:ambient representation2}
     Let $M_\Gamma=(\overline L_i)_i$ be the representation of $Q(\Gamma)$ with linear maps $(f_{i,j})_{i,j}$ for each arrow $(i,j)$ in $Q(\Gamma)_1$.
\end{nt}

\begin{prop}\cite[Proposition 2.10, 2.15]{he2023degenerations}\label{prop:ambient representation}
    \begin{enumerate}
        \item The special fiber $LG_r(\Gamma)_0$ is set-theoretically identified with the quiver Grassmannian $\Gr(\mathbf r,M_\Gamma)$, where $\mathbf r=(r,\dotsc,r)$ as in Notation~\ref{nt:dimension vector number};
        \item $\Gamma$ is contained in an apartment if and only if $M_\Gamma$ is a projective representation. 
    \end{enumerate}
\end{prop}

In practice, we will sometimes identify $\Gamma$ with the set of vertices of $Q(\Gamma)$, and write $\Gamma=\{[L_v]\}_v$ instead of $\Gamma=\{[L_i]\}_i$ for naturality. This way, by $v\in\Gamma$ we mean that $v$ is a vertex of $Q(\Gamma)$. In addition, the maps $f_{i,j}$ will be replaced with $f_{u,v}$.

We end this section with one more notation for future use.

\begin{nt}\label{nt:rank vector}
    Given a sub-representation $M=(V_v)_{v\in \Gamma}$ of $M_\Gamma$. 
    We want to record all dimensions of $V_v$ as well as ranks of maps $f_\ell$ along $M$, where $\ell$ is any path in $Q(\Gamma)$ such that $f_\ell$ is non-trivial on $M_\Gamma$. According to Proposition~\ref{prop:ambient representation1}, this is equivalent to 
  associating a tuple $$\Phi(M)=(\Phi_{u,v}(M))_{u,v}:=(\dim f_{u,v}(V_u))_{u,v}\in\mathbb Z_{\geq 0}^{|\Gamma|^2}$$ to $M$. The $\Phi(M)$ will be referred to as the \textit{rank vector} of $M$. This gives a map $\Phi\colon \Gr({\underline x},M_\Gamma)\rightarrow\mathbb Z^{|\Gamma|^2}_{\geq 0}$, where $\underline x$ is any dimension vector.
    The fibers of $\Phi$ induces a stratification of $\Gr(\underline x,M_\Gamma)$ where each point in the same stratum has the same rank vector. We call this stratification the \textit{rank-vector stratification}. The partial order ``$\leq$" on $\mathbb Z^{|\Gamma|^2}_{\geq 0}$ induces a partial order on the rank-vector strata, which we call the \textit{rank-vector (partial) order} and still denote by ``$\leq$". 
\end{nt}

\begin{rem}\label{rem:rank n orbit strata}
    Let $\Gamma$ be a simplex in $\fb_d$, and suppose two representations in $\Gr(\mathbf r,M_\Gamma)$ have the same rank vectors. Then by Theorem~\ref{thm:open_cover1} (4), their corresponding lattice chains in the affine flag variety differ by the action of an element in $P_{\omega_I}$, which then reduces to an automorphism of $M_\Gamma$ that transfers one of the quiver representation to the other. This implies that the rank-vector stratification of $\Gr(\mathbf r,M_\Gamma)=LG_r(\Gamma)_0$ is equivalent to the stratification induced by orbits of the automorphism group of $M_\Gamma$.
\end{rem}

\section{Standard coverings of linked Grassmannians}
In this section, we will define the standard open covering for the linked Grassmannian associated to a convex lattice configuration contained in an apartment. It is the mathematical object that captures the topological structure of a linked Grassmannian (in particular, the topology of the special fiber), thus shedding light on its geometry.
%\medskip
%Recall that to each convex lattice configuration one can associate a connected quiver with relation $(Q(\Gamma),J_{\Gamma})$ (\cite[Definition 2.7]{he2023degenerations}). 
\begin{nt}\label{nt:complex} Let $\Gamma$ be a convex lattice configuration in $\fb_d$. 
Denote $\Theta(\Gamma)$ to be the complex spanned by the vertices in $\Gamma$ inside $\fb_d$. For a face $\tau$ of $\Theta(\Gamma)$, let $V(\tau)$ be the set of vertices of $\tau$.
\end{nt}

Throughout this section, we assume that $\Gamma$ is contained in the
%Since $\Theta(\Gamma)$ is a sub-simplicial complex of the 
standard apartment $\cA$ as stated above. Then every lattice in $\Gamma$ is %vertex is represented by some $R$-lattice in $K^d$ 
of the form $L=\lag \pi^{c_{1}}e_1\hh\pi^{c_{d}}e_d \rag$, where $e_1\hh e_d$ is the standard basis of $K^d$. We also fix a lattice quiver representing $\Gamma$
%elements in the vertex set $\cI=V(\Gamma)$ 
and identify it with an array of integer vectors in $(\ZZ^d)^{|\Gamma|}$. 
We shall single out the maximal simplices of $\Theta(\Gamma)$, which are themselves faces in $\fb_d$. In particular, each of them has a well-defined type:  
\begin{defn}\cite[\S 9.1]{kottwitz2000minuscule}
Let $I$ a non-empty subset of $\{0\hh d-1\}$. A tuple $v=(v_i)_{i\in I}\subset (\ZZ^d)^{|I|}$ is a face of \textit{type} $I$ if 
\begin{enumerate}
    \item $v_i(k)\le v_j(k)\le v_i(k)+1$, for any $ i\le j\in I$ and $k=1\hh d$;
    \item $\sum (v_i)-\sum (v_j)=i-j$, $\fa i,j\in I$. 
\end{enumerate}
\end{defn}
We shall reserve the notation $\omega_I$ for the standard face of type $I$, that is, $\omega_I=(\omega_i)_{i\in I}$, and denote $W_I\subset\ti{W}$ to be the stabilizer of $\omega_I$ as before. Following \cite[\S 9]{kottwitz2000minuscule}, we identify the set of type-$I$ faces with the coset space $\ti{W}/W_I$. In other words, each of such faces will be presented by a coset $gW_I$. It is not hard to see that the stabilizer group of $g\cd\omega_I$ in $\ti{W}$ is $W_{g\omega_I}:=gW_Ig^{-1}$. 

\begin{defn}
Given a face $g\cd \omega_I$ of type $I$, define 
\[\mathbf{adm}^{g,I}_{r,d}=\{h\cd W_{g\omega_I}\mid g\cd\omega_i(k)\le h\cd (g\cd\omega_i)(k)\le g\cd\omega_i(k)+1,\hspace{.1in} \Sigma(hg\cd\omega_i)-\Sigma(g\cd\omega_i)=r,\hspace{.1in} \fa i\in I,\ 1\leq k\leq d\}.\]
\end{defn}
%\begin{rem}
%This is a direct analogue to the fact that when $\Gamma=\{[\omega_i]\}_{i\in I}$, $\mu$-admissible faces $v=(v_i)_{i\in I}$ are precisely the ones satisfying the conditions that $\omega_i\le v_i\le \omega_i+\mathbf{1}$ and that $\Sigma(v_i)-\Sigma(\omega_i)=r$, for all $i\in I$  (\cite[Theorem 9.6]{kottwitz2000minuscule}).
%\end{rem}
This definition is an analogue of Definition~\ref{defn:adm2} and~\ref{defn:adm2'}. 
Hereafter, we will also refer to elements in $\mathbf{adm}^{g,I}_{r,d}$ as \textit{admissible} should no confusion arise. 
Similar to the consideration for cases where $\Gamma$ is a face in the previous section, in order to connect the quiver-representation-theoretic perspective with techniques from affine flag varieties, we shall stratify the special fiber of a general linked Grassmannian into (fibered products of) some parahoric Schubert cells, which are parametrized by double cosets in $W_{g\omega_I}\bs W_a/W_{(\iota^rg)\omega_I}$. For this purpose, we generalize Definition \ref{defn:equiv1} into the following equivalence relation on $\mathbf{adm}^{g,I}_{r,d}$:
\begin{equation}\label{eqn:equiv}
    h_1W_{g\omega_I}\sim h_2W_{g\omega_I} \text{ if and only if } W_{g\omega_I}h_1W_{g\omega_I}=W_{g\omega_I}h_2W_{g\omega_I}.
\end{equation}
\begin{nt}\label{nt:equiv}
Given any $\mathbf x\in \mathbf{adm}^{g,I}_{r,d}$, denote $C_{\mathbf x}$ to be the equivalence class of $\mathbf x$ inside $\mathbf{adm}^{g,I}_{r,d}$. 
\end{nt}
\begin{rem}\label{rem:equiv2}
Set $F$ to be the face whose vertex set is represented by $(\iota^rg)\cd\omega_I$. The same argument as in Remark \ref{rem:equiv} shows that for equivalent admissible faces, the corresponding $(P_{g\omega_I},P_{F})$-cells are identical.  
\end{rem}
%Since $\mathbf{adm}^{I}_{r,d}$ comes equipped with a partial order given by the Bruhat order on $\ti{W}/W_I$ (Definition \ref{defn:adm2}), so do all $\mathbf{adm}^{g,I}_{r,d}$. Abusing language, we shall refer to these partial orders as \textit{Bruhat orders} on $\mathbf{adm}^{g,I}_{r,d}$.
%\begin{ex}
%When $I=\ZZ$, a type-$I$ facet is just an alcove and $W_I$ is trivial and it is natural to present an alcove by an element in $\ti{W}$: $g\in\ti{W}$ stands for the alcove $(g\cd\omega_i)$. In general, we will do so for a facet of arbitrary type, and denote $[g]$ for the corresponding coset in the relevant coset space. 
%\end{ex}

\begin{nt}\label{nt:simplicial decomp and lattice quiver}
    Suppose %$\Gamma$ is a convex configuration represented by a lattice quiver $\mathfrak{L}:=\{L_i\}_{i=1}^n$, and
$\Gamma=\bigcup_{j\in\cJ} V(\tau_j)$ is a convex configuration represented by a lattice quiver $\{g_j\cd \omega_{I_j}\}_{j\in \cJ}$, where $\{\tau_j\mid j\in \cJ\}$ is the set of maximal simplices of $\Theta(\Gamma)$, $I_j$ is the type of $\tau_j$, and $g_j\cd \omega_{I_j}$ represents $\tau_j$. 
We will also write $\tau_j=(\tau_{j,i})_{i\in V(\tau_j)}\subset (\ZZ^d)^{|I_j|}$.
As an index set, we may assume $\cJ$ is totally ordered. Let $\Gamma_j=V(\tau_j)$ and $\Gamma_{j,j'}=V(\tau_j\cap\tau_{j'})$ be the set of vertices. 
Denote $W_j$ and $W_{j,j'}$ to be the subgroups of $W_a$ of pointwise stabilizers of $\Gamma_j$ and $\Gamma_{j,j'}$ respectively. 
\end{nt}

Consider the linked Grassmannians $LG_r(\Gamma)$ and $LG_r(\Gamma_j)$. %\com{$\Delta(\Gamma)$ not defined. $\cI$ is the set of maximal ones, how is it totally ordered?}) 
Using the associated moduli functors of linked Grassmannians, it is not hard to see that the following diagram is an equalizer: 
\begin{equation}\label{eqn:equalizer}
LG_r(\Gamma)\to X:=\begin{tikzcd}
 \prod_{j\in\cJ} LG_r(\Gamma_j) \ar[yshift=2pt]{r}{\aaa} \ar[yshift=-2pt]{r}[swap]{\beta} & \prod_{\substack{j< j'\\\Gamma_{j,j'}\neq\emptyset}}LG_r(\Gamma_{j,j'})=:Y
\end{tikzcd}    
\end{equation}
where the first arrow is the product of the natural forgetful morphisms, $\aaa$ is induced by the forgetful morphisms $LG_r(\Gamma_{j})\to LG_r(\Gamma_{j,j'})$, and $\beta$ by all $LG_r(\Gamma_{j'})\to LG_r(\Gamma_{j,j'})$. In other words, $LG_r(\Gamma)\cong Y\times_{Y\times Y}X$ (\cite[\href{https://stacks.math.columbia.edu/tag/01KM}{Tag 01KM}]{stacks-project}). This, together with the content in the previous section, motivates the following definition:

\begin{defn}\label{defn:adm_type}
%Suppose %$\Gamma$ is a convex configuration represented by a lattice quiver $\mathfrak{L}:=\{L_i\}_{i=1}^n$, and
%$\Gamma=\bigcup_{j\in\cJ} V(\sg_j)$ is a convex configuration represented by a given lattice quiver, where $\{\sg_j\mid j\in \cJ\}$ is the set of maximal simplices of $\Theta(\Gamma)$. Suppose further that $\sg_j\in(\ZZ^d)^{|I_j|}$ is of type $I_j$. 
%Denote $W_j$ and $W_{j,j'}$ to be the subgroups of $W_a$ of pointwise stabilizers of $V(\sg_j)$ and $V(\sg_j\cap\sg_{j'})$ respectively. Let $g_j\cd \omega_{I_j}$ be the restriction of the given lattice quiver to $\sg_j$. 
Let $\Gamma$ be as in Notation~\ref{nt:simplicial decomp and lattice quiver}. 
A $\Gamma$-\textit{admissible collection} $\mathbf{x}$ of \textit{size} $r$ is a tuple of cosets $(h_jW_j)_{j\in\cJ}$ such that %$(h_jg_j)\cd\omega_{I_j}\in \mathbf{adm}^{g_j,I_j}_{r,d}$
$h_jW_j\in \mathbf{adm}^{g_j,I_j}_{r,d}$
and that $W_{j,j'}h_jW_{j,j'}=W_{j,j'}h_{j'}W_{j,j'}$ for all $j,j'\in \cJ$. %We also call $r$ the \textit{size} of the admissible collection.It has \textit{size}-$r$ if $\sum (g_j\cd \sg_{j,k})-\sum(\sg_{j,k})=r$ for all $k\in I_j$. 
\end{defn}
\begin{rem}Suppose $W_{j}hW_j=W_jh'W_j$, i.e. $h=ah'b$, where $a,b\in W_j$. Since $W_j\subset W_{j,j'}$, \[W_{j,j'}hW_{j,j'}=W_{j,j'}(ah'b)W_{j,j'}=W_{j,j'}h'W_{j,j'}.\] 
In particular, the double coset $W_{j,j'}hW_{j,j'}$ is independent of the choice of a representative of $W_jhW_j$, and the notion of $\Gamma$-admissible collection is well-defined.
\end{rem}
\begin{nt}
Hereafter, with Notation~\ref{nt:simplicial decomp and lattice quiver}, given $j,j'\in \cJ$, we shall denote $P_j$ to be the parahoric subgroup associated to $\tau_j$, and $P_{j,j'}$ to $\tau_j\cap\tau_{j'}$. Let $F_j$ be the face represented by $\iota^rg_j\cd\omega_{I_j}$ and $F_{j,j'}$ represented by $(\iota^rg_{j,j'})\cd \omega_{I_{j,j'}}$, where $I_{j,j'}$ is the type of $\tau_j\cap\tau_{j'}$ and 
%$\tau_j\cap \tau_{j'}$ is of type $I_{j,j'}$ and let $F_{j,j'}$ be the face represented by $(\iota^rg_{j,j'})\cd \omega_{I_{j,j'}}$, where 
$g_{j,j'}\cd \omega_{I_{j,j'}}$ is the restriction of the lattice quiver in Notation~\ref{nt:simplicial decomp and lattice quiver} to $\tau_j\cap \tau_{j'}$.
\end{nt}
By Remark \ref{rem:equiv2}, if $(h_jW_j)_{j\in\cJ}$ is admissible, one then has $C_{h_j\iota^{-r}}(P_{j,j'},P_{F_{j,j'}})=C_{h_{j'}\iota^{-r}}(P_{j,j'},P_{F_{j,j'}})$. This is the main motivation of our definition, which will lead to a stratification of the special fiber of the linked Grassmannian in terms of fibered products of parahoric Schubert cells. More precisely,
by Proposition \ref{prop:chain}, one can embed the special fiber $LG_r(\Gamma_j)_0$ into the affine flag variety $\cF_{F_j}:=SL_d/\cP_{F_j}$.  %where $F_j$ is the face represented by $\iota^rg_j\cd\omega_{I_j}$ and $\Gamma_j=V(\tau_j)$. 
Hence, by Diagram (\ref{eqn:equalizer}), we get an embedding of  $LG_r(\Gamma)_0$ into a fibered product of affine flag varieties $\prod_{j\in\cJ}\cF_{F_j}$. Moreover, from the fundamental stratification of each $LG_r(\Gamma_j)_0$ given in Theorem \ref{thm:open_cover1} (1), we conclude that $LG_r(\Gamma)_0$ can be stratified into a union of fiber products of parahoric Schubert cells
\[S_{\mathbf{x}}:=\prod_{j\in \cJ} C_{h_j\iota^{-r}}(P_j,P_{F_j})\] 
where $\mathbf{x}=(h_jW_j)_{j\in \cJ}$ runs through all $\Gamma$-admissible collections, and each fibered product is fibered over admissible cells of common faces, i.e. $C_{h_j\iota^{-r}}(P_j,P_{F_j})$ and $C_{h_{j'}\iota^{-r}}(P_{j'},P_{F_{j'}})$ are fibered over $C_{h_j\iota^{-r}}(P_{j,j'},P_{F_{j,j'}})=C_{h_{j'}\iota^{-r}}(P_{j,j'},P_{F_{j,j'}})$. 
%where $g_j\cd\omega_{I_{j,j'}}$ represents $V(\tau_j\cap\tau_{j'})$ and $F_{j,j'}$ is the face whose vertex set is represented by $(\iota^rg_j)\cd \omega_{I_{j,j'}}=(\iota^rg_{j'})\cd \omega_{I_{j,j'}}$. 

The next lemma is an analogue of Theorem \ref{thm:open_cover1}(4), which compares Schubert cells inside different affine flag varieties by considering them as rank loci. Recall that all affine flag varieties in this paper are associated to parahoric subgroups of $\SL_d(\kappa(\!(t)\!))$, and are ind-schemes over $\kappa$ (see Section \ref{sec:aff_flag}). Therefore, in the next proof we may take $K=\kappa(\!(t)\!)$\footnote{Note that this does not change our setting for linked Grassmannian as it is only relevant to the embedding of the \textit{special fiber} of the linked Grassmannian into a fibered product of affine flag varieties.}, and  this is the only occasion in the current section where we impose such a restriction.  
\begin{lem}\label{lem:rank2}
The Schubert cell $C_{h_j\iota^{-r}}(P_{g_j\omega_{I_j}},P_{F_j})$ is a rank locus of $\kappa$-points in $LG_r(\Gamma_j)_0$ which has the same rank vector $r_{i,j}$ (see Theorem \ref{thm:open_cover1} (4)) as $C_{g^{-1}_jh_jg_j\iota^{-r}}(P_{\omega_{I_j}},P_{\iota^r\cd \omega_{I_j}})$ in $LG_r([\omega_{I_j}])_0$. 
\end{lem}
\begin{proof}
    First of all, the admissibility of the two cells are equivalent, as $g_j\omega_{I_j}\le (h_jg_j)\cd\omega_{I_j}\le g_j\omega_{I_j}+(1,1\hh 1)$ if and only if $\omega_{I_j}\le (g_j^{-1}h_jg_j)\cd\omega_{I_j}\le \omega_{I_j}+(1,1\hh 1)$.
    
    Second, recall that $P_{g_j\omega_{I_j}}=g_jP_{\omega_{I_j}}g^{-1}_j$ (\cite[Remark 2]{haines2008parahoric}), where abusing notation we identify $g_j$ with a lift of it to $\GL_d(\kappa(\!(t)\!))$.  
    Therefore, multiplying with $g_j^{-1}$ on the left induces an isomorphism from the $P_{g_j\omega_{I_j}}$-orbit of the lattice chain $(h_jg_j)\cd\omega_{I_j}$ to the $P_{\omega_{I_j}}$-orbit of the lattice chain $(g_j^{-1}h_jg_j)\cd\omega_{I_j}$, hence an isomorphism between the two Schubert cells.  The corresponding rank vectors are obviously the same.
    %the Schubert cell $C_{h_j\iota^{-r}}(P_{g_j\omega_{I_j}},P_{F_j})$, or equivalently, the $P_{g_j\omega_{I_j}}$-orbit of the lattice chain $(h_jg_j)\cd\omega_{I_j}$, to the $P_{\omega_{I_j}}$-orbit of the lattice chain $(g_j^{-1}h_jg_j)\cd\omega_{I_j}$, or equivalently the Schubert cell $C_{g^{-1}_jh_jg_j\iota^{-r}}(P_{\omega_{I_j}},P_{\iota^r\cd \omega_{I_j}})$.  
\end{proof}

The natural forgetful maps $LG_r(\Gamma_j)_0\to LG_r(\Gamma_{j,j'})_0$ and $LG_r(\Gamma_{j'})_0\to LG_r(\Gamma_{j,{j'}})_0$ clearly preserve the rank vectors in the following sense: the rank vector of the restriction of a lattice chain to the smaller lattice configuration is a sub-vector of the rank vector of the original lattice chain. This justifies our claim on the fibered products of Parahoric Schubert cells. 

\begin{defn}\label{defn:Bruhat_strata}
Let $\Gamma$ be a convex configuration, and $0<r<d$ be fixed integers. Denote $\mathbf{adm}_r(\Gamma)$ to be the set of all $\Gamma$-admissible collections of size-$r$. We shall refer to the above stratification in terms of $S_{\mathbf x}$ ($\mathbf x\in \mathbf{adm}_r(\Gamma)$) as the \textit{Bruhat stratification} of the special fiber.
%The Bruhat stratification induces a partial order, $\preceq_{\text{top}}$, on the index set $\mathbf{adm}_r(\Gamma)$: $\mathbf{x}\preceq_{\text{top}}\mathbf{y}$ if and only if $S_{\mathbf{x}}\subset S_{\mathbf{y}}^c$. We call it the \textit{topological (partial) order} on $\mathbf{adm}_r(\Gamma)$.\com{moved to Notation~\ref{nt:topological order}}
\end{defn}

We shall now construct an open cover of $LG_r(\Gamma)$, which is a natural generalization of the one given in Theorem \ref{thm:open_cover1} (3), where $\Gamma$ was a face in the standard alcove. Before doing so, note that the equivalence relations on $\mathbf{adm}^{g_j,I_j}_{r,d}$ ($j\in\cJ$) (Equation \ref{eqn:equiv}) induce an equivalence relation on $\mathbf{adm}_r(\Gamma)$ in an obvious way, and each distinct $\Gamma$-admissible collection $\mathbf{x}$ can be identified with a $\cJ$-tuple $(C_{\mathbf{x}_j})_{j\in \cJ}$, where each $C_{\mathbf{x}_j}$ is an equivalence classes of elements in $\mathbf{adm}^{g_j,I_j}_{r,d}$. Every $C_{\mathbf{x}_j}$ can be further identified with a finite set of elements in $(\ZZ^d)^{|I_j|}$.

\begin{nt}
Given a lattice quiver $(L_i)_{i\in \cI}=(\lag \pi^{c_{i,1}}e_1\hh \pi^{c_{i,d}}e_d\rag )_{i\in \cI}$ representing some convex configuration $\Gamma$, we denote $e^i_k$ for $\pi^{c_{i,k}}e_k$. 
\end{nt}
   
\begin{defn}\label{defn:stan_open}
Let $\Gamma=\bigcup_{j\in \cJ}V(\tau_j)$ be as in Notation~\ref{nt:simplicial decomp and lattice quiver}. We may also write $(L_i)_{i\in \cI}=(g_j\omega_{I_j})_{j\in\cJ}$ as the lattice quiver representing $\Gamma$. 
%convex configuration of lattice classes in $\cA$ represented by a lattice quiver $(L_i)_{i\in\cI}$, and $\cJ$ an index set for maximal simplices of $\Theta(\Gamma)$.
Let $\mathbf{x}=(h_jW_j)_{j\in\cJ}$ be any $\Gamma$-admissible collection of size $r$ and let $\mathbf{x}_j=(h_jg_j)\cd \omega_{I_j}\in(\ZZ^d)^{|I_j|}$. 
We define $U_{\mathbf{x}}$ to be the open sub-scheme of $LG_r(\Gamma)$ parametrizing objects $(E_i)_i$ such that for every $j\in\cJ$, there exists some $\mathbf{x'}=(\mathbf x'_i)_{i\in V(\tau_j)}\in C_{\mathbf{x}_j}$ such that $L_i/E_i$ is generated by the image of $\{e^i_k\mid \mathbf x' _i(k)=\tau_{j,i}(k)\}$ for every $i$ in $V(\tau_j)$. %(Here, $[e^i_t]$ is the image of $e^i_t$ in the quotient.) 
%\com{what is $F_j$ and $\Lambda_j$?} for which the following conditions hold: if $j$ is a vertex of a maximal simplex $\sg_i=(\sg^j_i)$, then $\La_j/F_j$ can be generated by the images $[e^t_j]$ of standard basis elements under the quotient map $\La_j\to \La_j/F_j$, for all $t$ such that $(g_i\cd\sg_i^j)(t)=\sg_i^j(t)$.
\end{defn}
The next proposition justifies our nomenclature:
\begin{prop}\label{prop:std_open} Let $\Gamma=\{L_i\}_i$ be a convex lattice configuration  in the standard apartment $\cA$.  Then
$\mathfrak{U}_{\Gamma}:=\{U_\mathbf{x}\mid \mathbf{x}\in \mathbf{adm}_r(\Gamma)\}$ is an open cover of $LG_r(\Gamma)$. Moreover, $S_{\mathbf x}$ is contained in $U_{\mathbf x}$.
\end{prop}
\begin{proof}
When $\Gamma$ is a face of the standard alcove, the proposition is given in Theorem \ref{thm:open_cover1} (3). Similarly, suppose $\Gamma$ consists of vertices of a simplex $[g\omega_I]$. Then by replacing $W_I$, $F$ and $\mathbf{adm}^{I}_{r,d}$ in \textit{loc.cit.} by $W_{g\omega_I}$, the face $F'=(\iota^rg)\cd\omega_I$ and $\mathbf{adm}^{g,I}_{r,d}$ respectively, and similarly replacing $B$ and $P_F$ by appropriate Iwahori subgroups $B'$ and $P_{F'}$ respectively, the argument remains valid. 
%Indeed, each $S_{\mathbf{x}}$ is a $(P_{g\omega_I},P_{\ti{F}})$-Schubert cell, or equivalently a union of some $(B',P_{\ti{F}})$-Schubert cells: $\bigcup C_{uw}(B',P_{\ti{F}})$, where the union runs over $u\in W_{g\omega_I}$ such that $(uw\iota^rg)\cd\omega_I\in \mathbf{adm}^{g,I}_{r,d}$. Thought of as an array of integer vectors, $(\mathbf x'_i)_{i\in I}:=(uw\iota^rg)\cd\omega_I\in (\ZZ^d)^{|I|}$ gives an element of $C_{\mathbf{x}}$. The objects $(E_i)_{i\in I}$ parametrized by $C_{uw}(B',P_{\ti{F}})$ satisfy the property that elements in $\{[e^i_t]\mid \mathbf x'_i(t)=\sg_i(t)\}$ generate the quotient objects $L_i/E_i$. In particular, $C_{uw}(B',P_{\ti{F}})$ is contained in $U_{\mathbf x}$, and hence $S_{\mathbf x}$ is contained in $U_{\mathbf{x}}$, for every admissible $\mathbf{x}$. 
The general case $\Gamma=\bigcup_{j\in \cJ}V(\tau_j)$ follows directly from the one-simplex case and the construction of $S_\mathbf x$ and $U_\mathbf x$. 
%as in Notation~\ref{nt:simplicial decomp and lattice quiver}. Given $\mathbf{x}=(C_{\mathbf{x}_j})\in \mathbf{adm}_r(\Gamma)$, by the definition of $S_{\mathbf x}$ and the discussion after Definition \ref{defn:Bruhat_strata}, for every $j\in\cJ$ there exists some $\mathbf{x'}=(\mathbf x'_i)_{i\in V(\tau_j)}\in C_{\mathbf{x}_j}$ such that elements in $\{[e^i_k]\mid \mathbf x'_{i}(k)=\tau_{j,i}(k)\}$ generate the quotient objects $L_i/E_i$, for all $i\in V(\tau_j)$. The definition of $U_{\mathbf x}$ then guarantees that it contains $S_{\mathbf x}$, and consequently $\mathfrak{U}_{\Gamma}$ is an open covering of $LG_r(\Gamma)$. 
\end{proof}

%Second, given an object $(F_i)$, in general there are more than one choice for an index set $I'$ of size $d-r$ such that $\lag e^t_j\mid t\in I'\rag\to \La_j/F_j$ is surjective; if vertex $j$ belongs to two maximal simplices $\sg$, $\sg'$, it is natural that by applying the definition to $\sg$, $\sg'$ resp., one obtains different sets of generators for $\La_j/F_j$. This merely means that the image of $(F_i)$ in $LG_r(V(\sg)\cap V(\sg'))$ under the natural projection lies in the intersection of two open sets of $LG_r(V(\sg)\cap V(\sg'))$. %Third\com{maybe this third point can be omitted}, $U_x$ is clearly open in $LG_r(\Gamma)$, since it is a finite intersection of unions of open subsets, each of which is the complement of the degeneracy locus of a morphism of locally-free sheaves inside $LG_r(\Gamma)$. 

\begin{defn}\label{defn:main}
Let $\Gamma$ be a convex configuration in $\cA$, and $0<r<d$ be fixed integers. We shall call the $\mathfrak{U}_{\Gamma}$ in Proposition~\ref{prop:std_open} the \textit{standard open cover} of $LG_r(\Gamma)$.
\end{defn}
The index set $\mathbf{adm}_r(\Gamma)$ comes equipped with some natural partial order:
\begin{defn}\label{defn:order1}
With all notation as in Definition \ref{defn:adm_type}, the \textit{generalized Bruhat order} ``$\preceq$" on $\mathbf{adm}_r(\Gamma)$ is defined as follows: $$\mathbf{x}=(h_jW_j)_{j\in\cJ}\preceq \mathbf{y}=(h'_jW_j)_{j\in\cJ}$$ if and only if $W_j(h_j\iota^{-r})W_{F_j}\preceq_j W_j(h'_j\iota^{-r})W_{F_j}$ for all $j\in\cJ$, where $\preceq_j$ is the Bruhat order on $W_j\backslash W_a/W_{F_j}$ (Remark \ref{rem:order_double_coset}) and $F_j$ is the face whose vertices are represented by $(\iota^rg_j)\cd\omega_{I_j}$.  
\end{defn}

Recall from Notation~\ref{nt:topological order} that the Bruhat stratification of $LG_r(\Gamma)_0$ induces a topological order $\preceq_{\text{top}}$. When $\Theta(\Gamma)$ is a simplex, the generalized Bruhat order is just the usual Bruhat order among double cosets, and  agrees with the topological order by Theorem \ref{thm:open_cover1} (2). It also follows that $\mathbf{x}\preceq_{\text{top}}\mathbf{y}$ implies $\mathbf{x}\preceq\mathbf{y}$ even when  $\Theta(\Gamma)$ is not a simplex. However, it is not clear whether the two orders above coincide in general. Meanwhile, recall from Section~\ref{subsec:quiver grass preliminary} that we have the identification $LG_r(\Gamma)_0=\Gr(\mathbf r,M_\Gamma)$. To every point $[M]$ of $LG_r(\Gamma)$ corresponding to a representation $M$ we associated a rank vector $\Phi(M)$ as in Notation~\ref{nt:rank vector}. The following lemma shows that the Bruhat stratification is compatible with the rank-vector stratification.

\begin{lem}\label{lem:bruhat order and rank vector}
    The rank vector is independent of the choice of points in $S_{\mathbf x}$, and points in different $S_\mathbf x$'s have distinct rank vectors. Equivalently, the Bruhat stratification of $LG_r(\Gamma)_0$ agrees with the rank-vector stratification of $\Gr(\mathbf r,M_\Gamma)$.
    %let $\Phi(\mathbf x)$ denote the rank vector associated to any point of $S_{\mathbf x}$. Then $\mathbf x\preceq \mathbf x'$ if and only if $\Phi(\mathbf x)\leq \Phi(\mathbf x')$.
 \end{lem}
\begin{proof}
    Pick any point in $S_{\mathbf x}$, which corresponds to a representation $M=(V_v)_v$ of $Q(\Gamma)$. Given $u,v\in \Gamma$, let $w$ be the vertex in the convex hull of $u$ and $v$ that is adjacent to $u$. By Lemma~\ref{lem:convex hull of two points} (3), we see that $\dim f_{u,v}(V_u)=\dim f_{u,w}(V_u)$. Since $u$ and $w$ are adjacent, they must be contained in the same maximal simplex of $\Theta(\Gamma)$. As a result, $\Phi(M)$ is determined by all $\Phi_{u,v}(M)$'s where $u$ and $v$ are in the same simplex. The conclusion hence follows from Theorem~\ref{thm:open_cover1} (4) and Lemma~\ref{lem:rank2}.
\end{proof}

We will show that the generalized Bruhat order agrees with topological order in the following two cases: (1) when $\Gamma$ is weakly independent (\S \ref{sec:weakind}); (2) when $r=1$ (\S \ref{sec:r1}). According to Lemma~\ref{lem:bruhat order and rank vector}, this is equivalent to showing that both orders are equivalent to the rank-vector order.

\section{Locally Weakly Independent Configurations}\label{sec:weakind}
In this section, we introduce the notion of a locally weakly independent configuration $\Gamma$ in $\fb_d$, which is a simultaneous generalization of locally linearly independent configurations in \cite{he2023degenerations} and simplices in $\fb_d$. For such configurations, we prove that the quiver stratification of $\Gr(\mathbf r,M_\Gamma)$ agrees with the rank-vector stratification (hence coincides with the
Bruhat-stratification of $LG_r(\Gamma)_0$ by Lemma~\ref{lem:bruhat order and rank vector});  and the topological order agrees with rank-vector order, see Theorem~\ref{thm:stratification}. As a corollary, we show that for arbitrary lattice configurations in an apartment, the generalized Bruhat order is the same as the rank-vector order (Corollary~\ref{cor:bruhat order=rank order}), and, in particular, the same as the topological order for locally weakly independent lattice configurations.

The main ingredient is the decomposition (Theorem~\ref{thm:decomposition}) of every sub-representation of $M_\Gamma$ 
%(Notation~\ref{nt:ambient representation2}) 
into a sum of representations of $Q(\Gamma)$ of dimension at most $\mathbf 1$. As an application, we characterize all possible rank vectors of representations in $\Gr(\mathbf r,M_\Gamma)$ when $\Gamma$ is a simplex (Proposition~\ref{prop:simplex strata}), which provides an alternate perspective of the existing result in \cite[\S~2.5]{gortz2010supersingular}.
Finally, we investigate the global geometry of linked Grassmannians of locally weakly independent configurations in Theorem~\ref{thm:flatness}.

Denote $\Gamma=([L_v])_v$.  We will use $v$ to represent a lattice class $[L_v]$ in $\Gamma$, and thus identify $\Gamma$ with the set of vertices in $Q(\Gamma)$, if there is no confusion. 
For $u,v\in \Gamma$, let $\Conv(u,v)$ be the convex hull of $[L_u]$ and $[L_v]$ in $\fb_d$. 
Let $\overline L_u$ and $f_{u,v}\colon \overline L_u\rightarrow \overline L_v$ be as in Notation~\ref{nt:morphisms in a lattice configuration}.

\begin{defn}\label{defn:ess_lli}
A convex lattice configuration $\Gamma$ is \textit{locally weakly independent} at $v\in\Gamma$ if there is a subset $I_v\subset \Gamma$ such that 
\begin{enumerate}
    \item for each $u\neq v$, $f_{u,v}$ factors through $w$ for some $w\in I_v$; and 
    \item $\{\im f_{w,v}=\ker f_{v,w}\}_{w\in I_v}$ are linearly independent subspaces of $\overline L_v$.
\end{enumerate}
We say that $\Gamma$ is \textit{locally weakly independent} if it is so at all $v\in \Gamma$.
\end{defn}

\subsection{First properties}\label{subsec:weakly indep first properties} In this subsection, we prove the basic properties of locally weakly independent lattice configurations, which will be used in the sequel to investigate the representations of the corresponding quivers.

\begin{prop}\label{prop:weakly indep}
Suppose $\Gamma$ is locally weakly independent. Then 
\begin{enumerate}
     \item  the arrows of $Q(\Gamma)$ are exactly the ones of the form $(w\rightarrow v)$ where $v\in \Gamma$ and $w\in I_v$;
     \item for any non-repeating path $v_0,...,v_m$ in $Q(\Gamma)$, we have $f_{v_{m-1},v_m}\circ\cdots \circ f_{v_0,v_1}\neq 0$. Moreover, there is a unique non-repeating path in $Q(\Gamma)$ from $v_0$ to $v_m$;
    \item all cycles of $Q(\Gamma)$ are induced in a bijective way from maximal simplices of $\Gamma$;
        \item $Q(\Gamma)$ is a union of cycles where any two cycles  share at most one vertex.
\end{enumerate}
\end{prop}
\begin{proof}
(1) Follows directly from the definition.

(2) To reduce notation, we write $f_{i,j}=f_{v_i.v_j}$. For the first half of the statement, assume on the contrary that $f_{m-1,m}\circ\cdots \circ f_{0,1}=0$. We can take $m$ to be minimal and hence $m\geq 2$ and $f_{m-2,m-1}\circ\cdots \circ f_{0,1}\neq 0$. Then we have $f_{0,m-1}= f_{m-2,m-1}\circ\cdots \circ f_{0,1}=f_{m,m-1}\circ f_{0,m}$ by \cite[Proposition 2.15 (2)]{he2023degenerations}. %It follows that $f_{0,m}=f_{m-1,m}\circ f_{0,m-1}=f_{m+1,m}\circ f_{0,m+1}$.
As each $f_{i,j}$ is a composition of morphisms represented by arrows in $Q(\Gamma)$, we found, from the equality, two paths in $Q(\Gamma)$ from $v_0$ to $v_{m-1}$, one passing through $v_m$ and the other not. Moreover, if we denote the two distinct paths respectively by $v_0,u_1,...,u_l,v_{m-1}$ and $v_0,w_1,...,w_s,v_{m-1}$, then $f_{u_l,v_{m-1}}\circ f_{v_0,u_l}=f_{0,m-1}=f_{w_s,v_{m-1}}\circ f_{v_0,w_s}$. According to part (1), we have $u_l,w_s\in I_{v_{m-1}}$, hence $u_l=w_s$ by the locally weak independence of $\Gamma$ at $v_{m-1}$. It follows that $f_{u_{l-1},u_l}\circ f_{v_0,u_{l-1}}=f_{v_0,u_l}=f_{v_0,w_s}=f_{w_{s-1},w_s}\circ f_{v_0,w_{s-1}}$. Argue as above, we have $u_{l-1}=w_{s-1}$. Inductively, we have $l=s$ and $u_i=w_i$ for all $1\leq i\leq s$, which is absurd. 

For the second half, suppose there are two paths from $v_0$ to $v_m$. From what we just proved, the compositions of maps along these paths are both nonzero, hence equal to $f_{0,m}$. A similar argument as above then shows that the two paths are the same. 

(3) We first note that for two adjacent lattice classes $[L_u]$ and $[L_v]$, the map $f_{u,v}$ factors through a lattice class $[L]$ if and only if we can take representatives such that $L_u\subset L\subset L_v\subset \pi^{-1}L_u$. It follows directly that the maximal simplices in $\Gamma$ induce cycles in $Q(\Gamma)$.

Now suppose there is a cycle $v_0,v_1,...,v_n,v_0$ in $Q(\Gamma)$. We write $f_{i,j}=f_{v_i.v_j}$ as in (2), and $L_i=L_{v_i}$. Since $v_0$ and $v_n$ are adjacent, we may assume that $L_0\subset L_n\subset \pi^{-1}L_0$.
By (2), we have $f_{n-1,n}\circ\cdots\circ f_{0,1}\neq 0$. Then $f_{1,n}\circ f_{0,1}\neq 0$. It follows that we can choose representative of $L_1$ such that $L_0\subset L_1\subset L_n$. Similarly, we have $f_{2,n}\circ f_{1,2}\neq 0$, hence $L_1\subset L_2\subset L_n$. Argue recursively, we have $L_{v_0}= L_0\subset L_1\subset\cdots\subset L_n\subset \pi^{-1} L_0$ and $L_0,...,L_n$ form a simplex $\Delta$. 
On the other hand, by the construction of $Q(\Gamma)$, there is no $[L]\in \Gamma$ such that $L_i\subsetneq L\subsetneq L_{i+1}$ for some $i$, where we let $L_{n+1}=\pi^{-1}L_0$, because otherwise $f_{i,i+1}$ factors through $[L]$ and $(v_i,v_{i+1})$ is not an arrow in $Q(\Gamma)$. Hence $\Delta$ is maximal.

%Now assume that we can find an $m<n$ such that $f_{m-1,m}\circ\cdots \circ f_{0,1}\neq 0$ and $f_{m,m+1}\circ\cdots \circ f_{0,1}=0$. Then we have $f_{0,m}= f_{m-1,m}\circ\cdots \circ f_{0,1}=f_{m+1,m}\circ f_{0,m+1}$ by \cite[Proposition 2.15 (2)]{he2023degenerations}. %It follows that $f_{0,m}=f_{m-1,m}\circ f_{0,m-1}=f_{m+1,m}\circ f_{0,m+1}$.
%As each $f_{i,j}$ is a composition of morphisms represented by arrows in $Q(\Gamma)$, we found, from the equality, two paths in $Q(\Gamma)$ from $v_0$ to $v_m$, one passes through $v_{m+1}$ and the other not. Moreover, if we denote the two distinct paths respectively by $v_0,u_1,...,u_l,v_m$ and $v_0,w_1,...,w_s,v_m$, then $f_{u_l,v_m}\circ f_{v_0,u_l}=f_{0,m}=f_{w_s,v_m}\circ f_{v_0,w_1}$. According to part (1), we have $u_l,w_s\in I_{v_m}$, hence $u_l=w_s$ by the local weak independence of $\Gamma$ at $v_m$. It follows that $f_{u_{l-1},u_l}\circ f_{v_0,u_1}=f_{v_0,u_l}=f_{v_0,w_s}=f_{w_{s-1},w_s}\circ f_{v_0,w_{s-1}}$. Argue as above, we have $u_{l-1}=w_{s-1}$. Inductively, we have $l=s$ and $u_i=w_i$ for all $1\leq i\leq s$, which is absurd. 

(4) Of course, every vertex is contained in a cycle. Suppose there are two cycles in $Q(\Gamma)$ sharing two vertices $v_1$ and $v_2$, then there must be two distinct non-repeating paths either from $v_1$ to $v_2$ or from $v_2$ to $v_1$. This contradicts  (2).%Since both paths are contained in single simplices by part (2), the composition of maps along the two paths are non-zero and hence coincide. By the same argument in part (2) we know that the two paths are the same, which is a contradiction.
\end{proof}

\begin{rem}\label{rem:weakly indep and pre-linked}
It follows from Proposition~\ref{prop:weakly indep} (2) that $LG_r(\Gamma)_0$ is a pre-linked Grassmannian in the sense of \cite[Definition A.1.2]{Ohrk} when $\Gamma$ is locally weakly independent, see also Remark 2.13 and Remark 2.19 of \cite{he2023degenerations}. Moreover, it follows from \cite[Proposition A.2.2]{Ohrk} that $LG_r(\Gamma)_0$ is smooth at points represented by projective sub-representations of $M_\Gamma$.
\end{rem}

Plainly, a simplex in $\fb_d$ is a locally weakly independent configuration, and the associated quiver is just a cycle as in Example~\ref{ex:quiver of a simplex}. 
Following Proposition~\ref{prop:weakly indep}, one should consider a locally weakly independent configuration $\Gamma$ as a ``tree of simplices", and $Q(\Gamma)$ as a ``tree of cycles." As we already mentioned, locally linearly independent configurations are special examples of locally weakly independent configurations:

\begin{ex}\label{ex:locally linearly indep}
As in \cite{he2023degenerations},  a \textit{locally linearly independent} lattice configuration is a locally weakly independent configuration $\Gamma$ such that for each $v\in\Gamma$,  $I_v$ is the set of vertices $u\neq v\in \Gamma$ such that $L_u$ is adjacent to $L_v$. If $\Gamma$ is locally linearly independent, then all cycles in $Q(\Gamma)$ have length 2 and hence contain exactly two vertices. Moreover, if we replace all cycles with single edges connecting the two vertices, we get a tree. See Figure~\ref{fig:linearex} for an example. 
  \tikzset{every picture/.style={line width=0.5pt}} 
\begin{figure}[ht]
\begin{tikzpicture}[x=0.45pt,y=0.45pt,yscale=-1,xscale=1]
\import{./}{figlinearex.tex}
\end{tikzpicture}
\caption{The quiver associated to a locally linearly independent lattice configuration.}\label{fig:linearex}
\end{figure}
\end{ex}

\begin{nt}\label{nt:Ae}
    Let $\vec e$ be an arrow in $Q(\Gamma)$. As in \cite{he2023degenerations}, we denote by $A_{\vec e}$ the set of vertices $v$ such that the minimal path in $Q(\Gamma)$ from $v$ to $t(\vec e)$ passes through $s(\vec e)$; and by $A_{\cev e}$ the set of vertices such that the minimal path does not pass through $s(\vec e)$. Equivalently, $A_{\vec e}$ is the set of vertices $v$ such that $f_{v,t(\vec e)}$ factors through $s(\vec e)$ and $A_{\cev e}=\Gamma\backslash A_{\vec e}$.
\end{nt}

We include one more example to illustrate the notion of locally weakly independent configurations and Notation~\ref{nt:Ae}.

%By construction, every cycle in $Q(\Gamma)$ form a maximal simplex in $\Gamma$. We can associate a graph $G_\Gamma$ to $\Gamma$ as follows: the set of vertices of $G_\Gamma$ is the union of $\Gamma$ with the set of maximal simplices in $\Gamma$; the edges are of the form $(v,\Delta)$ where $\Delta$ is a maximal simplex in $\Gamma$ and $v\in\Delta$. See Example~\ref{fig:weakex}. 
%It follows from  Proposition~\ref{prop:weakly indep} that $G_\Gamma$ is actually a tree. For any edge $\mathfrak e=(v,\Delta)$ we denote by %$\Delta_\mathfrak v$ the simplex in $\Gamma$ corresponding to $\mathfrak v$ and $v_\mathfrak e$ the vertex in $\Gamma$ corresponding to $\mathfrak e$. If $\vec{\mathfrak e}$ is oriented, we denote
%$A_\mathfrak e$ the subset of $\Gamma$ which is the union of all $\Delta'$ such that the minimal path in $G_\Gamma$ from $[\Delta']$ to $v$ does not pass through $[\Delta]$; and $A_{-\mathfrak e}$ the union of all $\Delta'$ such that the minimal path passes through $[\Delta]$. 

%consisting of all $\Delta_\mathfrak v$s such that $\mathfrak v$ runs over all vertices of $G_\Gamma$ such that the minimal path in $G_\Gamma$ from $\mathfrak v$ to the tail of $\vec{\mathfrak e}$ does not pass through the head of $\mathfrak e$. Let $\overline \Gamma$ be the subset of $\Gamma$ consisting of vertices not of the form $v_\mathfrak e$.

\begin{ex} On Figure~\ref{fig:weakex} is an example of a locally weakly independent configuration $\Gamma$ in $K^5$. Here we use $(a_1,a_2,a_3,a_4,a_5)\in\mathbb Z^5$ to represent the lattice generated by $\pi^{-a_i}e_i$ for $1\leq i\leq 5$, as in Convention~\ref{conv:alcove}. Let $\vec e$ be the arrow from $(1,1,0,0,0)$ to $(0,0,0,0,0)$. Then $A_{\vec e}$ is the set of vertices in the blue box.
%is the oriented edge from the left 3-valent vertex to the right 3-valent vertex. The set consists of vertices in the blue circle is $\Gamma_{\vec{\mathfrak e}}$, and the red vertices are the convex hull of $(2,2,0,1)$ and $(1,0,2,2)$. 
Moreover, the red edges form the unique path in $Q(\Gamma)$ from $(0,0,0,2,0)$ to $(1,2,0,0,0)$, and their convex hull consists of three more vertices $(0,0,0,1,0)$, $(0,0,0,0,0)$ and $(1,1,0,0,0)$. Compare to Proposition~\ref{prop:weakly indep 2}.
\tikzset{every picture/.style={line width=0.5pt}} 
\begin{figure}[ht]
\begin{tikzpicture}[x=0.45pt,y=0.45pt,yscale=-1,xscale=1]
\import{./}{figweakex.tex}
\end{tikzpicture}
\caption{The quiver associated to a locally weakly independent lattice configuration.}\label{fig:weakex}
\end{figure}
\end{ex}

\begin{nt}\label{nt:cycle and simplex}
For a maximal simplex $\Delta$ in $\Gamma$, we denote by $\vec \Delta$ the corresponding cycle in $Q(\Gamma)$.
    For a non-repeating path $\ell$ and a cycle $\vec\Delta$ in $Q(\Gamma)$, it follows from Proposition~\ref{prop:weakly indep} that $\ell\cap\vec \Delta$ is still a path. We call the source and target of this intersection path  \textit{local extremal vertices} of $\ell$.
\end{nt}

The following proposition is a generalization of \cite[Lemma 2.18]{he2023degenerations}.

\begin{prop}\label{prop:weakly indep 2} Let $v$ and $v'$ be two elements in $\Gamma$. 
%For any $v,v'$ in $\Gamma$, suppose $v\in\Delta$ and $v\in\Delta'$. Let $\mathfrak e_1,\dotsc,\mathfrak e_n$ be the minimal path in $G_\Gamma$ connecting $[\Delta]$ and $[\Delta']$. Then
\begin{enumerate}
    \item $\Conv(v,v')$ consists of all vertices of $Q(\Gamma)$ which is a local extremal vertex of the minimal path from $v$ to $v'$;
    \item for any arrow $\vec e$ in $Q(\Gamma)$, both $A_{\vec e}$ and $A_{\cev e}$ are locally weakly independent, as well as $A_{\vec e}\cup\{t(\vec e)\}$ and $A_{\cev e}\cup\{s(\vec e)\}$.
\end{enumerate}
\end{prop}
\begin{proof}
(1) Let $\ell$ be the (unique) minimal path in $Q(\Gamma)$ connecting $v$ and $v'$. Suppose the vertices in $\Conv(v,v')$ are  $v=v_0,v_1,\dotsc,v_{n-1},v'=v_n$, where $v_i$ is adjacent to $v_{i+1}$. By Lemma~\ref{lem:convex hull of two points} (1), we have $f_{v,v'}=f_{v_{n-1},v_n}\circ\cdots \circ f_{v_0,v_1}$. Hence $\ell$ contains all $v_i$. On the other hand, for each cycle $\vec \Delta$ in $Q(\Gamma)$, since all vertices in $\Delta$ are adjacent, the oriented path $\ell\cap\vec \Delta$ contains at most two of the $v_i$. Moreover, by Proposition~\ref{prop:weakly indep}, $v_i$ and $v_{i+1}$ must be contained in the same cycle of $Q(\Gamma)$. It then follows that $\Conv(v,v')$ is the set of all local extremal vertices of $\ell$.

%Suppose the vertices in $\Gamma\subset V(G_\Gamma)$ that lies on the non-repeating path are $v=v_0,v_1,\dotsc,v_{n-1},v'=v_n$.
%By construction, $L_{v_i}$ is adjacent to $L_{v_{i+1}}$ for each $i$. Hence it suffices to show that each $v_i$ is contained in the convex hull of $v$ and $v'$. The unique non-repeating path $\ell$ in $Q(\Gamma)$ from $v$ to $v'$ passes through all vertices in their convex hull and all $v_i$. Therefore, fixing an $i$, there must be two adjacent vertices $u_1$ and $u_2$ in the convex hull such that the path in $Q(\Gamma)$ connecting them passes through $v_i$. It follows that $u_1,u_2$ and $v_i$ are contained in the same simplex $\Delta$. On the other hand, $v_i$ must be one of the extremal vertex of $\ell\cap \Delta$, so it must be one of $u_1$ and $u_2$.
%by the definition of $\mathfrak e_i$, there must be another simplex containing $v_{\mathfrak e_i}$ which has non-empty intersection with $\ell$. Hence $\mathfrak e_i=v_1$ or $v_2$.
(2) Follows directly from (1).
\end{proof}

We end this subsection with a lemma that will be used to analyze the smoothing property of limit linear series (see Theorem~\ref{thm:smoothing of lls} and Corollary~\ref{cor:conditions}). It is a natural generalization of  \cite[Lemma 3.3]{he2023degenerations} and the  proof is almost identical. Nevertheless, we include it here for the readers' convenience. 

\begin{lem}\label{lem:lifting limit linear series}
Let $\Gamma$ be locally weakly independent. Given a non-empty subset $I\subset \Gamma$ and $r$-dimensional vector spaces $V_v\subset \ov L_v$ for $v\in I$. Suppose for all $u\in \Gamma$, the vector space 
$$V'_u:=\{\epsilon\in \ov L_u|f_{u,v}(\epsilon)\in V_v\mathrm{\ for\ all\ }v\in I\}$$ has dimension at least $r$. Then there is an $\mathbf r$-dimensional sub-representation $M=(U_v)_v$ of $M_\Gamma$ such that $U_v=V_v$ for all $v\in I$.
\end{lem}
\begin{proof}
Note that $(V'_v)_v$ is a sub-representation of $M_\Gamma$ such that $\dim V'_v\geq r$ for each $v$ and $V'_v=V_v$ for all $v\in I$. We proceed by adjusting $(V'_v)_v$ until $\dim V'_v=r$ for all $v\in\Gamma$. In any non-trivial case, there is a pair of adjacent vertices $u_1$ and $u_2$ such that $\dim V'_{u_1}=r$ and $\dim V'_{u_2}>r$. Moreover, we may assume $u_1\in I_{u_2}$. Denote $W_{u_2}=\bigoplus_{u\in I_{u_2}}f_{u,u_2}(V'_u)$. This is the vector space generated by the images from all $V'_u$ for $u\neq u_2$.

We claim that $\dim W_{u_2}\leq r$. Indeed, by locally weak independence at $u_2$, we have an injection
$$W'_{u_2}:=\bigoplus_{u\in I_{u_2}\backslash \{u_1\}}f_{u,u_2}(V'_u)\xhookrightarrow{f_{u_2,u_1}} \ker f_{u_1,u_2}\cap V'_{u_1}\subset V'_{u_1}.$$
It follows that 
$$\dim W_{u_2}= \dim W'_{u_2}+\dim f_{u_1,u_2}(V'_{u_1})\leq \dim (\ker f_{u_1,u_2}\cap V'_{u_1})+ \dim f_{u_1,u_2}(V'_{u_1}) =r.$$

Now we can replace $V'_{u_2}$ with any $r$-dimensional subspace that contains $W_{u_2}$ while keeping the other $(V'_v)_{v\neq u_2}$. This gives a sub-representation $(U'_v)$ of $M_\Gamma$ such that $\dim U'_v\geq r$ and $\sum_v\dim U'_v<\sum_v\dim V'_v$. Since $u_2\not\in I$, we still have $U'_{v}=V_v$ for all $v\in I$. Hence, by induction on $\sum_v\dim V'_v$, we are done.
%Note that $W=(W_v)_v$ is a sub-representation of $M_\Gamma$ such that $\dim W_v\geq r$ for each $v$ and $W_v=V_v$ for all $v\in I$. We proceed by induction on $\sum_v\dim W_v$. The base case $\dim W_v=r$ for all $v\in\Gamma$ is trivial. In any non-trivial case, there is a pair of adjacent vertices $u_1$ and $u_2$ such that $\dim W_{u_1}=r$ and $\dim W_{u_2}>r$. Moreover, we may assume $u_1\in I_{u_2}$. Denote $\widetilde W_{u_2}=\oplus_{u\in I_{u_2}}f_{u,u_2}(W_u)$. This is the vector space generated by the images from all $W_u$ for $u\neq u_2$. We claim that $\dim\widetilde W_{u_2}\leq r$. 
%Indeed, by locally weak independence at $[L_{u_2}]$ we have an injection$$\widetilde W'_{u_2}:=\bigoplus_{u\in I_{u_2}\backslash \{u_1\}}f_{u,u_2}(W_u)\xhookrightarrow{f_{u_2,u_1}} \ker f_{u_1,u_2}\cap W_{u_1}\subset W_{u_1}.$$
%It follows that $$\dim\widetilde W_{u_2}= \dim\widetilde W'_{u_2}+\dim f_{ u_1,u_2}(W_{u_1})\leq \dim (\ker f_{u_1,u_2}\cap W_{u_1})+ \dim f_{u_1,u_2}(W_{u_1}) =r.$$
%Now we can replace $W_{u_2}$ with any $r$-dimensional subspace that contains $\widetilde W_{u_2}$ while keeping the other $W_v$'s. This gives a sub-representation $W'=(W'_v)$ of $M_\Gamma$ such that $\dim W'_v\geq r$ and $\sum_v\dim W'_v<\sum_v\dim W_v$. Since $u_2\not\in I$, we still have $W'_{v}=V_v$ for all $v\in I$. Hence by induction we are done.
\end{proof}

\subsection{Decomposition}\label{subsec:weakly indep decomp}
In this subsection, we prove that any sub-representation of $M_\Gamma$, where $\Gamma$ is locally weakly independent, can be decomposed as a direct sum of representations of $Q(\Gamma)$ that is generated by a single vector. A similar statement was proved for locally linearly independent configurations in \cite{he2023degenerations} by induction on $|\Gamma|$. In this paper, we give a straight construction of such decomposition. 
\begin{lem}\label{lem:decomposition generate} Suppose $(U_v)_v$ is a sub-representation of $M_\Gamma$ and we have a decomposition
$$U_v=V_v\oplus \sum_{u\neq v}f_{u,v}(U_u).$$
Then $(U_v)_v$ is generated by $(V_v)_v$.
\end{lem}
\begin{proof}
Let $(U'_v)_v$ be the quotient of $(U_v)_v$ by the sub-representation generated by $(V_v)_v$. Then it suffices to show that $(U'_v)_v$ is trivial. Denote by $\bar f_{u,v}\colon U'_u\rightarrow U'_v$ the induced map from $f_{u,v}$. By construction, we have $U'_v=\sum_{u\in \Gamma\backslash \{v\}} \bar f_{u,v}(U'_u)$. Suppose there is a  nonzero element $\ep_1\in V'_{v_1}$, then there exists $\ep_u\in U'_u$ such that $\ep_1=\sum_{u\in\Gamma\backslash \{v_1\}}\bar f_{u,v_1}(\ep_u)$. It follows that there is a $v_2\in \Gamma\backslash \{v_1\}$ and $\ep_2:=\ep_{v_2}\in U'_{v_2}$ such that $\bar f_{v_2,v_1}(\ep_2)\neq 0$. Replacing $\ep_2$ with $\ep_1$ and run the argument again we find $\ep_2=\sum_{u\in\Gamma\backslash \{v_2\}}\bar f_{u,v_2}(\ep'_u)$, hence
$0\neq\bar f_{v_2,v_1}(\ep_2)=\sum_{u\in\Gamma\backslash \{v_2\}}\bar f_{v_2,v_1}\circ\bar f_{u,v_2}(\ep'_u)$
and there is a $v_3\in \Gamma\backslash \{v_2\}$ and $\ep_3:=\ep'_{v_3}\in U'_{v_3}$ such that $\bar f_{v_2,v_1}\circ \bar f_{v_3,v_2}(\ep_3)\neq 0$. Inductively, we get an infinite sequence $\ep_1,\ep_2,\dotsc$ where $\ep_i\in U'_{v_i}$ and $v_i\neq v_{i-1}$ such that $\bar f_{v_2,v_1}\circ\cdots\circ \bar f_{v_i,v_{i-1}}(\ep_i)\neq 0$ for each $i$. This is impossible. 
\end{proof}

For $v\in \Gamma$, from Proposition~\ref{prop:weakly indep} we know that $I_v$ is naturally identified with the set of cycles in $Q(\Gamma)$ containing $v$. 
\begin{nt}
    For all $u\in I_v$, let $ \Delta_{v,u}$ be the
set of vertices of $Q(\Gamma)$ that lies in the unique cycle (or equivalently, maximal simplex of $\Gamma$) containing $u$ and $v$.
\end{nt}

\begin{lem}\label{lem:decompose independent} 
Suppose $\Gamma$ is locally weakly independent, and $(U_v)_v$ is a sub-representation of $M_\Gamma$. For each $v$, let $V_v$ be a subspace of $U_v$ such that $$V_v\cap \Big(\bigoplus_{u\in I_v}f_{u,v}(U_u)\Big)=0,$$
and for each $w\in  \Delta_{v,u}$, we have 
\begin{equation}\label{eq:decompose independent}\ker f_{v,w}|_{U_v}=\ker f_{v,w}|_{V_v}\oplus\ker f_{v,w}|_{f_{u,v}(U_u)}.\end{equation}
Then for each $v$, the spaces $(f_{u,v}(V_u))_{u\in\Gamma}$ are linearly independent. 
\end{lem}

\begin{proof} We proceed by induction on $|\Gamma|$. The base case $|\Gamma|=1$ is trivial. For each $u\in I_v$, let $\vec e_{u,v}$ be the arrow from $u$ to $v$. 
If $I_v$ contains at least two vertices, then by the locally weak independence of $\Gamma$ at $v$, the linear independence of $(f_{u,v}(V_u))_{u\in\Gamma}$ follows from the linear independence of 
$(f_{u,v}(V_u))_{u\in A_{\vec e_{w,v}}}$ for all $w\in I_v$. According to Proposition~\ref{prop:weakly indep 2} (2), each $A_{\vec e_{w,v}}\cup\{v\}$ is locally linearly independent. Hence the conclusion follows from the inductive hypothesis on $A_{\vec e_{w,v}}\cup\{v\}$. 

We may now assume that $I_v$ consists of a single vertex. Hence there is a unique cycle $\vec\Delta$ in $Q(\Gamma)$ that contains $v$. For each $u\in \Delta$, let $\vec e_u$ be the arrow in $\vec \Delta$ with target $t(\vec e_u)=u$. Then $\Gamma=\bigsqcup_{u\in\Delta}A_{\cev e_{u}}$, and $A_{\cev e_v}=\{v\}.$ Moreover, the source $s(\vec e_u)$ is contained in  $I_u$. See Figure~\ref{fig:linearindeplemma}.

%Write the cycle of $\Delta_{v,z}$ as $v_0=v,v_1,...,v_{n-1},
%v_n=z$, and let $\vec e_i$ be the arrow from $v_{i-1}$ to $v_i$ for $1\leq i\leq n$. See Figure~\ref{fig:linearindeplemma}. We have $v_i\in I_{v_{i+1}}$ for each $i$.

  \tikzset{every picture/.style={line width=0.5pt}} 
\begin{figure}[ht]
\begin{tikzpicture}[x=0.45pt,y=0.45pt,yscale=-1,xscale=1]
\import{./}{figlinearindeplemma.tex}
\end{tikzpicture}
\caption{A vertex $v$ of $\Gamma$ such that $I_v$ is a single vertex, as in the proof of Lemma~\ref{lem:decompose independent}.}\label{fig:linearindeplemma}
\end{figure}

Suppose there is a $J\subset\Gamma$ and, for each $w\in J$, an $\ep_w\in V_w$ such that $f_{w,v}(\ep_w)\neq 0$, and the vectors $(f_{w,v}(\ep_w))_{w\in J}$ are linearly dependent. Up to scaling, we can assume $\sum_{w\in 
J} f_{w,v}(\ep_w)=0$.  By construction, we have $v\not\in J$. 
For each $u\in\Delta\backslash\{v\}$, by the inductive hypothesis on each $A_{\cev e_u}$, we know that $\sum_{w\in J\cap A_{\cev e_u}} f_{w,u}(\ep_w)\neq 0$ whenever $J\cap A_{\cev e_u}\neq \emptyset$. Thus, replacing $\ep_u$ with $\sum_{w\in J\cap A_{\cev e_u}} f_{w,u}(\ep_w)$ and $V_u$ with 
$$V_u\oplus \bigoplus_{w\in I_u\backslash\{s(\vec e_u)\}}f_{w,u}(U_w)=V_u\oplus \sum_{w\in A_{\cev e_u}\backslash \{u\}}f_{w,u}(U_w),$$ we may assume that $J\subset \Delta$, and hence reduce to the case $\Gamma=\Delta$. One checks that equation (\ref{eq:decompose independent}) is preserved. 
%It follows from (\ref{eq:decompose independent}) that we have
%\begin{equation}\label{eq:decompose independent'}\ker f_{v,w}|_{U_v\oplus f_{u,v}(V_u)}=\ker f_{v,w}|_{U_v}\oplus\ker f_{v,w}|_{f_{u,v}(V_u)}.\end{equation}
%Furthermore, after removing irrelevant vertices in $\Delta_{v,z}$, we can assume $I=\Delta_{v,z}\backslash\{v\}$.

Let $z\in J$ be the vertex such that the path in $\vec\Delta$ from $z$ to $v$ does not contain other vertices in $J$. 
We then have $\ep_z+\sum_{u\in J\backslash\{z\}} f_{u,z}(\ep_u)\in \ker f_{z,v}|_{U_z}$. Denote $z'=s(\vec e_z)$.  According to equation (\ref{eq:decompose independent}), we have
\begin{equation*}\ker f_{z,v}|_{U_z}=\ker f_{z,v}|_{V_z}\oplus\ker f_{z,v}|_{f_{z',z}(U_{z'})}.\end{equation*}
Since $\ep_z\in V_z$ and $\sum_{u\in J\backslash\{z\}} f_{u,z}(\ep_u)\in f_{z',z}(U_{z'})$,
it follows that $\ep_z\in \ker f_{z,v}$. This contradicts the assumption that $f_{z,v}(\ep_z)\neq 0$ at the beginning. 
\end{proof}

\begin{thm}\label{thm:decomposition}
Suppose $\Gamma$ is locally weakly independent. Then each sub-representation of $M_\Gamma$ is decomposible into a direct sum of sub-representations, where each summand is generated by a single vector, hence has dimension at most $\mathbf 1$. Moreover, $M_\Gamma$ is projective, hence $\Gamma$ is contained in an apartment.
\end{thm}

It follows directly from the theorem that the irreducible sub-representations of $M_\Gamma$ are those generated by a single vector.

\begin{rem}\label{rem:irred rep}
Let us first describe the sub-representations $M$ of $M_\Gamma$ generated by a single vector. Suppose $M$ is generated by $\epsilon\in \overline L_v$. Then either $M$ is projective, or there exists a $w\in\Gamma$ such that $f_{v,w}(\epsilon)=0$. 
In the second case, there is a unique simplex $\Delta_{v,z}$, where $z\in I_v$, that contains a vertex $v'\neq v$ in $\mathrm{Conv}(v,w)$.
%the minimal path from $v$ to $w$ contains a vertex of  $\Delta_{v,z}$ other than $v$. Let $\ell$ be the restriction of this minimal path to $\Delta_{v,z}$.
%Then $t(\ell)$ is contained in the convex hull of $v$ and $w$ by Proposition~\ref{prop:weakly indep 2} (1).  
According to Lemma~\ref{lem:convex hull of two points}, we have $f_{v,v'}(\epsilon)=f_{v,w}(\epsilon)=0$. Write $\vec\Delta_{v,z}$ as $v=v_0,v_1,...,z=v_n,v_0$. Then there exists $1\leq m\le n$ such that $f_{v,v_{m-1}}\neq 0$ and $f_{v,v_m}=0$. See Figure~\ref{fig:repexample}.
\tikzset{every picture/.style={line width=0.5pt}} 
\begin{figure}[ht]
\begin{tikzpicture}[x=0.45pt,y=0.45pt,yscale=-1,xscale=1]
\import{./}{figrepexample.tex}
\end{tikzpicture}
\caption{A sub-representation of $M_\Gamma$ generated by $s\in \overline L_v$. Its support consists of vertices represented by solid dots.}\label{fig:repexample}
\end{figure}
For $1\leq i\leq n$ let $\vec e_i$ be the arrow in the $\vec\Delta_{v,z}$ with target $v_i$. We again have $\Gamma=\bigsqcup_{0\leq i\leq n}A_{\cev e_i}$. For each $1\leq i\leq n$,  the convex hull of any $u\in A_{\cev e_i}$ and $v$ contains $v_i$. If $1\leq i\leq m-1$, then $f_{v,v_i}(\epsilon)\neq 0$ implies $f_{v,u}(\epsilon)\neq 0$; if $m\leq i\leq n$ then $f_{v,u}=0$. On the other hand, if $u\in A_{\cev e_0}$, then $\Conv(u,v)$ contains some $u'\in \Delta_{v,z'}$, where $z'\in I_v\backslash \{z\}$. Since $f_{v,z}(\epsilon)=0$, by the locally weak independence at $v$, we have $f_{v,u'}(\epsilon)\neq 0$, hence $f_{v,u}(\epsilon)\neq 0$. As a result, the support of $M$ is $\bigsqcup_{0\leq i\leq m-1}A_{\cev e_i}$, which only depends on the support of $M$ in $\Delta_{v,z}$.
\end{rem}

\begin{proof}[Proof of Theorem~\ref{thm:decomposition}]
Let $(U_v)_v$ be a sub-representation of $M_\Gamma$. We claim that $U_v$ can be decomposed as 
$$U_v=V_v\oplus \bigoplus_{u\in I_v}f_{u,v}(U_u),$$
such that %$(U_v)_v$ satisfies (\ref{eq:decompose independent}).
for each $w\in  \Delta_{v,u}$, we have 
\begin{equation}\label{eq:thmdecompose}
    \ker f_{v,w}|_{U_v}=\ker f_{v,w}|_{V_v}\oplus\ker f_{v,w}|_{f_{u,v}(U_u)}
\end{equation}
Indeed, let us write the cycle of $\Delta_{v,u}$ as $v,v_1,\dotsc,v_n=u,v$. We then have two nested sequences
$$A_1:=\ker f_{v,v_1}|_{f_{u,v}(U_u)}\subset \cdots \subset A_n:=\ker f_{v,v_n}|_{f_{u,v}(U_u)}=f_{u,v}(U_u)$$
and $$B_0:=\{0\}\subset B_1:=\ker f_{v,v_1}|_{U_v}\subset \cdots \subset B_n:=\ker f_{v,v_n}|_{U_v}=\ker f_{v,u}|_{U_v}$$
such that $A_i\subset B_i$ and $A_j\cap B_i=A_i$ for each $j>i$. Take $C_i$ such that $B_i=(A_i+B_{i-1})\oplus C_i$ for $1\leq i\leq n$, 
and denote $U_{v,u}=\oplus_i C_i$. One checks directly by induction on $i$ that $B_i=A_i\oplus (U_{v,u}\cap B_i)$ for each $i$, which is close to what we need for (\ref{eq:thmdecompose}). In particular, taking $i=n$ gives $\ker f_{v,u}|_{U_v}=U_{v,u}\oplus f_{u,v}(U_u)$. 
By locally weak independence of $\Gamma$, $U_v$ can be written as 
$$U_v=U'_v\oplus \bigoplus_{u\in I_v}\Big(U_{v,u}\oplus f_{u,v}(U_u)\Big).$$
We then can take $V_v=U'_v\oplus\big(\bigoplus_{u\in I_v}U_{v,u}\big)$, which satisfies (\ref{eq:thmdecompose}) for all $u\in I_v$ and $w\in \Delta_{v,u}$.

Let $\mathrm{span}(V_v)$ denote the representation generated by $V_v$. 
By Lemma~\ref{lem:decomposition generate} and Lemma~\ref{lem:decompose independent}, $(U_v)_v=\bigoplus\limits_v\mathrm{span}(V_v).$ 
It now remains to show that each $\mathrm{span}(V_v)$ can be decomposed as sub-representations generated by a single vector. If $\mathrm{span}(V_v)$ is projective, then we are done. 
Hence we can assume that there is a $w\in\Gamma$ and $\epsilon\in V_v$ such that $f_{v,w}(\epsilon)=0$. 
By Remark~\ref{rem:irred rep}, we may assume $w=v_m\in \Delta_{v,z}$ for some $z\in I_v$ as in Figure~\ref{fig:repexample}. In particular, $f_{v,v_{m-1}}(\epsilon)\neq 0$. 
We then decompose $V_v$ as $V'_v\oplus \langle \epsilon\rangle$ such that $V'_v$ contains $V_v\cap \Big(\big(\bigoplus_{u\in I_v\backslash z}\ker f_{v,u}\big)\oplus\ker f_{v,v_{m-1}}\Big)$. Since $\epsilon\in \ker f_{v,v_m}\backslash\ker f_{v,v_{m-1}}$, this is possible by  the locally weak independence at $v$. 
%Let $u\in I_v$ be the vertex such that $\Delta_{v,u}$ contains another vertex $z$ that is contained in the convex hull of $[L_v]$ and $[L_w]$. Since $f_{v,w}(s)=0$, we have\com{probably copy a lemma from [hz21]} $f_{v.z}(s)=0$, and hence $s\in\ker f_{v,w}\subset \ker f_{v,u}$. We can decompose $U_v$ as $U'_v\oplus \langle s\rangle$ such that $U'_v$ contains $U_v\cap (\oplus_{x\in I_v\backslash u}\ker f_{v,x})$ by the locally weak independence at $v$. 
It is easy to check that $\mathrm{span}(V_v)=\mathrm{span}(V'_v)\oplus\mathrm{span}(\epsilon)$, and we can argue inductively that $V_v$ has the desired decomposition. %Indeed, if there is an $s'\in U'_v$ and $w'\in \Gamma$ such that $f_{v,w'}(s)=f_{v,w'}(s')\neq 0$,  
%Now each $U_v$ can be decomposed as $$U_v=W_v\oplus\bigoplus_{u\in I_v}\Big(\bigoplus_{w\in  \Delta_{v,u}\backslash v}W^w_v\Big)$$
%such that $$\bigoplus_{v<x\leq w}W^x_v=U_v\cap \ker f_{v,w}$$ for all $w\in  \Delta_{v,u}\backslash \{v\}$. Then this gives a decomposition of each representation $\mathrm{span}(U_v)$ by sub-representations of dimension $\leq \mathbf 1$.\com{details to be added.}  
Moreover, if $(U_v)_v=M_\Gamma$ then $V_v$ is projective, hence so is $M_\Gamma$, and $\Gamma$ is in an apartment by Proposition~\ref{prop:ambient representation} (2).
\end{proof}
%\begin{rem}\label{rem:decomposition}
%Suppose there is a vector $x\in U_v\backslash\bigoplus_{u\in I_v}\ker f_{v,u}|_{U_v}$. Then we have $f_{v,w}(x)\neq 0$ for any $w\in\Gamma$ according to Remark~\ref{rem:irred rep}. In the proof of Theorem~\ref{thm:decomposition}, we can assume $U'_v$ contains $x$, hence $x\in V_v$. Moreover, in the decomposition  $V_v=V'_v\oplus \langle s\rangle$, we can take $V'_v$ such that it even contains $V_v\cap \Big(\big(\bigoplus_{u\in I_v\backslash z}\ker f_{v,u}\big)\oplus\ker f_{v,v_{m-1}}\oplus \langle x\rangle\Big)$. As a result, $\mathrm{span}(\langle x\rangle)$ is a direct summand of $\mathrm{span}(V_v)$, hence a direct summand of $(U_v)_v$. We will need this fact in the next section.
%\end{rem}

\subsection{The stratification of the special fiber}\label{subsec:stratification}
%Given a lattice configuration $\Gamma$ and a dimension vector $\underline x=(x_v)_v$ on $Q(\Gamma)$.
%There is a natural stratification of the quiver Grassmannian $\mathrm{Gr}(\underline x,M_\Gamma)$ given by the family of subsets $(\mathcal S_M)_{[M]}$, where $[M]$ runs over all isomorphic classes of sub-representations of $M_\Gamma$ of dimension $\underline x$ and $S_M\subset \mathrm{Gr}(\underline x,M_\Gamma)$ is the subset parametrizing sub-representations isomorphic to $M$. We associate a preorder ``$\prec$" on $\Gr(\underline x,M_\Gamma)$ where $M\prec M'$ if $M\in \mathcal S^c_{M'}$. We will see in Theorem~\ref{thm:stratification} that this preorder only depends on the isomorphism class of $M$ and $M'$, hence induces a preorder on $(\mathcal S_M)_{[M]}$.
%Recall from Notation~\ref{nt:rank vector} that we associated a rank vector to any representation of $Q(\Gamma)$. 
Let $\Gamma=\{[L_v]\}_v$ be a convex lattice configuration and $\underline x=(x_v)_v$ a dimension vector of $Q(\Gamma)$-representations. In \cite{he2023degenerations} we showed for locally linearly independent configurations that the rank-vector stratification is the same as the quiver stratification of $\Gr(\underline x,\Gamma)$, and that the topological order between quiver strata agrees with the rank-vector order. In this subsection we generalize the conclusion to the locally weakly independent case. Recall from Lemma~\ref{lem:bruhat order and rank vector} that we already know that the Bruhat stratification coincides with the rank-vector stratification for arbitrary convex lattice configurations.

\begin{thm}\label{thm:stratification}Suppose $\Gamma$ is locally weakly independent. Let $[M],[M']\in\Gr({\underline x},M_\Gamma)$ and $\Phi$ be the rank vector as in Notation~\ref{nt:rank vector}. Let $\mathcal S_{M}$ be the quiver stratum as in Notation~\ref{nt:quiver grass}.
\begin{enumerate}
    \item   $\Phi(M)=\Phi(M')$ if and only if $M\simeq M'$;
    \item  $\mathcal S_M\prec_{\mathrm{top}} \mathcal S_{M'}$ if and only if $\Phi(M)\leq \Phi(M')$. 
\end{enumerate}
In particular, the quiver stratification of $\Gr(\mathbf r,M_\Gamma)$ coincides with the Bruhat and rank-vector stratifications, and the topological order is equivalent to the rank-vector order.
\end{thm}
\begin{proof}
According to Theorem~\ref{thm:decomposition}, we can decompose $M$ and $M'$ as $\bigoplus_k M_k^{\oplus\alpha_k}$ and $\bigoplus_k (M'_k)^{\oplus \beta_k}$ respectively, where both $M_k$ and $M'_k$ are sub-representations of $\Gamma$ generated by a single element. We also write $M=(U_v)_v$ and $M'=(U'_v)_v$.

(1) It suffices to show that each $\alpha_k$ is determined by the isomorphic type of $M_k$ and $\Phi(M)$. Suppose $M_k$ is generated by a vector $\epsilon\in U_v$. If $M_k$ is projective, then $$\alpha_k=x_v-\sum_{u\in I_v}\dim\ker f_{v,u}|_{U_v}=x_v-\sum_{u\in I_v}(x_v-\Phi_{v,u}(M)).$$
Now suppose $M_k$ is not projective. As in Remark~\ref{rem:irred rep}, we may assume that there is $z\in I_v$ and $w=v_m\in \Delta_{v,z}$ such that $f_{v,w}(\epsilon)=0$. Recall that we write $\vec\Delta_{v,z}$ as $v=v_0,v_1,...,z=v_n,v_0$, and $m$ is the number such that $f_{v,v_{m-1}}(\epsilon)\neq 0$ and $f_{v,v_m}(\epsilon)=0$. We then have 
$$\alpha_k=\Phi_{v,v_{m-1}}(M)+\Phi_{z,v_m}(M)-\Phi_{v,v_m}(M)-\Phi_{z,v_{m-1}}(M)$$
when $z\neq v_m$; and 
$$\alpha_k=\Phi_{v,v_{m-1}}(M)-\Phi_{v,v_m}(M)-\Phi_{z,v_{m-1}}(M)$$otherwise. 
This completes the proof.

(2)
The ``only if" part is obvious, we need to proof the ``if" part.
Note that for any $u,v\in\Gamma$, if $w\in\Conv(u,v)$, then $\Phi_{v,u}(\boldsymbol{\cdot})=\Phi_{v,w}(\boldsymbol\cdot)$ by Lemma~\ref{lem:convex hull of two points} (3). 
Suppose $\Phi(M)<\Phi(M')$. Then $\Phi_{v,w}(M)<\Phi_{v,w}(M')$ for some $v$ and $w$. After replacing $w$ with a vertex in $\Conv(v,w)$ we may assume $w=v_m\in \Delta_{v,z}$ for some $z\in I_v$. Again, we write $\vec\Delta_{v,z}$ as $v=v_0,v_1,...,z=v_n,v_0$ as in Remark~\ref{rem:irred rep}.
We also write $v_{i+n+1}=v_i$ for convenience. 
Moreover, we assume that $m$ is the smallest number such that $\Phi_{v_i,v_{i+m}}(M)<\Phi_{v_i,v_{i+m}}(M')$ for some $i$. We want to deform $M$ in a way that $\Phi(M)$ increases strictly but is still bounded above by $\Phi(M')$. This will complete the proof when the deformation process is terminated.

%for some $u$ and $v$. Let $[w]$ be the second vertex represented by an element of $\Gamma$ that appears in the unique path in $G_\Gamma$ from $[u]$ to $[v]$. By Proposition~\ref{prop:weakly indep 2} we know that $u$ and $w$ are contained in the same simplex $\Delta=(v_0,\dotsc,v_{n-1},v_0)$ and $\Phi_{u,w}(M)=\Phi_{u,v}(M)<\Phi_{u,v}(M')=\Phi_{u,w}(M')$.  

For all integers $i\leq j$ let $a_{i,j}$ (resp. $a'_{i,j}$) denote the number of copies $M_k$s (resp. $M'_k$s) whose restriction on $\Delta_{v,z}$ is generated by a vector in $U_{v_i}$ (resp. $U'_{v_i}$) and supported on the path $v_i,v_{i+1},\dotsc, v_{j-1},v_j$.  In particular, let $a_{i,j}=a'_{i,j}=0$ when $j-i\geq n+1$. To reduce notation, for $i\leq j$, denote $\Phi_{i,j}=\Phi_{v_i,v_j}(M)$ and $\Phi'_{i,j}=\Phi_{v_i,v_j}(M')$, and set $\Phi_{i,j}=\Phi'_{i,j}=0$ if $j-i\geq n+1$. It follows that
$$\Phi_{i,j}=\sum_{k\leq i}\sum_{l\geq j}a_{k,l}\mathrm{\ and\ }\Phi'_{i,j}=\sum_{k\leq i}\sum_{l\geq j}a'_{k,l}. $$
Therefore, we have
$$a_{i,j}=\Phi_{i,j}+\Phi_{i-1,j+1}-\Phi_{i,j+1}-\Phi_{i-1,j}\mathrm{\ and\ }a'_{i,j}=\Phi'_{i,j}+\Phi'_{i-1,j+1}-\Phi'_{i,j+1}-\Phi'_{i-1,j}.$$

By the assumption on $m$, we have $\Phi_{0,m-1}=\Phi'_{0,m-1}$ and 
$$\sum_{i\leq 0}a_{i,m-1}=\Phi_{0,m-1}-\Phi_{0,m}>\Phi'_{0,m-1}-\Phi'_{0,m}\geq 0.$$
Therefore, there exist some $\lambda\leq 0$ such that $a_{\lambda,m-1}>0$. Let $\lambda$ be the largest one. Then $a_{i,m-1}=0$ for all $\lambda<i\leq 0$. Similarly, $a_{1,\delta'}> 0$ for some $\delta'\geq m$. Let $\delta\geq m$ be the smallest such that there is some $\lambda<\mu\leq 1$ such that $a_{\mu,\delta}\neq 0$. See Figure~\ref{fig:order}. Denote $\sigma_{i,j}=\Phi'_{i,j}-\Phi_{i,j}\geq 0$.
\tikzset{every picture/.style={line width=0.75pt}} 
\begin{figure}[ht]
\begin{tikzpicture}[x=0.45pt,y=0.45pt,yscale=-1,xscale=1]
\import{./}{figorder.tex}
\end{tikzpicture}
\caption{Two direct summands of $M=\bigoplus_k M_k^{\oplus\alpha_k}$ as in the proof of Theorem~\ref{thm:stratification}, part (2).}\label{fig:order}
\end{figure}
\begin{claim}\label{cl:inequality}
We have $\sigma_{i,j}>0$ for all $\lambda \leq i\leq 0$ and $m\leq j\leq \delta$. In particular, $\delta\leq \lambda+n$.
\end{claim}
We postpone the proof of the claim to the end. For all $u\in\Gamma$ such that $\Conv(v,u)$ contains some $v_j$ with $m\leq j\leq\delta$, we have $\Phi_{v_i,u}(M')-\Phi_{v_i,u}(M)=\Phi_{v_i,v_j}(M')-\Phi_{v_i,v_j}(M)=\sigma_{i,j}>0$ for all $\lambda \leq i\leq 0$. Since $a_{\lambda,m-1}>0$ and $a_{\mu,\delta}>0$, in the decomposition $M=\bigoplus_k M_k^{\oplus\alpha_k}$, we may assume that the summand $M_1$ has support $v_\lambda,...,v_{m-1}$ in $\Delta_{v,z}$, and $M_1|_{\Delta_{v,z}}$ is generated by $\ep_1\in U_{v_\lambda}$. Respectively, $M_2|_{\Delta_{v,z}}$ is generated by $\ep_2\in U_{v_\mu}$, and supported on $v_\mu,v_{\mu+1},...,v_\delta$. Since $f_{v_\mu,v_\lambda}(\ep_2)=f_{v_{\delta+1},v_\lambda}(f_{v_\mu,v_{\delta+1}}(\ep_2))=0$, there is an $\ep_3\in \overline L_{v_\lambda}$ such that $f_{v_\lambda,v_\mu}(\ep_3)=\ep_2$. We then deform $M$ by 
taking a copy of $M_1$ in the decomposition $M=\bigoplus_k M_k^{\oplus\alpha_k}$ and replacing with the sub-representation of $M_\Gamma$ generated by $\ep_1+t\ep_3$ where $t\in\kappa$ is  general. We denote the resulting sub-representation by $M^t$. From the claim, one checks directly that $$\Phi(M^t)=\Phi(M)+\sum_{\lambda\leq i\leq 0}\sum_ue_{v_i,u}\leq\Phi(M'),$$
where $e_{\bullet,\bullet}$ is the standard basis of $\mathbb Z^{|\Gamma|^2}$ and $u$ runs through all vertices in $\Gamma$ such that $\Conv(v,u)$ contains some $v_j$ with $m\leq j\leq\delta$. We then replace $M$ by $M^t$ and do the deformation process until $\Phi(M^t)=\Phi(M')$. Then by part (1) we know that $M^t\simeq M'$ and hence $[M]\in \mathcal S^c_{M'}$, which implies that $\mathcal S_M\prec_{\mathrm{top}}\mathcal S_{M'}$.

It now remains to prove Claim~\ref{cl:inequality}. We first consider the case $j=m$. For all $\lambda\leq i< 0$ we have $a_{i+1,m-1}=0\leq a'_{i+1,m-1}$, equivalently, 
$$ \sigma_{i,m}-\sigma_{i,m-1}\geq\sigma_{i+1,m}-\sigma_{i+1,m-1}.$$
By induction on $i$, we have $$\sigma_{i,m}-\sigma_{i,m-1}\geq\sigma_{0,m}-\sigma_{0,m-1}=\sigma_{0,m}>0$$
for all $\lambda\leq i\leq 0$. Therefore, $\sigma_{i,m}\geq \sigma_{0,m}>0$ for all $\lambda\leq i\leq 0$.

Now assume $j>m$. For $\lambda \leq i\leq 0$  we have $a_{i+1,j-1}=0$ by assumption, thus $$ \sigma_{i,j}-\sigma_{i,j-1}\geq\sigma_{i+1,j}-\sigma_{i+1,j-1}.$$
Inductively, we get $$ \sigma_{i,j}-\sigma_{i,j-1}\geq\sigma_{0,j}-\sigma_{0,j-1}.$$
Hence $$\sigma_{i,j}-\sigma_{0,j}\geq\sigma_{i,j-1}-\sigma_{0,j-1}\geq\cdots\geq \sigma_{i,m}-\sigma_{0,m}\geq 0,$$ which means $\sigma_{i,j}\geq \sigma_{0,j}$.
On the other hand, by the assumption on $\delta$, we have $a_{1,j-1}=0\leq a'_{1,j-1}$. Hence $$ \sigma_{0,j}-\sigma_{1,j}\geq\sigma_{0,j-1}-\sigma_{1,j-1}.$$
Inductively, we have $$ \sigma_{0,j}-\sigma_{1,j}\geq\sigma_{0,m}-\sigma_{1,m}=\sigma_{0,m}>0,$$
which means $\sigma_{0,j}>0$. Therefore $\sigma_{i,j}\geq \sigma_{0,j} >0$ for all $\lambda \leq i\leq 0$ and we are done.
\end{proof}

\begin{rem}\label{rem:projective max}
   In the proof above, assuming $\Phi(M)<\Phi(M')$, we have found a direct summand $M_\bullet$ of $M$ supported on $v_\lambda,\dotsc,v_{m-1}$ with $m-1<\delta\leq \lambda+n$ by Claim~\ref{cl:inequality}. Note that this can not happen if $M$ is projective. Therefore, the projective sub-representations have the maximal rank vectors among $\mathrm{Gr}(\underline x,M_\Gamma)$.
\end{rem}

\begin{cor}\label{cor:bruhat order=rank order}
    Let $\Gamma$ be an arbitrary convex lattice configuration contained in an apartment of $\fb_d$. Then the generalized Bruhat-order on $LG_r(\Gamma)$ is the same as the rank-vector order. 
\end{cor}
\begin{proof}
    As we already seen in the proof of Lemma~\ref{lem:bruhat order and rank vector}, the rank vector of a sub-representation $M$ of $M_\Gamma$ is determined by the rank vector of the restriction of $M$ on each maximal simplex in $\Theta(\Gamma)$. On the other hand, by definition, the generalized Bruhat order is also defined in terms of each maximal simplex. Therefore, we may assume that $\Gamma$ is a simplex, hence locally weakly independent. By Theorem~\ref{thm:stratification}, the rank-vector order is same as topological order, which is equivalent to the Bruhat order by Theorem~\ref{thm:open_cover1} (2).
\end{proof}

Now, following Theorem~\ref{thm:stratification}, an identical argument as in \cite[Proposition 3.5]{he2023degenerations} shows that $(\mathcal S_{M})_{[M]}$ is a family of ``well-behaved" strata. More precisely, we have

\begin{prop}
$\mathcal S_M$ is an irreducible and locally closed subset of $\Gr(\underline x,M_\Gamma)$ of dimension $$\dim \h(M,M_\Gamma)-\dim\End{(M)}.$$
\end{prop}

We next compute all possible rank vectors that arise from $\Gr(\underline x,M_\Gamma)$. According to Theorem~\ref{thm:stratification} (1), this gives a numerical classification of the strata $(\mathcal S_{M})_{[M]}$. We have done this for $LG_r(\Gamma)_0=\Gr(\mathbf r,M_\Gamma)$ when $\Gamma$ is a locally linearly independent configuration in \cite{he2023degenerations}. In this paper, let us focus on the case when $\Gamma$ is a simplex for simplicity while varying the dimension $\underline x$, although the argument also works generally for linearly weakly independent configurations. Note that  a similar result was already carried out for $LG_r(\Gamma)_0$, where $\Gamma$ is a maximal simplex in $\fb_d$, in \cite{gortz2010supersingular} in terms of Bruhat stratification. However, 
as stated in the paragraph following Proposition 2.6 of \textit{loc.cit.}, a ``simple characterization" of the possible rank vectors was not obtained.

We may write $\Gamma=\{v_0,...,v_n\}$ and $v_i=v_{i+n+1}$ for each $i$. For convenience, we write $x_i=x_{v_i}$ as coordinates of $\underline x$ and  $\Phi_{i,j}=\Phi_{v_i,v_j}$; we also use $(d_{i,j})_J$ as coordinates of $\mathbb Z^{|\Gamma|^2}$, where $J=\{(i,j)|0\leq i\leq n\mathrm{\ and\ }i\leq j\leq n+i\}$. 
In addition, for arbitrary $i$ and $j$, write $d_{i,j}=d_{i+n+1,j+n+1}$ if $0\leq j-i\leq n$, and $d_{i,j}=0$ otherwise. The rank-vector map $\Phi$  is then given by $d_{i,j}=\Phi_{i,j}$ for $(i,j)\in J$ and $j\leq n$, and $d_{i,j}=\Phi_{i,j-(n+1)}$ otherwise. In addition, we have $\Phi_{i,i}=x_i$ for each $i$ by construction. 

\begin{prop}\label{prop:simplex strata}
Let $\Gamma$ be a simplex as above. Then a tuple $(d_{i,j})_J$ is contained in the set  $\Phi(\Gr(\underline x,M_\Gamma))$ of rank vectors if and only if 
\begin{enumerate}
    \item for each $(i,j)\in J$, $d_{i,j}+d_{i-1,j+1}-d_{i,j+1}-d_{i-1,j}\geq 0$; and
    \item     
    $d_{i,i}=x_i$  and $d_{i,i}-d_{i,i+1}\leq \Phi_{i,i}(M_\Gamma)-\Phi_{i,i+1}(M_\Gamma)=\dim\ker f_{i,i+1}$ for $0\leq i\leq n$.
\end{enumerate}
\end{prop}

Suppose $(d_{i,j})$ is realizable as the rank vector of a representation $M$, which is decomposed as a direct sum of irreducible ones. Then
Condition (1) is just saying that, for each $(i,j)\in J$, the expected number of representations in the decomposition supported on the path $v_i,...,v_j$ is non-negative; and Condition (2) is saying that the number of all representations in the decomposition of $M$ supported on a path with target $v_i$ is no more than that in the decomposition of $M_\Gamma$.

\begin{proof}[Proof of Proposition~\ref{prop:simplex strata}]
Write $f_{i,j}=f_{v_i,v_j}$ and $\overline L_i=\overline L_{v_i}$ for convenience. We first prove the ``only if" part. By the discussion above, it remains to prove the second half of part (2). Suppose $(d_{i,j})=\Phi(M)$, then $d_{i,i}-d_{i,i+1}=\dim\ker f_{i,i+1}|_M\leq \dim\ker f_{i,i+1}$ and we are done.

For the ``if" part, let $M=\oplus_J M_{i,j}$ where $M_{i,j}$ is the direct sum of $a_{i,j}:=d_{i,j}+d_{i-1,j+1}-d_{i,j+1}-d_{i-1,j}\geq 0$ copies of the irreducible sub-representations of $M_\Gamma$ supported on $v_i,...,v_j$. Similarly, we set $a_{i,j}=a_{i+n+1,j+n+1}$ if $0\leq j-i\leq n$, and $a_{i,j}=0$ otherwise.
It suffices to show that $M$ can be realized as a sub-representation of $M_\Gamma$. 

Denote $m_j:=d_{j,j}-d_{j,j+1}=\sum_{k\leq j}  a_{k,j}$. Take $m_j$ independent vectors $y_1,\dotsc,y_{m_j}\in \overline L_{j+1}$ such that they generate a subspace of $\overline L_{j+1}$ which intersects trivially with $f_{j,j+1} (\overline L_j)=\ker f_{j+1,j}$. This is possible by Condition (2). It follows that these vectors generate a sub-representation $M'_j=\oplus_{1\leq l\leq m_j}P_{y_l}$ of $M_\Gamma$, where $P_{y_l}$ is the projective sub-representation generated by $y_l$. According to Lemma~\ref{lem:decompose independent}, $\sum_{0\leq j\leq n} M'_j$, as a sub-representation of $M_\Gamma$, is a direct sum. Now for each $1\leq l\leq m_j$, there is a $k_l\leq j$ such that $\sum_{k\leq k_l-1}a_{k,j}<l\leq \sum_{k\leq k_l}a_{k,j}$. We take $R_l$ to be the sub-representation of $P_{y_l}$ generated by $f_{j+1,k_l}(y_l)$, and let $M_j\subset M'_j$ be the direct sum of all $R_l$s. Then for each $0\leq j\leq n$, $M_j$ is isomorphic to the direct sum of the items $M_{\bullet,\bullet}$ supported on a path with target $v_j$ (in other words, $M_{\bullet,j}$ and $M_{\bullet,j+n+1}$). We then have $M\simeq \oplus_{0\leq j\leq n} M_j$, which is a sub-representation of $M_\Gamma$.
\end{proof}

We now focus on $LG_r(\Gamma)_0=\Gr(\mathbf{r},M_\Gamma)$.
In \cite{he2023degenerations} we proved that, when $\Gamma$ is locally linearly independent, every sub-representation of $M_\Gamma$ of dimension $\mathbf{r}$ can be decomposed as a direct sum of sub-representations of dimension $\mathbf{1}=(1,1,...,1)$. However, this is not necessarily the case when $\Gamma$ is only locally weakly independent. For instance, consider $\Gamma$ as a simplex with 3 vertices $v_0,v_1,v_2$. By Proposition~\ref{prop:simplex strata}, there exist sub-representation $[M]\in LG_2(\Gamma)_0$ such that $M$ is decomposed as $M=\oplus_{0\leq i\leq 2}M_i$, where each $M_i$ is supported on $\Gamma\backslash\{v_i\}$. Then it is easy to see that $M$ can not be decomposed into representations of dimension $\mathbf{1}$. Nevertheless, it is still true that the projective sub-representations are dense in $LG_r(\Gamma)_0$, even when $\Gamma$ is only locally weakly independent. This, combined with the fact that projective points in $LG_r(\Gamma)_0$ are contained in the closure of the generic fiber $LG_r(\Gamma)_\eta$, shows that $LG_r(\Gamma)$ is irreducible. See Theorem~\ref{thm:flatness}.

\begin{prop}\label{prop:projective dense}
Suppose $\Gamma$ is locally weakly independent, then projective sub-representations are dense in $LG_r(\Gamma)_0$. Moreover, a representation $[M]\in LG_r(\Gamma)_0$ is projective if and only if $\Phi(M)$ is maximal among $LG_r(\Gamma)_0$. In particular, the irreducible components of $LG_r(\Gamma)_0$ are exactly of the form $\mathcal S_M^c$, where $M$ is any isomorphic class of projective sub-representation of $M_\Gamma$.
\end{prop}
\begin{proof}
Let $M:=(U_v)_v$ be a sub-representation of $M_\Gamma$ that is not projective. By Theorem~\ref{thm:decomposition}, we may decompose $M$ as $M=\oplus_i M_i$, where each $M_i$ is generated by a single vector. We may assume that $M_1$ is not projective and generated by $\ep_1\in U_v$. As in Remark~\ref{rem:irred rep}, there is a $z\in I_v$ and $w\in\Delta_{v,z}$ such that $f_{v,w}(\ep_1)=0$. We again write $\vec\Delta_{v,z}$ as $v=v_0,v_1,...,w=v_m,...,z=v_n,v_0$ and assume $f_{v,v_{m-1}}(\ep_1)\neq 0$. Moreover, we assume that $m$ is the largest number such that there is some $M_i$ such that $M_i|_{\Delta_{v,z}}$ is not projective and supported on $m$ vertices.

Since $\dim U_u=r$ for all $u\in\Delta_{v,z}$, there must be some $M_i$, say $M_2$, whose restriction on $\Delta_{v,z}$ is generated by a vector $\ep_2\in U_{v_k}$ with $m\leq k\leq n$, such that for some $k>l\geq 0$, we have $f_{v_k,v_l}(\ep_2)=0$ and $f_{v_k,v_{l-1}}(\ep_2)\neq 0$ (here we set $v_{-1}=v_n$ if necessary). See Figure~\ref{fig:projective}. Since $m$ is maximal, we must have $l<m$ and $M_2$ is not projective. Note that $M_2$ is generated by $\ep_2$ according to Remark~\ref{rem:irred rep}.

  \tikzset{every picture/.style={line width=0.75pt}} 
\begin{figure}[ht]
\begin{tikzpicture}[x=0.45pt,y=0.45pt,yscale=-1,xscale=1]
\import{./}{figprojective.tex}
\end{tikzpicture}
\caption{The direct summands $M_1$ and $M_2$ of $M$ as in the proof of Proposition~\ref{prop:projective dense}.}\label{fig:projective}
\end{figure}

Since $f_{v,v_k}(\ep_1)=0$, there is an $\ep_3\in \overline L_{v_k}$ such that $f_{v_k,v}(\ep_3)=\ep_1$. We now deform $M$ as follows: In the decomposition $M=\oplus_i M_i$,
replace $M_2$ with the representation generated by $\ep_2+t\cdot \ep_3$. As a result, we get a sub-representation $M^t$ for general $t\in \kappa$. 
As in the proof of Theorem~\ref{thm:stratification}, we see that
\begin{equation}\label{eq:dimension of defom rep}
    \Phi(M^t)=\Phi(M)+\sum_{k\leq i\leq n}\sum_we_{v_i,w}
\end{equation} where $w$ runs over all vertices in $\Gamma$ such that $\Conv(v,w)$ contains one of $v_l,\dotsc,v_{m-1}$. It follows that $\Phi(M^t)>\Phi(M)$ and $\mathcal S_M\subset \mathcal S_{M^t}^c$. If $M^t$ is not projective, we replace $M$ with $M^t$ and apply the same argument. Since $\Phi(M^t)$ is bounded above by $(r,\dotsc,r)$, the argument must terminate, and in this case $M$ is deformed to projective representations. Hence, we get the density of projective representations. 

We now prove the second assertion. If $M$ is projective, then $\Phi(M)$ is maximal by Remark~\ref{rem:projective max}. On the other hand, suppose $\Phi(M)$ is maximal. By the density of projective points in $LG_r(\Gamma)_0$ and Theorem~\ref{thm:stratification} (2), there must be some projective representation $[P]\in LG_r(\Gamma)_0$ such that $\Phi(M)\leq \Phi(P)$. Hence $\Phi(M)= \Phi(P)$ and $M\simeq P$ is projective by Theorem~\ref{thm:stratification} (1).
\end{proof}

\subsection{Global geometry}\label{subsec:weakly indep global geometry}
We end this section by showing that the linked Grassmannian $LG_r(\Gamma)$ is a flat degeneration of its generic fiber, i.e., a Grassmannian, when $\Gamma$ is locally weakly independent. In particular, we obtain the irreducibility of $LG_r(\Gamma)$, which is the key property that plays a role in the study of smoothability of limit linear series. See \S~\ref{sec:lls}. Note that a similar result was obtained when $\Gamma$ is locally linearly independent in \cite{he2023degenerations}.

%\begin{lem}\label{lem:cohen macaulay}
%Suppose $\Gamma$ is a simplex. Then the special fiber $LG_r(\Gamma)_0=\Gr(\mathbf r, M_\Gamma)$ is Cohen-Macaulay.
%\end{lem}
%\begin{proof}
%Then case $\mathrm{char}(\kappa)>0$ was covered by \cite[Theorem 1.2]{he2013normality}. Now assume $\mathrm{char}(\kappa)=0$. Note that the quiver Grassmannian $\Gr(\mathbf r, M_\Gamma)$ is defined over $\mathbb Z$; in other words, there is a ``universal quiver Grassmannian" $f\colon \mathcal Gr(\mathbf r, M_\Gamma)\rightarrow \mathbb Z$ from which $\Gr(\mathbf r, M_\Gamma)$ is obtained by base extension from $\mathbb Z$ to $\kappa$. 

%According to \cite[Corollary 4.20]{gortz2001flatness}, $\mathcal Gr(\mathbf r, M_\Gamma)$  has reduced fibers. Since $\mathcal Gr(\mathbf r, M_\Gamma)$ is contained in the closure of its generic fiber over $\mathbb Z$, it is flat over $\mathbb Z$. By \cite[\href{https://stacks.math.columbia.edu/tag/045Q}{Tag 045Q}]{stacks-project}, the Cohen-Macaulay locus of $f$ on $\mathcal Gr(\mathbf r, M_\Gamma)$, which contains the points over $\mathbb Z\backslash \{0\}$, is open. Hence $f$ is Cohen-Macaulay on $\mathcal Gr(\mathbf r, M_\Gamma)$; in particular, its fiber over $0$ (i.e. the quiver Grassmannian $\Gr(\mathbf r, M_\Gamma)$ defined over $\mathbb Q$) is Cohen-Macaulay. Again, by \cite[\href{https://stacks.math.columbia.edu/tag/045Q}{Tag 045Q}]{stacks-project}, the Cohen-Macaulayness is invariant under base field extensions, hence $\Gr(\mathbf r, M_\Gamma)$ is Cohen-Macaulay over any $\kappa$.
%\end{proof}

\begin{thm}\label{thm:flatness}
Let $\Gamma$ be a locally weakly independent configuration. 
Then $LG_r(\Gamma)$ is irreducible and flat over $R$. Moreover, both $LG_r(\Gamma)$ and its special fiber are reduced and Cohen-Macaulay. %As a result, $LG_r(\Gamma)$ is scheme-theoretically isomorphic to the corresponding Mustafin degeneration.
\end{thm}
With all the ingredients we've developed so far, the proof is now identical to the proof of \cite[Theorem 3.12]{he2023degenerations}. We include it here for the convenience of the readers.
\begin{proof} 
By Proposition \ref{prop:projective dense}, we have a rational and dominant morphism  $$\prod_{1\leq i\leq r}LG_1(\Gamma)\dashrightarrow LG_r(\Gamma)$$ induced by taking the direct sum of the dimension-$1$ subspaces, where the product on the left is over $R$. Since $LG_1(\Gamma)$ is irreducible, so is the product. Hence $LG_r(\Gamma)$ is irreducible.

We now show the Cohen-Macaulayness of $\mathrm{Gr}(\mathbf r,M_\Gamma)=LG_r(\Gamma)_0$ by induction on the number of maximal simplices contained in $\Theta(\Gamma)$. The base case is that $\Gamma$ is a simplex, which is covered in \cite[Theorem 1.2]{he2013normality}.
So we may assume $\Theta(\Gamma)$ contains at least two maximal simplices. Let $\Delta$ be a maximal simplex in $\Theta(\Gamma)$ intersecting the union of the other maximal simplices in exactly one vertex $v$. Let $\Gamma'=(\Gamma\backslash \Delta)\cup\{[L_v]\}$. Then $\Gamma'$ is locally weakly independent by Proposition~\ref{prop:weakly indep 2} (2). 
Since the linked Grassmannians are irreducible, the special fibers $\Gr(\mathbf r,M_\Gamma)$, $\Gr(\mathbf r,M_{\Gamma'})$ and $\Gr(\mathbf r, M_{\Delta})$  all have pure dimensions $r(d-r)$. It follows that $\Gr(\mathbf r, M_\Gamma)$ is a local complete intersection in $\Gr(\mathbf r,M_{\Gamma'})\times \Gr(\mathbf r,M_\Delta)$. By the inductive hypothesis and \cite[\href{https://stacks.math.columbia.edu/tag/045Q}{Tag 045Q}]{stacks-project},  $\Gr(\mathbf r,M_{\Gamma'})\times \Gr(\mathbf r, M_\Delta)$ is Cohen-Macaulay, hence so is $\Gr(\mathbf r,M_{\Gamma})$. 

By Proposition~\ref{prop:projective dense}, the points represented by projective sub-representations of $M_\Gamma$ are dense in $\Gr(\mathbf r,M_\Gamma)$. Thus, according to Remark~\ref{rem:weakly indep and pre-linked}, $\Gr(\mathbf r,M_\Gamma)$ is generically smooth and hence 
generically reduced. Therefore, $\Gr(\mathbf r,M_\Gamma)$ is reduced by Cohen-Macaulayness. The reducedness and flatness of $LG_r(\Gamma)$ now follow from the irreducibility and \cite[Lemma 6.13]{Olls}. The Cohen-Macaulayness of $LG_r(\Gamma)$ is a consequence of \cite[Cor., page 181]{matsumura_1987}.
\end{proof}

\section{The Dimension-$\mathbf 1$ case}\label{sec:r1}
In this section, we assume $\Gamma$ is an arbitrary convex lattice configuration. The linked Grassmannian $LG_1(\Gamma)$ was historically known as \textit{Mustafin variety}. It was introduced by Mustafin (\cite{mustafin1978}) to study certain flat degenerations of projective spaces. The combinatorial geometry of Mustafin varieties was further investigated by Cartwright et al. in \cite{cartwright2011mustafin}. 

We will show that, as in Section~\ref{sec:weakind}, the topological order on $LG_1(\Gamma)_0$ with respect to the Bruhat stratification agrees with the generalized Bruhat-order. More precisely, we show that (1) the Bruhat stratification of $LG_1(\Gamma)_0$ is the same as the quiver stratification of $\Gr(\mathbf 1,M_\Gamma)$; and (2) the topological order of the latter is the same as the rank-vector order, hence the same as the generalized Bruhat order by Corollary~\ref{cor:bruhat order=rank order}. 
We do this by proving that the set of isomorphic classes of dimension-$\mathbf 1$ sub-representations $M$ of $M_\Gamma$ is identified with the set of faces of $\Theta(\Gamma)$ (Notation~\ref{nt:complex}), and the topological order is induced by the (reversed) containment relation between such faces. 
This fact was essentially obtained by Faltings in \cite[\S 5]{faltings2001toroidal} by carrying out the scheme structure of the closure of each quiver stratum $\mathcal S_M$.
%Note that Faltings used a dual notation of convexity of lattice configurations; and using this fact, Cartwright et. al. proved that the reduction complex of $LG_1(\Gamma)$ is identified with the simplicial structure on $\Gamma$ induced from $\mathfrak B_d$.
For the reader's convenience, we recover this fact using the language of quiver representations. This will also give a combinatorial description of each $M$ through its corresponding face of $\Theta(\Gamma)$.

As mentioned before, we denote $\Gamma=([L_v])_v$ and $M_\Gamma=(\overline L_v:=L_v/\pi L_v)_v$. 
Denote $M=(\langle \ep_v\rangle)_v$. Following Faltings' argument, we call a vertex $u\in \Gamma$ \textit{maximal} with respect to $M$ if $f_{v,u}(\ep_v)=0$ for all $v\neq u$. It is easy to see that $M$ is generated by vectors from maximal vertices. We extract the observation in Faltings' argument that the maximal vertices with respect to $M$ are mutually adjacent. Indeed, suppose $u$ and $v$ are maximal and not adjacent, then we can pick an vertex $w$ in the convex hull of $u$ and $v$. According to Lemma~\ref{lem:convex hull of two points} (3), either $f_{w,u}(\ep_w)\neq 0$ or $f_{w,v}(\ep_w)\neq 0$, which provides a contradiction. It follows that the set of all maximal vertices of $\Gamma$ with respect to $M$ form a simplex, which we denote by $\Delta_M$. Given $u\in \Delta_M$, we denote by $J_{u,M}$ the set of vertices $v\in \Gamma$ such that $f_{u,v}(\ep_u)\neq 0$. Apparently, we have $\Gamma=\bigcup_u J_{u,M}$.

\begin{lem}\label{lem:simplex unique factor through}
    Let $\Delta\subset \fb_d$ be a simplex. Then for any vertex $v\in \mathfrak B_d$, there is a unique vertex $u$ of $\Delta$ such that for all $u'\in \Delta$, the map $f_{u',v}$ factors through $u$.
\end{lem}
\begin{proof}
    By induction, we may assume that $\Delta$ only consists of two lattice classes $[L_u]$ and $[L_w]$. Let $L_v$ be a representative of $v$. We may assume that $L_u\subset L_w\subset \pi^{-1}L_u$ and $L_u\subset L_v$ and $\pi^{-1}L_u\not\subset L_v$. If $L_w\subset L_v$ then $f_{u,v}=f_{w,v}\circ f_{u,w}$; otherwise $f_{w,v}=f_{u,v}\circ f_{w,u}$.
\end{proof}

By construction, we have $f_{u,u'}(\ep_u)=0$ for any $u\neq u'\in \Delta_M$. Combining the lemma above, we directly have the following description of $J_{u,M}$: 

\begin{cor}
$J_{u,M}$ consists of all vertices $v$ in $\Gamma$ such that $f_{u',v}=f_{u,v}\circ f_{u',u}$ for all $u'\in \Delta_M$. In particular, 
suppose $u\neq w\in \Delta_M$, then $J_{u,M}\cap J_{w,M}=\emptyset$.
\end{cor}

It follows that one can write $M$ as $\oplus_{u\in\Delta_M}M_u$, where each $M_u$ is supported on $J_{u,M}$ and generated by $\ep_u$. Note also that $J_{u,M}$ only depends on $\Delta_M$ and $u\in\Delta_M$, hence $\mathcal S_M$ consists of representations $[M']$ such that $\Delta_{M'}=\Delta_M$. 

\begin{cor}
    Given two sub-representations $M,M'$ of $M_\Gamma$ with dimension $\mathbf 1$, we have $M\simeq M'$ if and only if $\Phi(M)=\Phi(M')$. As a result, the quiver stratification on $\Gr(\mathbf 1,M_\Gamma)$ agrees with the rank-vector and Bruhat stratifications of $LG_1(\Gamma)_0$.
\end{cor}
In order to get a correspondence between the quiver strata of $\Gr(\mathbf 1,M_\Gamma)$ and the set of faces in $\Theta(\Gamma)$, we prove the following:

\begin{lem}\label{lem:dimension 1 all subrep}
Let $\Delta$ be any face of $\Theta(\Gamma)$. There is a representation $[M]\in \mathrm{Gr}(\mathbf 1,M_\Gamma)$ such that $\Delta_M=\Delta$.
\end{lem}
\begin{proof}
For each $u\in\Delta$, let $J_u$ be the subset of $\Gamma$ consisting of vertices $w$ such that $f_{u',w}=f_{u,w}\circ f_{u',u}$ for all $u'\in\Delta$. According to Lemma~\ref{lem:simplex unique factor through}, $\Gamma=\bigsqcup_{u\in\Delta}J_u$. We denote the cycle $\vec\Delta$ by $(u_1,\dotsc,u_n,u_1)$, where the subscripts are considered modulo $n$.

Now for any $u_i\in\Delta$ and $w\in J_{u_i}$, we have $f_{u_i,w}\neq f_{u_{i-1},w}\circ f_{u_i,u_{i-1}}$, since otherwise $f_{u_{i-1},w}=f_{u_i,w}\circ f_{u_{i-1},u_i}=f_{u_{i-1},w}\circ f_{u_i,u_{i-1}}\circ f_{u_{i-1},u_i}=0$ by Proposition~\ref{prop:ambient representation1} (3).
Therefore, $$\bigcup_{j\neq i}\ker f_{u_i,u_j}=\ker f_{u_i,u_{i-1}}\not\subset \ker f_{u_i,w}.$$
Therefore, we can take a vector $\ep_i\in \ker f_{u_i,u_{i-1}}\backslash\bigcup_{w\in J_u}\ker f_{u_i,w}$. It is then straightforward to check that $(\ep_i)_{1\leq i\leq n}$ generates a dimension-$\mathbf 1$ sub-representation $M$ of $M_\Gamma$ such that $\Delta_M=\Delta$.
\end{proof}

We now state the main conclusion of this section.

\begin{prop}\label{prop:dimension 1 stratification}
Let $\Gamma$ be a convex lattice configuration and $M,M'$ be two dimension-$\mathbf 1$ sub-representations of $M_\Gamma$. The following are equivalent:
\begin{enumerate}
    \item $\mathcal S_{M}\subset \mathcal S_{M'}^c$;
    \item $\Phi(M)\leq \Phi(M')$;
    \item $\Delta_{M'}\subset \Delta_M$.
\end{enumerate}
In particular, the Bruhat order on $LG_1(\Gamma)_0=\Gr(\mathbf 1,\Gamma)$ agrees with the topological order and the rank-vector order. 
\end{prop}
\begin{proof}
(1)$\Rightarrow$(2) follows from semi-continuity of the dimension function; and (2)$\Rightarrow$(3) is immediate from construction. We now show that (3) implies (1). By induction we may assume $\Delta_M=\Delta_{M'}\cup \{u\}$. We may still write the cycle $\vec\Delta_M$ as $(u=u_1,\dotsc,u_n,u_1)$ as in the proof of Lemma~\ref{lem:dimension 1 all subrep}.

Suppose $M$ is generated by $(\ep_i\in\overline L_{u_i})_{1\leq i\leq n}$. Since $f_{u_1,u_n}(\ep_1)=0$, there is an $\ep\in \overline L_{u_n}$ such that $f_{u_n,u_1}(\ep)=\ep_1$. For each $t\neq 0\in \kappa$ let $M_t$ be the representation generated by $(\ep_i\in\overline L_{u_i})_{2\leq i< n}$ and $\ep_n+t\ep\in\overline L_{u_n}$. Then it is easy to check that $\Delta_{M_t}=\Delta_{M'}$ and $M_t$ approaches $M$ as $t$ approaches $0$. Hence $[M]\in\mathcal S_{M'}^c$.   
%This shows that the topological order agrees with the rank-vector order. The conclusion about the generalized Bruhat-order follows from Corollary~\ref{cor:bruhat order=rank order}.
\end{proof}

%is a direct sum over the vertices in $\Delta_M$ of sub-representations of $M_\Gamma$ with disjoint support, and each direct summand is generated by a single vector that is from a vertex of $\Delta_M$. Moreover, the union of the support of each di

\section{Application in degeneration of curves and linear series}\label{sec:lls}
In \cite{he2023degenerations}, we investigated the connection between the theories of linked Grassmannians and limit linear series on nodal curves. As an application of the irreducibility of linked Grassmannians associated to locally linearly independent lattice configurations, we proved the smoothing theorem of certain limit linear series on a curve consisting of three rational components. In this paper, using the theory of locally weakly independent configurations, we prove the smoothing theorem for a new family of curves. See \S~\ref{subsec:smoothing}. We start by recalling the basic theory of limit linear series.  Throughout this section, we assume that the discrete valuation ring $R$ is complete.

\subsection{Preliminaries}
We recall the notion of limit linear series on a nodal curve in a \textit{regular smoothing family}, and briefly explain its connection with linked Grassmannians. Unless otherwise stated, all definitions in this subsection are from \cite{osserman2019limit}. We also refer~\cite[\S~4]{he2023degenerations} for a more generalized and detailed illustration, and use the notation thereof. 
\begin{defn}\label{defn:smoothing family} 
We say that a flat and proper family $\varphi\colon X\rightarrow \spec(R)$ of curves is a \textit{regular smoothing family} if (1) $X$ is regular and the generic fiber $X_\eta$ is smooth; (2) the special fiber $X_0$ of $\varphi$ is a (split) nodal curve; and (3) $\varphi$ admits sections through every component of $X_0$.
\end{defn}

Note that condition (3) is satisfied automatically if $R$ is complete by \cite[Proposition 10.1.40(a)]{liu2002algebraic}.
Let $G$ be the dual graph of $X_0$ and $Z_v$ the irreducible component of $X_0$ corresponding to $v\in V(G)$. 
The set of multidegrees on $X_0$ is in one-to-one correspondence with the set of divisors on $G$ in a natural way. We say that a multidegree $w$ is obtained from $w'$ by a \textit{twist} at $v\in V(G)$ if the divisor associated to $w$ is obtained from $w'$ as follows: if $v'$ is adjacent to $v$, we increase the degree of $w'$ at $v'$ by one; we decrease the degree of $w'$ at $v$ by the number of vertices adjacent to $v$. In this case, we also say that $w'$ is obtained from $w$ by a \textit{negative twist} at $v$. Geometrically, if a line bundle $\mathscr L$ on $X$ has multidegree $w$ on $X_0$, then $w$ twisted (resp. negatively twisted) at $v$ is the multidegree of $\mathscr L(Z_v)$ (resp. $\mathscr L(-Z_v)$), that is, the multidegree of $\mathscr L$ twisted (resp. negatively twisted) by the irreducible component of $X_0$ corresponding to $v$. Note that the twist of $w$ is commutative, and twisting all vertices of $G$ preserves $w$.

\begin{defn}\label{defn:concentrated multidegree}
A multidegree $w$ is \textbf{concentrated} on $v$ if there is an ordering on $V(G)$ starting at $v$, and such that for each subsequent vertex $v'$, we have that $w$ becomes negative in vertex $v'$ after taking the composition of the negative twists at all previous vertices. 
\end{defn}

Let $V(G(w_0))$ be the set of all multidegrees on $G$ obtained from $w_0$ by a sequence of twists. 
Choose a tuple $(w_v)_{v\in V(G)}\subset V(G(w_0))$ of multidegrees on $X_0$ such that $w_v$ is concentrated on $v$. Let $V(\ov G(w_0))$ be the (finite) subset of $V(G(w_0))$ consisting of multidegrees $w$ in $V(G(w_0))$ such that, for all $v\in V(G)$, $w_v$ can be obtained from $w$ by twisting vertices other than $v$.

Suppose $\mathscr L$ is a line bundle on $X_0$ of multidegree $w_0$. For any $w\in V(G(w_0))$ we get a line bundle $\mathscr L_w$ with multidegree $w$ on $X_0$ as follows: if $w$ is obtained from $w_0$ by subsequently twisting $a_v$ times at $v$, where the $a_v$'s are non-negative and at least one of them equals $0$, then we set 
$$\mathscr L_w=\mathscr L\otimes\mathscr O_X(\sum a_vZ_v)|_{X_0}= \mathscr L\otimes \Big(\bigotimes_v \mathscr (\mathscr O_X(Z_v)|_{X_0})^{\otimes a_v}\Big).$$
Given another $w'\in V(G(w_0))$. Suppose $w'$ is obtained from $w$ by subsequently twisting $a'_v$ times at $v$, with the $a'_v$'s satisfying similar conditions as the $a_v$'s. We denote by 
$f_{w,w'}\colon \mathscr L_w\rightarrow \mathscr L_{w'}\simeq\mathscr L_w\otimes\mathscr O_X(\sum a'_vZ_v)|_{X_0}$
the natural morphism induced by $\mathscr O_X\hookrightarrow \mathscr O_X(\sum a'_vZ_v)$.
In general, to define $\mathscr L_w$ and $f_{w,w'}$ without referring $X$, the tuple $(\mathscr O_X(Z_v)|_{X_0})_v$ 
is replaced by the so-called \textit{enriched structure} on $X_0$. 

\begin{defn}\label{defn:limit linear series}
Let $X_0$ and $G$ be as above. Fix a multidegree $w_0$ with total degree $d$, and fix a number $r<d$.  A \textit{limit linear series} on $X_0$, usually denoted by a \textit{limit} $\mathfrak g^r_d$, consists of a line bundle $\mathscr L$ of multidegree $w_0$ on $X_0$ together with subspaces $V_v\subset H^0(X_0,\mathscr L_{w_v})$ of dimension $(r+1)$  such that for all $w\in V(\overline G(w_0))$, the kernel of the linear map 
\begin{equation}\label{eq:limit linear series}  H^0(X_0,\mathscr L_w)\to\displaystyle\bigoplus_{v\in V(G)}H^0 (X_0,\mathscr L_{w_v})/V_v 
\end{equation}
induced by $\op_vf_{w,w_v}$  has dimension at least $r+1$.
\end{defn}

Note that the definition above is independent of the choice of $(w_v)_v$. A closely related concept is \textit{linked linear series}, which, as we shall see shortly, provides a bridge between limit linear series and linked Grassmannians. We now recall its definition, which was first introduced in  {\cite{Ohrk}}.

\begin{defn}\label{defn:linked linear series}
Use the same notation as in Definition \ref{defn:limit linear series}. A \textit{linked linear series}, or a \textit{linked} $\mathfrak g^r_d$, on $X_0$ consists of a line bundle $\mathscr L$ on $X_0$ of multidegree $w_0$ together with subspaces $V_w\subset H^0(X_0,\mathscr L_w)$ of dimension $(r+1)$ for all $w\in V(\overline G(w_0))$ such that 
\begin{equation}\label{eq:linked linear series}
   f_{w,w'}(V_w)\subset V_{w'}\mathrm{\ for\ all\ } w,w'\in V(\overline G(w_0)). 
\end{equation}
\end{defn}

It is possible, according to \cite[Remark 4.12]{he2023degenerations}, to choose $(w_v)_v$ such that $w_v\in V(\ov G(w_0))$ for all $v$. In this case, given a linked linear series $(V_w)_{w\in V(\overline G(w_0))}$, we get immediately a limit linear series by setting $(V_v)_{v\in V(G)}=(V_{w_v})_{v\in V(G)}$. This induces a forgetful map from the moduli space of linked linear series to the moduli space of limit linear series. We next explain how the constructions above are related to linked Grassmannians. Assume in the sequel that $w_v\in V(\ov G(w_0))$ for all $v$. 

%We next recall briefly the construction of the these moduli spaces. Assume in the sequel that $w_v\in V(\ov G(w_0))$ for all $v$. 

%Let $\mathrm{Pic}^{w_0}(X/R)$ be the moduli scheme of line bundles on $X$ of relative degree $d$ over $\spec R$ which have multidegree $w_0$ on $X_0$. There is a relative moduli space $\mathcal G$ (resp. $\widetilde {\mathcal G}$) over $\spec R$ such that the generic fiber $\mathcal G_\eta$ (resp. $\widetilde {\mathcal G}_\eta$) parametrizes $\mathfrak g^r_d$s on $X_\eta$ and the special fiber $\mathcal G_0$ (resp. $\widetilde {\mathcal G}_0$) parametrizes limit (resp. linked) $\mathfrak g^r_d$s on $X_0$ of multidegree $w_0$. There are natural forgetful maps $ \widetilde{\mathcal G}\rightarrow\mathcal G\rightarrow\mathrm{Pic}^{w_0}(X/R)$. In order to describe the connection between these moduli spaces and linked Grassmannians, we need a auxilary construction.

Take an effective divisor $D=\sum_{v\in V(G)}D_v$ with relative degree $d'$ on $X$ such that $D_v$ is a union of sections of $X/R$ that pass through $Z_v$ and avoid the nodes of $X_0$.
Assume $D$ is ``\textit{sufficiently ample}" in the sense that $d_v:=\deg D_v$ is big enough relative to all $w$ in $V(\overline G(w_0))$ and the genus of $Z_v$. Then tensoring with $\mathscr O_X(D)$ kills all first cohomologies of degree-$d$ line bundles on $X_\eta$ and line bundles on $X_0$ with multidegrees in $V(\overline G(w_0))$. Let $D_\eta$ be the generic fiber of $D$. 

%Denote by $D_\eta$ and $D_0$ the generic and special fiber of $D$, respectively.

%Let $\mathcal G^2$ be the space over $\spec R$ of ``twisted (limit) linear series". More precisely, the generic fiber $\mathcal G^2_\eta$ is the space of tuples $(\mathcal L,V)$ where $V\subset H^0(X_\eta,\mathcal L(D_\eta))$ is a $r+1$ dimensional subspace; and the special fiber $\mathcal G^2_0$ parametrizes tuple of dimension $(r+1)$ subspaces $V_v\subset H^0(X_0,\mathscr L_{w_v}(D_0))$ such that the morphism in (\ref{eq:limit linear series}), with all line bundles twisted by $D_0$, has kernel with dimension at least $r+1$. Similarly, let $\widetilde{\mathcal G}^2$ be the space over $\spec R$ of ``twisted (linked) linear series". Namely, the generic fiber $\widetilde{\mathcal G}^2_\eta$ is the same as $\mathcal G^2_\eta$; and the special fiber $\widetilde{\mathcal G}^2_0$ parametrizes tuple of dimension $(r+1)$ subspaces $V_w\subset H^0(X_0,\mathscr L_w(D_0))$ which satisfies equation~(\ref{eq:linked linear series}). The ampleness of $D$ ensures that both $\mathcal G^2$ and $\widetilde {\mathcal G}^2$ can be obtained as degeneracy loci in a product of relative Grassmannians over $ \mathrm{Pic}^{w_0}(X/B)$, and $\mathcal G$ (resp. $\widetilde {\mathcal G}$) is then a determinatal locus in $\mathcal G^2$ (resp. $\widetilde{\mathcal G}^2$).

%We now explain how the constructions above is related to linked Grassmannians. Let $s\colon \spec R\rightarrow \mathrm{Pic}^{w_0}(X/R)$  be a section and Let $\mathcal L$ be a line bundle on $X_\eta$ corresponding to the generic point of $s$.
Let $\mathcal L$ be a line bundle on $X_\eta$ of degree $d$. For $w\in V(\overline G(w_0))$ let $\mathcal L'_w$ be the extension of $\mathcal L$ to $X$ with multidegree $w$ on $X_0$ and $\mathcal  L_w=\mathcal L'_w(D)$. Then $H^0(X, \mathcal L_w)$ is a lattice in 
$H^0(X_\eta,\mathcal L(D_\eta))$, and its reduction over $\kappa$ is identified with $H^0(X_0,\mathcal L_w|_{X_0})$. Moreover,
the configuration $\{[H^0(X, \mathcal L_w)]\}_{w\in V(\overline G(w_0))}$ is the convex hull of $\{[H^0(X, \mathcal L_{w_v})]\}_{v\in V(G)}$ for suitably choosen $(w_v)_v$, see \cite[\S~4.2]{he2023degenerations}. We denote this configuration by $\Gamma_\mathcal L$. Note that different degrees $w$ may induce the same lattice class $[H^0(X, \mathcal L_w)]$. Let $\mathscr L_{w_0}$ be the restriction of $\mathcal L'_{w_0}$ on $X_0$. Then the set of linked linear series with underlying line bundle $\mathscr L_{w_0}$ is identified with a determinantal locus in $LG_{r+1}(\Gamma_\mathcal L)_0$. 
  %The fiber of $\widetilde{\mathcal G}^2$ over the image of $s$ is naturally identified with  $LG_{r+1}(\Gamma_s)$. Moreover, 
We have the following theorem, which is part of \cite[Theorem 4.17]{he2023degenerations}.

\begin{thm}\label{thm:smoothing of lls}
Let $X/R$ be a smoothing family with special fiber $X_0$. %(in other words the chain structure $\bm n$ is trivial)
Let $w_0$ be a multidegree on $G$ of total degree $d$, and choose suitable concentrated multidegrees $(w_v)_v$. Suppose 
\begin{enumerate}
    \item every limit linear series of multidegree $w_0$ lifts to a linked linear series;
    \item the linked Grassmannians $LG_{r+1}(\Gamma_\mathcal L)$ are irreducible for all line bundles $\mathcal L$ above.
\end{enumerate}
If the moduli space $\mathcal G_0$ of limit $\mathfrak g^r_d$s of multidegree $w_0$ on $X_0$ has dimension $\rho=g-(r+1)(g-d+r)$ at a given point, then the corresponding limit linear series arises as the limit of a linear series on the geometric generic fiber of $X$.
\end{thm}

The conclusive part of Theorem~\ref{thm:smoothing of lls} is often referred to as the smoothing theorem for limit linear series. On the other hand, according to Lemma~\ref{lem:lifting limit linear series} and Theorem~\ref{thm:flatness}, we have the following corollary asserting the conditions (1) and (2) of Theorem~\ref{thm:smoothing of lls} in special cases. Compare to the last statement of \cite[Theorem 4.17]{he2023degenerations}.

\begin{cor}\label{cor:conditions}
If $\Gamma_\mathcal L$ is locally weakly independent for all line bundles $\mathcal L$ above, then both (1) and (2) in Theorem~\ref{thm:smoothing of lls} are satisfied.
\end{cor}

\subsection{The example}\label{subsec:smoothing}
In the situation of Theorem~\ref{thm:smoothing of lls}, let $G$ be the complete graph $K_n$ with $n$ vertices $v_0,v_1,...,v_{n-1}$, corresponding to the components $(Z_i)_{0\leq i\leq n-1}$ of $X_0$. Let $w_0$ be the multidegree which assigns $n-1-i$ to the vertex $v_i$. Let $w_j$ be the multidegree obtained from $w_0$ by twisting at $v_0,...,v_{j-1}$ for $1\leq j\le n-1$. Alternatively, $w_j$ can be described as follows: for $0\leq i\leq n-1$, if $i<j$, then $w_j$ has degree $j-i-1$ on $v_i$; otherwise it has degree $n-1+j-i$ on $v_i$.
See Figure~\ref{fig:lls} for the case $n=4$.
  \tikzset{every picture/.style={line width=0.75pt}} 
\begin{figure}[ht] 
\begin{tikzpicture}[x=0.45pt,y=0.45pt,yscale=-1,xscale=1]
\import{./}{figlls.tex}
\end{tikzpicture}
\caption{The complete graph $K_4$ and a set of concentrated multidegrees obtained from $w_0$ by twisting.}\label{fig:lls}
\end{figure}

One checks easily that $w_{v_i}:=w_i$ is concentrated on $v_i$ and we have $V(\overline G(w_0))=\{w_0,...,w_{n-1}\}$. %It follows that $\tilde \pi \colon \widetilde {\mathcal G}^2\rightarrow \mathcal G^2$ is actually an isomorphism, and condition (1) of Theorem~\ref{thm:smoothing of lls} is satisfied. On the other hand, 
For any $\mathcal L$ on $X_\eta$ of degree $\frac{n(n-1)}{2}$, by construction we have $\mathcal L'_{w_i}=\mathcal L'_{w_0}(\sum_{j<i}Z_i)$. Therefore, $\mathcal L_{w_i}=\mathcal L_{w_0}(\sum_{j<i}Z_i)$ and 
$$H^0(X,\mathcal L_{w_0})\subset H^0(X,\mathcal L_{w_1})\subset\cdots\subset H^0(X,\mathcal L_{w_{n-1}})\subset \pi^{-1}H^0(X,\mathcal L_{w_0}).$$
It follows that $\Gamma_\mathcal L$ is always a simplex, which is locally weakly independent. As a result, conditions (1) and (2) of Theorem~\ref{thm:smoothing of lls} are satisfied by Corollary~\ref{cor:conditions}. Thus, we have the smoothing theorem of limit linear series in this case.

\bibliographystyle{amsalpha}
\bibliography{myrefs}

\providecommand{\bysame}{\leavevmode\hbox to3em{\hrulefill}\thinspace}
\providecommand{\MR}{\relax\ifhmode\unskip\space\fi MR }
% \MRhref is called by the amsart/book/proc definition of \MR.
\providecommand{\MRhref}[2]{%
  \href{http://www.ams.org/mathscinet-getitem?mr=#1}{#2}
}
\providecommand{\href}[2]{#2}
\begin{thebibliography}{CHSW11}

\bibitem[AB08]{abramenko2008buildings}
Peter Abramenko and Kenneth Brown, \emph{{Buildings: Theory and Applications}},
  Springer, 2008.

\bibitem[BT84]{bruhat1984groupes}
Fran{\c{c}}ois Bruhat and Jacques Tits, \emph{Groupes r{\'e}ductifs sur un
  corps local: {II}. {S}ch{\'e}mas en groupes. {E}xistence d’une donn{\'e}e
  radicielle valu{\'e}e}, Publications Math{\'e}matiques de l'Institut des
  Hautes {\'E}tudes Scientifiques \textbf{60} (1984), no.~1, 5--184.

\bibitem[CHSW11]{cartwright2011mustafin}
Dustin Cartwright, Mathias H{\"a}bich, Bernd Sturmfels, and Annette Werner,
  \emph{Mustafin {V}arieties}, Selecta Mathematica \textbf{17} (2011), no.~4,
  757--793.

\bibitem[ESV21a]{esteves2021quiver1}
Eduardo Esteves, Renan Santos, and Eduardo Vital, \emph{{Quiver Representations
  Arising from Degenerations of Linear Series, {I}}}, arXiv e-prints (2021),
  arXiv--2112.

\bibitem[ESV21b]{esteves2021quiver2}
\bysame, \emph{{Quiver Representations Arising from Degenerations of Linear
  Series, {II}}}, arXiv e-prints (2021), arXiv--2112.

\bibitem[Fal01]{faltings2001toroidal}
Gerd Faltings, \emph{{Toroidal Resolutions for Some Matrix Singularities}},
  Moduli of Abelian Varieties, Springer, 2001, pp.~157--184.

\bibitem[FHLR22]{fakhruddin2022singularities}
Najmuddin Fakhruddin, Thomas Haines, Jo{\~a}o Louren{\c{c}}o, and Timo Richarz,
  \emph{Singularities of local models}, arXiv:2208.12072 (2022).

\bibitem[FLP22]{feigin2021totally}
Evgeny Feigin, Martina Lanini, and Alexander P{\"u}tz, \emph{{Totally
  Nonnegative Grassmannians, Grassmann Necklaces, and Quiver Grassmannians}},
  Canadian Journal of Mathematics (2022), 1--34.

\bibitem[FLP23]{feigin2023generalized}
\bysame, \emph{{Generalized Juggling Patterns, Quiver Grassmannians and Affine
  Flag Varieties}}, arXiv preprint arXiv:2302.00304 (2023).

\bibitem[G{\"o}r01]{gortz2001flatness}
Ulrich G{\"o}rtz, \emph{{On the Flatness of Models of Certain {S}himura
  Varieties of {PEL}-type}}, Mathematische Annalen \textbf{321} (2001), no.~3,
  689--727.

\bibitem[G{\"o}r03]{gortz2003flatness}
\bysame, \emph{{On the Flatness of Local Models for the Symplectic Group}},
  Advances in Mathematics \textbf{176} (2003), no.~1, 89--115.

\bibitem[GY10]{gortz2010supersingular}
Ulrich G{\"o}rtz and Chia-Fu Yu, \emph{Supersingular {K}ottwitz-{R}apoport
  {S}trata and {D}eligne-{L}usztig {V}arieties}, Journal of the Institute of
  Mathematics of Jussieu \textbf{9} (2010), no.~2, 357--390.

\bibitem[Hai]{haines4introduction}
Thomas Haines, \emph{Introduction to {S}himura varieties with bad reduction of
  parahoric type. {H}armonic analysis, the trace formula, and {S}himura
  varieties, 583--642}, Clay Math. Proc \textbf{4}.

\bibitem[He13]{he2013normality}
Xuhua He, \emph{{Normality and Cohen--Macaulayness of Local Models of Shimura
  Varieties}}, Duke Mathematical Journal \textbf{162} (2013), no.~13,
  2509--2523.

\bibitem[HR08]{haines2008parahoric}
Thomas Haines and Michael Rapoport, \emph{On parahoric subgroups}, Advances in
  Mathematics \textbf{219} (2008), no.~1, 188--198.

\bibitem[HZ]{hz3}
Xiang He and Naizhen Zhang, \emph{{Degenerations of Grassmannians via Lattice
  Configurations, III}}, In preparation.

\bibitem[HZ23]{he2023degenerations}
\bysame, \emph{{Degenerations of Grassmannians via Lattice Configurations}},
  International Mathematics Research Notices \textbf{2023} (2023), no.~1,
  298--349.

\bibitem[KR00]{kottwitz2000minuscule}
Robert Kottwitz and Michael Rapoport, \emph{{Minuscule Alcoves for {$\GL_n$}
  and {$GSp_{2n}$}}}, Manuscripta Mathematica \textbf{102} (2000), no.~4,
  403--428.

\bibitem[Liu02]{liu2002algebraic}
Qing Liu, \emph{{Algebraic Geometry and Arithmetic Curves}}, Oxford Graduate
  Texts in Mathematics, Oxford University Press, 2002.

\bibitem[Lus90]{lusztig1990canonical}
George Lusztig, \emph{{Canonical Bases Arising from Quantized Enveloping
  Algebras}}, Journal of the American Mathematical Society \textbf{3} (1990),
  no.~2, 447--498.

\bibitem[Mat87]{matsumura_1987}
Hideyuki Matsumura, \emph{{Commutative Ring Theory}}, Cambridge Studies in
  Advanced Mathematics, Cambridge University Press, 1987.

\bibitem[Mus78]{mustafin1978}
G.A. Mustafin, \emph{{Nonarchimedean Uniformization}}, Mathematics of the
  USSR-Sbornik \textbf{34} (1978), no.~2, 187.

\bibitem[Oss06]{Olls}
Brian Osserman, \emph{{A Limit Linear Series Moduli Scheme}}, Annales de
  l'Institut Fourier \textbf{56, no.4} (2006), 1165--1205.

\bibitem[Oss14]{Ohrk}
\bysame, \emph{{Limit Linear Series Moduli Stacks in Higher Rank}}, arxiv:
  1405.2937v1 (2014).

\bibitem[Oss19]{osserman2019limit}
\bysame, \emph{{Limit Linear Series for Curves not of Compact Type}}, Journal
  f{\"u}r die reine und angewandte Mathematik \textbf{2019} (2019), no.~753,
  57--88.

\bibitem[OTiB14]{osserman2014linked}
Brian Osserman and Montserrat Teixidor~i Bigas, \emph{Linked alternating forms
  and linked symplectic grassmannians}, International Mathematics Research
  Notices \textbf{2014} (2014), no.~3, 720--744.

\bibitem[OTiB16]{Osp}
\bysame, \emph{Linked {S}ymplectic {F}orms and {L}imit {L}inear {S}eries in
  {R}ank 2 with {S}pecial {D}eterminant}, Advances in Mathematics \textbf{288}
  (2016), 576--630.

\bibitem[PR03]{pappas2003local}
George Pappas and Michael Rapoport, \emph{{Local Models in the Ramified Case.
  I. The EL-case}}, Journal of Algebraic Geometry \textbf{12} (2003), no.~1,
  107--145.

\bibitem[PR05]{PR2}
\bysame, \emph{{Local Models in the Ramified Case, II: Splitting models}}, Duke
  Mathematical Journal \textbf{127} (2005), no.~2, 193 -- 250.

\bibitem[PR08]{pappas2008twisted}
\bysame, \emph{{Twisted Loop Groups and Their Affine Flag Varieties}}, Advances
  in Mathematics \textbf{219} (2008), no.~1, 118--198.

\bibitem[PZ13]{pappas2013local}
George Pappas and Xinwen Zhu, \emph{Local {M}odels of {S}himura {V}arieties and
  a {C}onjecture of {K}ottwitz}, Inventiones mathematicae \textbf{194} (2013),
  147--254.

\bibitem[Ric13]{richarz2013schubert}
Timo Richarz, \emph{{Schubert Varieties in Twisted Affine Flag Varieties and
  Local Models}}, Journal of Algebra \textbf{375} (2013), 121--147.

\bibitem[RZ16]{rapoport2016period}
Michael Rapoport and Thomas Zink, \emph{Period spaces for p-divisible groups
  (am-141)}, vol. 141, Princeton University Press, 2016.

\bibitem[Shi86]{Shi}
Jian-Yi Shi, \emph{The {K}azhdan-{L}usztig {C}ells in {C}ertain {A}ffine {W}eyl
  {G}roups}, Springer, 1986.

\bibitem[{Sta}20]{stacks-project}
The {Stacks Project Authors}, \emph{{The Stacks Project}},
  \url{https://stacks.math.columbia.edu}, 2020.

\bibitem[Tit79]{tits1979reductive}
Jacques Tits, \emph{{Reductive Groups over Local Fields}}, Automorphic Forms,
  Representations and L-functions (Proc. Sympos. Pure Math., Oregon State
  Univ., Corvallis, Ore., 1977), vol.~1, 1979, pp.~29--69.

\end{thebibliography}
\end{document}